\documentclass[12pt,reqno]{amsart}

\newcommand{\defeq}{\stackrel{\rm{def}}{=}}

\usepackage{amssymb}
\usepackage{amsthm}
\usepackage{amsxtra}

\usepackage{fancyhdr}
\usepackage{graphicx}
\usepackage{color}
\usepackage{amsmath}
\usepackage{listings} 
\usepackage{float}
\usepackage{cases} 
\usepackage{threeparttable}
\usepackage{dcolumn}
\usepackage{multirow}
\usepackage{booktabs}

\DeclareMathOperator{\sign}{sign}

\newtheorem{theorem}{Theorem}[section]

\newtheorem{proposition}{Proposition}[section]
\newtheorem{lemma}[proposition]{Lemma}

\newtheorem{conjecture}[theorem]{Conjecture}

\theoremstyle{remark}
\newtheorem{remark}[proposition]{Remark}

\numberwithin{equation}{section}

\title[2D HBO]{Higher dimensional generalization of\\ the Benjamin-Ono equation: 2D 
case} 

\author[O. Ria\~no]{Oscar Ria\~no}
\address{Department of Mathematics  \& Statistics\\Florida International University,  Miami, FL, USA}
\curraddr{}
\email{orianoca@fiu.edu}

\author[S. Roudenko]{Svetlana Roudenko}
\address{Department of Mathematics \& Statistics\\Florida International University,  Miami, FL, USA}
\curraddr{}
\email{sroudenko@fiu.edu}

\author[K. Yang]{Kai Yang}
\address{Department of Mathematics \& Statistics\\Florida International University,  Miami, FL, USA}
\curraddr{}
\email{yangk@fiu.edu}



\begin{document}

\begin{abstract}
We consider a higher-dimensional version 
of the Benjamin-Ono (HBO) equation in the 2D setting: $u_t- \mathcal{R}_1 \Delta u +  \frac{1}{2}(u^2)_x=0, (x,y) \in \mathbb{R}^2$, which is $L^2$-critical, and investigate properties of solutions both analytically and numerically. For a generalized equation (fractional 2D gKdV) after  deriving the Pohozaev identities, we obtain non-existence conditions for solitary wave solutions, then prove uniform bounds in the energy space or conditional global existence,
and investigate the radiation region, a specific wedge in the negative $x$-direction. We then introduce our numerical approach in a general context, and apply it to obtain the ground state solution in the 2D critical HBO equation, then show that its mass is a threshold for global vs. finite time existing solutions, which is typical in the focusing (mass-)critical dispersive equations. We also observe that globally existing solutions tend to disperse completely into the radiation in this nonlocal equation. 
The blow-up solutions travel in the positive $x$-direction with the rescaled ground state profile while also radiating dispersive oscillations into the radiative wedge. We conclude with examples of different interactions of two solitary wave solutions, including weak and strong interactions.      

\end{abstract}

\maketitle

\tableofcontents

\section{Introduction}\label{S:Introduction}
We study the following higher dimensional version of the Benjamin-Ono  (HBO) equation 
\begin{align}\label{E:HBO}
u_t- \mathcal{R}_1 \Delta u+  \frac{1}{2}(u^2)_x=0, \qquad (x,y)\in \mathbb{R}^2, \, \, t\in \mathbb{R},
\end{align}
where the operator $\mathcal{R}_1$ denotes the Riesz transform operator with respect to the first variable defined by the singular integral 
\begin{equation*}
\mathcal{R}_1 f
(x,y)=\frac{1}{2\pi} \, \textit{p.v.} \int \frac{(x-z_1) \, 
f(z_1,z_2)}{\big((x-z_1)^2+(y-z_2)^2 \big)^{3/2}}\, dz_1 \, dz_2,
\end{equation*}
and $\widehat{\mathcal{R}_1 f}(\xi_1,\xi_2)=\frac{-i\xi_1}{|(\xi_1,\xi_2)|}\widehat{f}(\xi_1,\xi_2)$ with $|(\xi_1,\xi_2)|=\sqrt{\xi_1^2+\xi_2^2}$.
\smallskip

One of the first mentioning of this equations was by Shrira in \cite{S1989}, where he was describing the $2d$ long-wave perturbations in a boundary-layer type shear flow. These perturbations were weakly nonlinear, the flow did not have any inflection points, and the perturbations would be valid for the boundary layers along an inviscid boundary for free surface flows. That model reduced to an equation for the amplitude $u$ of the longitudinal velocity of the fluid, which is exactly the equation \eqref{E:HBO}. There are various extensions or reductions of the equation \eqref{E:HBO} that have been studied since then, for some initial studies, see \cite{PS}, \cite{A}, 
\cite{DK1994}, \cite{PS1994}, 
and for recent investigations, refer to \cite{Riano2021}, \cite{EsfaPastor2018}, \cite{Schippa2020} and references therein. 

In the one-dimensional case, the multiplier associated to the Riesz transform coincides with that of the Hilbert transform operator. From this standpoint, \eqref{E:HBO} can be regarded as a two-dimensional extension of the Benjamin-Ono equation (BO)
\begin{equation}\label{E:BO}
u_t- \mathcal{H}\partial_x^2 u+ \frac{1}{2}(u^2)_x=0, \qquad (x,t)\in \mathbb{R}^2,
\end{equation}
where $\widehat{\mathcal{H}f}(\xi)=-i\sign(\xi)\widehat{f}(\xi)$. We remark that the equation \eqref{E:BO}, including other nonlinearities, is of interest in various water wave models such as waves in deep water, e.g., see \cite{AbBOnaFellSaut,benjamin_1967,BONAHEN,BK04,Ono_1975,SvetKaiWangBO} and reviews \cite{RGustav,miller2019nonlinear}. 

On the other hand, the equation \eqref{E:HBO} can also be seen as a particular case of the higher dimensional fractional generalized KdV equation
\begin{equation}\label{E:fZK}
u_t- \partial_x (-\Delta)^{s} u+ \frac{1}{m}(u^m)_x=0, \quad (x,\dots)\in \mathbb{R}^d, \, \, t\in\mathbb{R}, \, \, \,  m >1,
\end{equation} 
where $d\geq 2$, $0<s<1$, $m$ is integer, and $(-\Delta)^{s}$ denotes the fractional Laplacian of order $s$ defined by the Fourier multiplier with symbol $|\xi|^{2s}=\big(\xi_1^2+\dots+\xi_d^2\big)^s$. This generalization is more evident by recalling that 
$\mathcal{R}_1(u) = \mathcal{F}^{-1}\left( \frac{i\xi_1}{|\xi|}\hat{u} \right)(x)=\partial_x (-\Delta)^{-\frac{1}{2}}u$, which yields 
$$
\mathcal{R}_1(-\Delta u)= \mathcal{F}^{-1}\Big( \frac{i\xi_1}{|\xi|} \cdot |\xi|^2 \hat{u} \Big)(x)=\partial_x (-\Delta)^{\frac{1}{2}}u,
$$
and \eqref{E:fZK} generalizes the HBO equation \eqref{E:HBO} with a fractional dispersion operator of order $s$ and nonlinearity $m$. Setting $s=\frac{1}{2}$ and $m=2$ in \eqref{E:fZK} yields \eqref{E:HBO}, whereas $s=1$ and $m=2$ in \eqref{E:fZK} agrees with the Zakharov-Kuznetsov equation (ZK), which in $3d$ describes the propagation of ionic-acoustic waves in magnetized plasma \cite{ZK} and in $2d$, for example, it serves as the amplitude equation for long waves on the free surface of a thin film in a specific fluid and viscosity parameters \cite{MM1989}.
The family of equations \eqref{E:fZK} is useful to measure the competition between the effects of dispersion and nonlinearity in a $d$-dimensional model. 

In general, the power $m-1>0$ in \eqref{E:fZK} does not need to be an integer number. One can take, for instance, $m-1=k/p$, where $k$ and $p$ are relatively prime and $p$ is odd. Consequently, it is  possible to set a branch of the map $ \omega \mapsto \omega^{1/p}$ real on the real axis. A similar condition has been used before in \cite{esfahani2015}. Alternatively, the nonlinearity in \eqref{E:fZK} can be replaced with $\partial_x(|u|^{m-1} u)$, and thus, one can consider the equation 
\begin{equation}\label{E:AfZK}
u_t- \partial_x (-\Delta)^{s} u+ \frac{1}{m}\partial_x(|u|^{m-1}u)=0, \qquad (x,\dots)\in \mathbb{R}^d, \, \, t\in\mathbb{R}, \, \, \,  m >1.
\end{equation} 
In what follows, when $m$ is not an integer, we will consider \eqref{E:AfZK}.

Real solutions of \eqref{E:HBO} formally satisfy at least three conservation laws: the $L^2$-norm (or mass) conservation
\begin{align}\label{E:mass}
M[u(t)] \defeq \int_{\mathbb{R}^2} |u(x,y,t)|^2 \, dx dy =M[u(0)],
\end{align}
the energy (or Hamiltonian) conservation 
\begin{align}\label{E:energy}
E[u(t)] \defeq \frac{1}{2} \int_{\mathbb{R}^2} |(-\Delta)^{\frac{1}{4}}u(x,y,t)|^2 dx dy -\frac{1}{6} \int_{\mathbb{R}^2} \big(u(x,y,t)\big)^{3} \,  dx dy = E[u(0)],
\end{align}
and the 
$L^1$-type conservation 
\begin{equation}\label{E:l1inte}
\int_{\mathbb{R}} u(x,y,t)\, dx =\int_{\mathbb{R}} u(x,y,0)\, dx,
\end{equation}
which can also be stated in a $2d$ form 
\begin{equation}\label{E:l1inte-g}
\int_{\mathbb{R}^2} u(x,y,t)\, dx dy=\int_{\mathbb{R}^2} u(x,y,0)\, dx dy.
\end{equation}
We mention that no other conserved quantities are known for \eqref{E:HBO}. In contrast, the BO equation \eqref{E:BO} is a completely integrable Hamiltonian system, see \cite{BK_79,NakaAk}.

The equation \eqref{E:HBO} is invariant under the scaling: if $u$ solves \eqref{E:HBO}, then so does 
\begin{equation}\label{scaling1}
u_{\lambda}(x,y,t)=\lambda u(\lambda x,\lambda y,\lambda^2 t)
\end{equation}
for any positive $\lambda$. Consequently, the homogeneous Sobolev space $\dot{H}^{r_{c}}$ is invariant under the scaling \eqref{scaling1} when $r_c=0$, in other words, the equation \eqref{E:HBO} is $L^2$-critical. (For a general case of \eqref{E:fZK} and \eqref{E:AfZK}, see Section \ref{S:general}.)


We next recall some results regarding the well-posedness for the Cauchy problem associated to \eqref{E:HBO} and \eqref{E:fZK} in Sobolev spaces. In \cite{linaO}, it was proved that \eqref{E:HBO} is locally well-posed in $H^r(\mathbb{R}^2)$ whenever $r>5/3$. In \cite{RobertS}, the local well-posedness theory was extended for regularities $r>3/2$. Furthermore, the well-posedness results in weighted Sobolev spaces as well as some unique continuation principles for the equation \eqref{E:HBO} were studied in \cite{OscarWHBO}. We remark that the above local well-posedness results were obtained via compactness methods as one cannot solve the initial value problem associated to \eqref{E:HBO} by a Picard iterative method implemented on its integral formulation for any initial data in the Sobolev space $H^r(\mathbb{R}^2)$, $r\in \mathbb{R}$ (see \cite[Theorem 4.1 and Corollary A.1]{linaO}). As far as the equation \eqref{E:fZK}, the following well-posedness results hold. When $d\geq 2$, $\frac{1}{2}\leq s <1$ and $m=2$ in \eqref{E:fZK}, it was proved in \cite{RobertS} that \eqref{E:fZK} is locally well-posed in $H^r(\mathbb{R}^d)$ provided that $r>\frac{d+3}{2}-2s$. This same result was proved before for $d\geq 3$ and $s=\frac{1}{2}$ in \cite{linaO}. On the other hand,  by the standard parabolic regularization argument (see \cite{AbBOnaFellSaut,Iorio}), the Cauchy problem associated to \eqref{E:fZK} is locally well-posed in $H^r(\mathbb{R}^d)$ whenever $r>\frac{d}{2}+1$ for any $m>1$ integer, and $0<s<1$ fixed. To the best of our knowledge there are no results concerning the global well-posedness (GWP) for the Cauchy problem associated to \eqref{E:fZK} or \eqref{E:AfZK} with $d\geq 2$ in the current literature. Regarding the GWP for \eqref{E:fZK} in $d=1$, see \cite{Herr_Ione_Keni_Koch,MOVPIL} and references therein.

The solitary-wave solutions for the equation \eqref{E:HBO} are of the form
$$
u(x,y,t)=Q_c(x-ct,y),
$$ 
where $c>0$ denotes the speed of propagation or the scaling factor for $Q_c(x,y)=c \, Q(cx,cy)$, and $Q$ is a real-valued vanishing at infinity solution of
\begin{align}\label{E:GS}
Q+(-\Delta)^{\frac{1}{2}}Q-\tfrac{1}{2}Q^{2}=0.
\end{align}
(For a general case of \eqref{E:fZK}, see Sections \ref{S:general} and \ref{S:Soliton}.)
The existence and spatial decay of solutions for \eqref{E:GS} were considered in \cite{M}. The uniqueness of positive solutions can be deduced as a particular case of the results obtained in \cite{Frank2016} (see also \cite{Frank2013} for 1d) for a class of nonlocal equations
\begin{equation}\label{NLEQ}
   \Psi+(-\Delta)^{s}\Psi-|\Psi|^{r}\Psi=0, \hspace{0.2cm} \text{in } \mathbb{R}^d,
\end{equation}
with $d\geq 1$, $s\in (0,1)$ and $0<r<r_{\ast}=r_{\ast}(d,s)$, where
\begin{equation*}
   r_{\ast}= \begin{cases} \frac{4s}{d-2s} &\mbox{ for } 0<s<\frac{d}{2}, \\
+\infty & \mbox{ for } s\geq \frac{d}{2}. 
\end{cases}
\end{equation*}

In the general $2d$ setting we present some further results regarding solitary waves in Section \ref{S:general}. In particular, from Pohozaev identities the nonexistence of solitary wave solutions for the fractional 2d generalized KdV equation is obtained. 

Furthermore, in the same general fractional setting we review existence of solutions in sufficiently regular space and then obtain uniform bounds or global existence criteria in the $L^2$-subcritical, critical and supercritical cases, see Theorem \ref{propcriti}, though this result is conditional on the local well-posedness (lwp) in the energy space $H^s(\mathbb R^2)$, $0<s<1$, since we only have the local well-posedness in $H^r(\mathbb R^2)$, $r>2$. 
We note that in the $L^2$-critical case, the threshold for global existence is given by the mass of the ground state, and we investigate this threshold more closely in the later part of the paper via numerical simulations. In particular, we show that all sufficiently localized data above the threshold blow up in finite time, confirming the Conjecture \ref{C:critical}, see subsection \ref{S:bounds} (we tried initial data with exponential and polynomial decays, with the polynomial decay as low as $r^{-2}$, which is below the ground state decay). We also studied the global existence and observed that such solutions tend to disperse completely into the radiation in the nonlocal $L^2$-critical $2d$ HBO equation, see subsection \ref{S:global}. The radiation region is formed as the wedge around the negative $x$-direction with the opening angle as big as $\tan \theta = 2 \sqrt {2}$ in this equation, and we show that in general this radiation region only depends on the dispersion operator in the linear equation, see subsection \ref{S:radiation}. 

After obtaining some results about a single maximum initial data, we turn to the interaction of the solitary waves, and show that various interactions are possible, which depends on the initial geometrical configuration and the distribution of mass in both solitary waves. In particular, there can be no significant (or only weak) interaction, and we also observe strong interactions, where both solitary waves can merge into one and either blow up in finite time, or disperse (eventually both of them), see subsection \ref{S:interaction}.

\smallskip

The paper is organized as follows: in Section \ref{S:general} we study a generalized fractional KdV in $2d$ setting and review the conserved quantities, scaling invariance, derive Pohozaev identities, which leads to the results about the non-existence of solitary wave solutions in various contexts. After that in subsection \ref{S:bounds} we consider $L^2$-subcritical, critical and supercritical cases in the general fractional setting and obtain uniform bounds in the energy space in Theorem \ref{propcriti}, parts (C1), (C2), and (C3), respectively; in particular, noticing that in the $L^2$-critical case of the HBO equation the mass of the ground state solution plays the role of the threshold for the global existence. In subsection \ref{S:radiation} we consider a linear 2d fractional KdV equation   and show the radiation region, which is a wedge with a specific angle, depending only on the dispersion operator and not on the nonlinearity or dimension. Next, in Section \ref{S:Numerical method} we describe our numerical approach, including space discretization via the rational basis (eigen)functions, discretization of the fractional Laplacian, in particular, using the Dunford-Taylor formula to change the fractional Laplacian into the full Laplacian in Galerkin formulation, and the extension to the higher dimensional computations. In Section \ref{S:Soliton} we use Petviashvili's iteration method to obtain the ground state profile (or its rescaled versions). Finally, in Section \ref{S:Numerical solution} we show the numerical results confirming the ground state mass threshold for global existence vs. finite time blow-up, and study the behavior of globally existing solutions more carefully, finding that even if a solitary wave-type solution starts traveling to the right, it eventually stops moving and subdues into the radiation via dispersive oscillations in the negative $x$-directions. Blow-up solutions, on the other hand, travel in the positive $x$-direction and blow up with the rescaled ground state profiles. Lastly, in subsection \ref{S:interaction} we examine interaction of two solitary waves in different geometrical settings and of different sizes and show weak and strong interactions. 
\smallskip
  
{\bf Acknowledgments.}
All three authors were 
partially supported by the NSF grant DMS-1927258 (PI: S. Roudenko).


\section{Remarks on the fractional $2d$ generalized KdV equation}
\label{S:general}

In this section we take a more general approach and consider equations \eqref{E:fZK} and \eqref{E:AfZK} in the $2d$ setting. We start with recalling some useful invariances and inequalities such as the conserved quantities and scaling invariance, as well as the Gagiardo-Nirenberg inequality and discuss  
some results regarding the existence or non-existence of solutions to the stationary problems of the form \eqref{NLEQ}.  As an application, we show uniform estimates in time for solutions of \eqref{E:fZK} in the energy space $H^s(\mathbb{R}^2)$. Additionally, we present a formal analysis of the dispersive relation of \eqref{E:fZK}, which allows us to conjecture and study regions in space where dispersive oscillations, or radiation, occur for solutions of \eqref{E:fZK}. In particular, setting $s=\frac{1}{2}$ and $m=2$ the results of this section are valid for \eqref{E:HBO}.

\subsection{Preliminaries on the fractional 2d gKdV and ground state solutions}\label{S:basics}

We focus our discussion on the following generalization of \eqref{E:fZK} 
\begin{equation}\label{SHBO-IVP}
   u_t+\nu_1 \partial_{x}(-\Delta)^{s} u +\frac{\nu_2}{m}(u^m)_x=0, \qquad (x,y)\in \mathbb{R}^2, \, \,t\in \mathbb{R}, \, \, m>1
\end{equation}
where $\nu_1\neq 0$, $\nu_2\in \{1,-1\}$ and $0<s<1$. We use two parameters $\nu_1 $ and $ \nu_2 $ to indicate subtle differences in the existence of the ground state solutions and other properties, see Remark \ref{remarothereq} below. 

During their lifespans, solutions of the equation \eqref{SHBO-IVP} satisfy the mass conservation \eqref{E:mass}, the $L^1$-type invariance \eqref{E:l1inte}, and the energy conservation, which in this case is given as
\begin{align}\label{genergy}
E_s[u(t)] = \frac{1}{2} \int_{\mathbb{R}^2} |(-\Delta)^{\frac{s}{2}}u(x,y,t)|^2 \, dxdy +\frac{\nu_2}{\nu_1 m(m+1)} \int_{\mathbb{R}^2} \big(u(x,y,t)\big)^{m+1} \,  dx dy = E_s[u(0)].
\end{align}
The equation \eqref{SHBO-IVP} is invariant under the scaling 
$$u_{\lambda}(x,y,t)=\lambda^{\frac{2s}{m-1}} u(\lambda x,\lambda y,\lambda^{1+2s} t)$$
for any positive $\lambda$. Thus, \eqref{SHBO-IVP} is invariant in the Sobolev space $\dot{H}^{r_{c}}(\mathbb{R}^2)$ with 
\begin{equation}\label{CritiInd}  
r_{c}=1-\frac{2s}{m-1}.
\end{equation}
The critical index $r_{c}$ is convenient for classifying the equation \eqref{SHBO-IVP} according to the values of $m$ and $s$ ($m>1$, $s>0$).  When $1<m<2s+1$ ($r_{c}<0$), the equation \eqref{SHBO-IVP} is referred to as the $L^2$-subcritical equation; if $m=2s+1$ ($r_{c}=0$), the equation is $L^2$-critical; when $m>2s+1$ ($r_{c}>0$), the equation \eqref{SHBO-IVP} is $L^2$-supercritical. We also note that the equation is energy-critical if $m=\frac{1+s}{1-s}$ (or $r_c=s$).

We are interested in studying localized solitary-wave solutions for the equation \eqref{SHBO-IVP} of the form $u(x,y,t)=\varphi(x-c\,t,y)$, where $c\in \mathbb{R}$. Substituting $\varphi(x-c\,t,y)$ into \eqref{SHBO-IVP}, integrating once with respect to the variable $z=x-ct$, and assuming that $\varphi$ vanishes at infinity, we deduce that $\varphi$ satisfies
\begin{equation}\label{SWEQ}
 -c\varphi+\nu_1 (-\Delta)^{s}\varphi+\frac{\nu_2}{m}\varphi^{m}=0.
\end{equation}
Setting $c>0$ and $\nu_1<0$, $\nu_2=1$, the existence and uniqueness of solutions for \eqref{SWEQ} are the consequences of the results in \cite{Frank2016} for the class of the equations of type \eqref{NLEQ}. We briefly recall these results. To establish the existence of solutions one can use the Weinstein classical approach, which consists of determining the best constant $C_{GN}$ in the Gagliardo--Nirenberg inequality
\begin{equation}\label{optconsg}
   \|f\|_{L^{m+1}(\mathbb{R}^2)}^{m+1} \leq C_{GN} \|(-\Delta)^{\frac{s}{2}}f\|_{L^2(\mathbb{R}^2)}^{\frac{m-1}{s}}\|f\|_{L^2(\mathbb{R}^2)}^{(m+1)-\frac{(m-1)}{s}},
\end{equation}
where the sharp constant $C_{GN}$ is obtained by minimizing the functional
\begin{equation}\label{optconsg1}
   J(f)=\frac{\|(-\Delta)^{\frac{s}{2}}f\|_{L^2(\mathbb{R}^2)}^{\frac{m-1}{s}}\|f\|_{L^2(\mathbb{R}^2)}^{(m+1)-\frac{(m-1)}{s}}}{\|f\|_{L^{m+1}(\mathbb{R}^2)}^{m+1}},
\end{equation}
defined for $f\in H^{s}(\mathbb{R}^2)$ with $f\neq 0$. Thus, one can use concentration-compactness arguments to show that $C_{GN}^{-1}=\inf_{f\neq 0} J(f)$ is attained. Moreover, recalling that $c>0$, $\nu_1<0$ and $\nu_2=1$,  by computing $J'(\cdot)$, it follows that any minimizer $\varphi\in H^{s}(\mathbb{R}^2)$ satisfies the equation \eqref{SWEQ} after a suitable rescaling, and the inequality $J(|\varphi|) \leq J(\varphi)$ implies that the minimizer $\varphi$ can be chosen to be nonnegative (for further properties see \cite[Appendix D]{Frank2016} and the reference therein), concluding the existence part. The uniqueness of the ground state (any nonnegative minimizer $\varphi$ of $J(\cdot)$ is a ground state) was established up to translation (or being radially symmetric and decreasing around some point) in \cite{Frank2016} (for $1d$ case see \cite{Frank2013}). Summarizing we have the following result:  

\begin{theorem}[\cite{Frank2016}]\label{existTHR}
Let $0<s<1$, $c>0$, $\nu_1<0$, $\nu_2=1$ and $1<m<\frac{1+s}{1-s}$. Then the equation \eqref{SWEQ} admits a unique, up to translation, positive solution $\varphi$ in $H^{s}(\mathbb{R}^2)$. Moreover, there exists some $(x_0,y_0)\in \mathbb{R}^2$ such that $\varphi(\cdot-x_0,\cdot-y_0)$ is radial, positive, and strictly decreasing in $|(x-x_0,y-y_0)|$. Additionally, the function $\varphi$ belongs to $H^{2s+1}(\mathbb{R}^2)\cap C^{\infty}(\mathbb{R}^2)$ and it satisfies
 \begin{equation}\label{poldecayGS}
    \frac{C_1}{1+|(x,y)|^{2+2s}} \leq \varphi(x,y) \leq \frac{C_2}{1+|(x,y)|^{2+2s}},
 \end{equation}
for all $(x,y) \in \mathbb{R}^2$, with some constants $C_2 \geq C_1>0$ depending on $m$ and $\varphi$. 
\end{theorem}
\begin{remark}
For the case when $m=\frac{1+s}{1-s}$, see a corresponding result in \cite{CCO2006},\cite{Li2004}. 
\end{remark}
By rescaling and setting $s=\frac{1}{2}$, $m=2$, Theorem \ref{existTHR} establishes the existence of the unique positive solution $Q$ for the equation \eqref{E:GS}, which we consider later in Sections \ref{S:Soliton} and \ref{S:Numerical solution}.
\smallskip

We next derive the key Pohozaev identities (we use them later in Section \ref{S:Soliton} for the verification of the ground state computations).
 
\begin{lemma}\label{PHident}
Assume $0<s<1$, $c\neq 0$, $m>1$ with $m\neq \frac{1+s}{1-s}$. Let $\varphi$ be a smooth vanishing at infinity solution of \eqref{SWEQ}. Then the following identities hold true
\begin{align}
&\|(-\Delta)^{\frac{s}{2}}\varphi \|_{L^2(\mathbb{R}^2)}^2 =-\frac{c(m-1)}{\nu_1(2-(1-s)(m+1))}\| \varphi \|_{L^2(\mathbb{R}^2)}^2, \label{Phideneq1} \\
& \int_{\mathbb{R}^2} \varphi^{m+1} \, dx dy=\frac{scm(m+1)}{\nu_2(2-(1-s)(m+1))}\| \varphi \|_{L^2(\mathbb{R}^2)}^2. \label{Phideneq2}
\end{align}
As a consequence, when $\nu_1=-1$ and $\nu_2=1$,
\begin{equation}\label{E:Energy-phi}
E[\varphi] = r_c \, \frac{c(m-1)}{(2-(1-s)(m+1))} \| \varphi \|_{L^2(\mathbb R^2)}^2.
\end{equation}
\end{lemma} 
By ``smooth", we mean that the functions have sufficient regularity to justify the arguments below in the proof of Lemma \ref{PHident}. By setting $\nu_1<0$, $\nu_2=1$, $c>0$ and $1<m<\frac{1+s}{1-s}$, the conclusion of Theorem \ref{existTHR} determines the existence of a solution for the equation \eqref{SWEQ}, which satisfies the required assumptions of regularity and decay to verify Lemma \ref{PHident}. 

\begin{remark}
For the specific case of the equation \eqref{SWEQ} with  $s=\frac{1}{2}$, $m=2$, $c=1$, $\nu_1=-1$, and $\nu_2=1$, i.e., the equation  \eqref{E:GS}, the identities  \eqref{Phideneq1}, \eqref{Phideneq2} and \eqref{E:Energy-phi} reduce to 
\begin{align}
\|(-\Delta)^{\frac{1}{4}}\varphi \|_{L^2(\mathbb{R}^2)}^2 & =2\| \varphi \|_{L^2(\mathbb{R}^2)}^2, \label{Phideneq1.1} \\
 \int_{\mathbb{R}^2} \varphi^{3} \, dx dy & =6\| \varphi \|_{L^2(\mathbb{R}^2)}^2, \label{Phideneq2.1} \\
\qquad E[\varphi] & = 0, \label{PhiE}
\end{align}
which we use to verify the accuracy of our computation for the ground state  (see $e_1,e_2, e_3$) in Section \ref{S:Soliton} below. 
\end{remark}

\begin{proof}[Proof of Lemma \ref{PHident}]

Multiplying \eqref{SHBO-IVP} by $\varphi$ and integrating on $\mathbb{R}^2$, yields
\begin{equation}\label{SW0}
   \int \Big(-c\varphi^2+\nu_1 (-\Delta)^{s}\varphi \,  \varphi +\frac{\nu_2}{m}\varphi^{m+1} \Big) \, dx dy=0.
\end{equation}
On the other hand, we claim 
\begin{align}
\int (-\Delta)^{s}\varphi \big(x \varphi_x\big) \, dx dy=-s\int (-\Delta)^{s-1}\partial_{x}^2\varphi \varphi \, dx dy -\frac{1}{2}\int(-\Delta)^{s}\varphi \varphi \,dx dy, \label{SW0.1} \\ 
\int (-\Delta)^{s}\varphi \big(y \varphi_y\big) \, dx dy=- s\int (-\Delta)^{s-1}\partial_{y}^2\varphi \varphi \, dx dy -\frac{1}{2}\int (-\Delta)^{s}\varphi \varphi \,dxdy.  \label{SW0.2}
\end{align}   
We only show \eqref{SW0.1} as the same reasoning leads to \eqref{SW0.2}. We write the left-hand side of \eqref{SW0.1} as follows 
\begin{equation*}
\begin{aligned}
\int (-\Delta)^{s}\varphi \big(x \varphi_x\big) \, dx dy&=\int (-\Delta)^{\frac{s}{2}}\varphi (-\Delta)^{\frac{s}{2}}(x \varphi_x) \, dx dy\\
&=\int (-\Delta)^{\frac{s}{2}}\varphi [(-\Delta)^{\frac{s}{2}};x]\varphi_x \, dx dy+\int (-\Delta)^{\frac{s}{2}}\varphi \big( x (-\Delta)^{\frac{s}{2}}\varphi_x \big) \, dx dy\\
&=\int (-\Delta)^{\frac{s}{2}}\varphi [(-\Delta)^{\frac{s}{2}};x]\varphi_x \, dx dy-\frac{1}{2}\int (-\Delta)^{s}\varphi \varphi  \, dx dy,
\end{aligned}
\end{equation*}
where we use the notation $[A;B]=AB-BA$ for given operators $A$, $B$, and the last term of the above identity is obtained after integration by parts and using that $(-\Delta)^{\frac{s}{2}}$ determines a symmetric operator. The proof of \eqref{SW0.1} is now a consequence of the identity
\begin{equation}
\big[(-\Delta)^{\frac{s}{2}};x\big]\partial_x f=-s(-\Delta)^{\frac{s-2}{2}}\partial_x^2f,
\end{equation}
which follows by computing the Fourier transform of the commutator. Note that $$\|(-\Delta)^{\frac{s-2}{2}}\partial_x^2 f\|_{L^2(\mathbb{R}^2)}\leq \|(-\Delta)^{\frac{s}{2}}f\|_{L^2(\mathbb{R}^2)}.$$
This establishes \eqref{SW0.1} and \eqref{SW0.2}.
   
Next, multiplying \eqref{SWEQ} by $x \varphi_{x}$ and using \eqref{SW0.1}, we get
\begin{equation}\label{SW0.3}
\begin{aligned}
\int \Big( c\varphi^2-2s\nu_1 (-\Delta)^{s-1}\partial_x^2 \varphi \varphi-\nu_1 (-\Delta)^{s} \varphi \varphi-\frac{2\nu_2}{m(m+1)}\varphi^{m+1} \Big)\, dx dy.
\end{aligned}
\end{equation}
Likewise, \eqref{SW0.2} yields
\begin{equation}\label{SW0.4}
\begin{aligned}
\int \Big( c\varphi^2-2s\nu_1 (-\Delta)^{s-1}\partial_y^2 \varphi \varphi-\nu_1 (-\Delta)^{s} \varphi \varphi-\frac{2\nu_2}{m(m+1)}\varphi^{m+1} \Big)\, dx dy.
\end{aligned}
\end{equation}
By adding \eqref{SW0.3} and \eqref{SW0.4}, and using that $\partial_x^2+\partial_y^2=\Delta$, we deduce
\begin{equation}\label{SW0.5}
\begin{aligned}
\int \Big( c\varphi^2+\nu_1 (s-1)(-\Delta)^{s} \varphi \varphi-\frac{2\nu_2}{m(m+1)}\varphi^{m+1}\Big) \, dx dy.
\end{aligned}
\end{equation}
Consequently, solving \eqref{SW0.5} and \eqref{SW0} yields the desired expressions in  \eqref{Phideneq1} and \eqref{Phideneq2}. 
\end{proof}
\begin{remark}
We notice that setting $c=0$ in \eqref{SWEQ} and inspecting the above proof of Lemma \ref{nonexsw},  the equation \eqref{SW0.3} must be equal to \eqref{SW0.4}. This imposes the condition $m=\frac{1+s}{1-s}$, and hence, for this case, we get the identity
\begin{equation}\label{SW0.6} 
\begin{aligned}
-\nu_1\|(-\Delta)^{\frac{s}{2}}\varphi\|_{L^2(\mathbb{R}^2)}^2=\nu_2\bigg(\frac{1-s}{1+s}\bigg)\int_{\mathbb{R}^2} \varphi^{\frac{2}{1-s}} \, dx dy.
\end{aligned}
\end{equation}
Note that $m=\frac{1+s}{1-s}$ is the energy-critical case for solutions of \eqref{SHBO-IVP}, i.e., $r_{c}=s$ in \eqref{CritiInd}). The function $\varphi$ that gives the equality in \eqref{SW0.6} is exactly a minimizer for the fractional Sobolev inequality (and the sharp constant that can be obtained from \eqref{SW0.6}), see more on that in \cite{CCO2006}, \cite{Li2004}.
\end{remark}
As a direct consequence of the Pohozaev identities deduced in Lemma \ref{PHident} and \eqref{SW0.6}, we establish the following non-existence criteria for solutions of the equation \eqref{SWEQ}.
\begin{proposition}\label{nonexsw}
The equation \eqref{SHBO-IVP} can not have a smooth non-trivial vanishing at infinity solitary-wave solution unless either one of the following holds
\begin{itemize}
 \item[(i)] $\nu_1<0$, $c>0$, $1<m<\frac{1+s}{1-s}$, 
   \item[(ii)] $\nu_1>0$, $c<0$, $1<m<\frac{1+s}{1-s}$,
   \item[(iii)] $\nu_1>0$, $c>0$, $m>\frac{1+s}{1-s}$, 
   \item[(iv)] $\nu_1<0$, $c<0$, $m>\frac{1+s}{1-s}$, 
   \item[(v)] $m>1$ is an odd integer, $\nu_2 c>0$, and either (i) and (ii) hold, 
   \item[(vi)] $m>1$ is an odd integer, $\nu_2c<0$, and either (iii) and (iv) hold, or
  \item[(vii)] $m=\frac{1+s}{1-s}$ is an odd integer, $c=0$, and $\nu_1\nu_2<0$.
\end{itemize}
\end{proposition}

\begin{remark}\label{remarothereq}
\begin{enumerate}
\item 
When $m>1$ is an odd integer, we know that $\varphi^{m+1}=|\varphi|^{m+1}\geq 0$, thus, we can use \eqref{Phideneq2} to obtain restrictions on the sign of $\nu_2$. This is exactly (v)-(vii) in Proposition \ref{nonexsw}. This explains why we included the parameter $\nu_2$ in \eqref{SHBO-IVP} (in part to distinguish the $\int \varphi^{m+1} \, dx dy$ integral from the $L^{m+1}$-norm $\int |\varphi|^{m+1} \, dx dy$). 

\item 
For the case of the nonlinearity in \eqref{E:AfZK}, the solitary wave solution of the form $u(x,y,t)=\phi(x-ct,y)$ yields the equation 
\begin{equation}\label{alterGSE}
   -c\phi+\nu_1 (-\Delta)^{s}\phi+\frac{\nu_2}{m}|\phi|^{m-1}\phi=0.
\end{equation}
Then, replacing $\varphi$ by $\phi$, and $\varphi^{m+1}$ by $|\phi|^{m+1}$, the estimate \eqref{Phideneq1} holds 
for solutions of \eqref{alterGSE}, and in this case \eqref{Phideneq2} becomes 
\begin{equation*}
\| \phi\|_{L^{m+1}(\mathbb{R}^2)}^{m+1}=\frac{scm(m+1)}{\nu_2(2-(1-s)(m+1))}\| \phi \|_{L^2(\mathbb{R}^2)}^2.
\end{equation*}
Furthermore, the Prop. \ref{nonexsw} is modified as follows:
\end{enumerate}
\end{remark}
\begin{proposition}\label{nonexsw-2}
The equation \eqref{alterGSE} can not have a smooth non-trivial vanishing at infinity solution unless either one of the following holds
\begin{itemize}
 \item[\textbullet] $\nu_1<0$, $\nu_2=1$, $c>0$, $1<m<\frac{1+s}{1-s}$, 
   \item[\textbullet] $\nu_1>0$, $\nu_2=-1$ $c<0$, $1<m<\frac{1+s}{1-s}$,
   \item[\textbullet] $\nu_1>0$, $\nu_2=-1$, $c>0$, $m>\frac{1+s}{1-s}$, or
   \item[\textbullet] $\nu_1<0$, $c<0$, $\nu_2=1$, $m>\frac{1+s}{1-s}$,
   \item[\textbullet] $c=0$, $\nu_1 \nu_2<0$, $m=\frac{1+s}{1-s}$. 
\end{itemize}

\end{proposition}
\smallskip

\subsection{Uniform bounds (conditional global existence)}\label{S:bounds}
For our next result, we first find an explicit relation between the sharp constant and the solution of \eqref{optconsg}. Let $\varphi > 0$ be a local minimizer of the functional $J$ defined in \eqref{optconsg1}. Then $J'(\varphi)=0$ implies
\begin{equation}\label{miniequ}
\Big( (m+1)-\frac{(m-1)}{s} \Big)\, c_1\varphi+\frac{(m-1)}{s}(-\Delta)^{s}\varphi-c_2\varphi^{m}=0
\end{equation}
for some (specific) positive constants $c_1$ and $c_2$. By \eqref{miniequ}, and using similar arguments as in the proof of Proposition \ref{nonexsw}, we deduce the following identities
\begin{align}
&\int \left[ \Big( (m+1)-\frac{(m-1)}{s} \Big) \, c_1 \varphi^2+\frac{(m-1)}{s}|(-\Delta)^{\frac{s}{2}}\varphi|^2-c_2\varphi^{m+1} \right] \, dx dy=0, \label{eqminiequ} \\
&\int \left[ \Big( (m+1)-\frac{(m-1)}{s} \Big) \,c_1\varphi^2+\frac{(m-1)(1-s)}{s}|(-\Delta)^{\frac{s}{2}}\varphi|^2-\frac{2c_2}{m+1}\varphi^{m+1} \right] \, dx dy=0.  \label{eqminiequ1}
\end{align}
Combining \eqref{eqminiequ} and \eqref{eqminiequ1}, it is seen that
\begin{align}
& \|(-\Delta)^{\frac{s}{2}}\varphi\|_{L^2(\mathbb{R}^2)}^2 =c_1 \| \varphi \|_{L^2(\mathbb{R}^2)}^2,  \label{eqminiequ1.1} \\
&\|\varphi\|_{L^{m+1}(\mathbb{R}^2)}^{m+1}=\frac{c_1(m+1)}{c_2} \| \varphi \|_{L^2(\mathbb{R}^2)}^2. \label{eqminiequ2}
\end{align}
Setting $\beta_1^{m-1}=\frac{c_2sm}{(s(m+1)-(m-1))c_1}$ and $\beta_2^{2s}=\frac{(m-1)}{(s(m+1)-(m-1))c_1}$, we find that $Q_{s,m}(x,y)=\beta_1 \varphi(\beta_2x,\beta_2y)$ solves 
\begin{align}\label{E:GGS}
Q_{s,m}+(-\Delta)^{s}Q_{s,m}-\frac{1}{m}Q^{m}_{s,m}=0.
\end{align}
Note that from \eqref{poldecayGS} it follows that $|Q_{s,m}(x,y)| \sim \frac{1}{1+|(x,y)|^{2+2s}}$ for all $1<m <\frac{1+s}{1-s}$. Since $C_{GN}=\frac{1}{J(\varphi)}$, \eqref{eqminiequ1.1} and \eqref{eqminiequ2} imply  
\begin{equation}\label{eqminiequ4}
C_{GN}=\frac{ms(m+1)}{(m-1)^{\frac{m-1}{2s}}\Big(s(m+1)-(m-1)\Big)^{\frac{2s-(m-1)}{2s}}}\frac{1}{\|Q_{s,m}\|_{L^2(\mathbb{R}^2)}^{m-1}}.
\end{equation}
Summarizing the previous discussion, we obtain the following result.
\begin{proposition}
Let $0<s<1$, $f\in H^{s}(\mathbb{R}^2)$, then $f\in L^{m+1}(\mathbb{R}^2)$ for any $1<m<\frac{1+s}{1-s}$, and there is a constant $C_{GN}$ such that \eqref{optconsg} holds true. Moreover, the sharp constant for which this inequality is valid is given by \eqref{eqminiequ4} with $Q_{s,m}$ being a ground state solution of \eqref{E:GGS}.
\end{proposition}

The Gagliardo--Nirenberg inequality \eqref{optconsg} is also convenient to obtain uniform bounds for solutions of \eqref{SHBO-IVP} in each regime of criticality established by \eqref{CritiInd}. Before, we recall that by a standard parabolic regularization argument, for integer powers $m>1$ and a given $u_0 \in H^{r}(\mathbb{R}^2)$, $r>2$, there exist a time $T>0$ and a unique solution $u\in C([0,T);H^r(\mathbb{R}))$ of the initial value problem associated to \eqref{SHBO-IVP} such that $u(0)=u_0$. We use this existence result to formulate the following proposition.

\begin{proposition}\label{propcritiFir}
Let $0<s<1$, $m>1$ be an odd integer and $\nu_1\nu_2>0$. Consider $u_0\in H^{r}(\mathbb{R}^2)$, $r>2$. Then the solution $u\in C([0,T);H^r(\mathbb{R}^{2}))$ of the initial value problem associated to \eqref{SHBO-IVP} with initial data $u_0$ is uniformly bounded in $H^{s}(\mathbb{R}^2)$ for any $t\in [0,T)$. 
\end{proposition}

When $m>1$ is an odd integer number and $\nu_1\nu_2>0$, the proof of the above proposition is a direct consequence of the energy \eqref{genergy} and the $L^2$ conservation law as follows 
\begin{equation*}
\begin{aligned}
\frac{1}{2}\|(-\Delta)^{\frac{s}{2}}u(t)\|_{L^2(\mathbb{R}^2)}^2&=E_s[u_0]-\frac{\nu_2}{\nu_1m(m+1)}\|u\|_{L^{m+1}(\mathbb{R}^2)}^{m+1} \\
&\leq E_s[u_0].
\end{aligned}
\end{equation*}
\begin{remark}
If the nonlinearity in \eqref{SHBO-IVP} is modified by $\partial_{x}(|u|^{m-1} u)$ as in \eqref{E:AfZK}, then the above proposition holds for any $m>1$ for the modified equation.
\end{remark}

To obtain a similar result to Proposition \ref{propcritiFir} in the case when $\nu_1 \nu_2<0$, we require the existence of solutions to \eqref{E:GGS}. For simplicity, we consider $\nu_1=-1$ and $\nu_2=1$ in \eqref{SHBO-IVP}, that is, the two-dimensional model in \eqref{E:fZK}.

\begin{theorem}\label{propcriti}
Let $0<s<1$ and $1<m\leq \frac{1+s}{1-s}$ be an integer\footnote{If the nonlinearity in the equation is modified to $\partial_x(|u|^{m-1}u)$, then any value $m \in (1,\frac{1+s}{1-s} ]$ can be considered, conditional on the local well-posedness.}. Consider $u_0\in H^{r}(\mathbb{R}^2)$, $r>2$, and $u\in C([0,T); H^r(\mathbb{R}^{2}))$ be the solution of the initial value problem associated to \eqref{E:fZK} with initial data $u_0$. 
\begin{itemize}
\item[(C1)] 
Assume $1<m<2s+1$. Then the solution $u(t)$ is uniformly bounded in $H^{s}(\mathbb{R}^2)$ for any $t\in [0,T)$.
\item[(C2)] 
Assume  $m=2s+1$ and
\begin{equation*}
\|u_0\|_{L^2(\mathbb{R}^2)} <\|Q_{s,m}\|_{L^2(\mathbb{R}^2)},
\end{equation*}
where $Q_{s,m}$ is the ground state solution of \eqref{E:GGS}. Then the solution $u(t)$ is uniformly bounded in $H^{s}(\mathbb{R}^2)$ for any $t\in [0,T)$.
\item[(C3)] 
Let $\theta = \frac{r_c}{s} \equiv \frac1{s} - \frac2{m-1}$. Assume $2s+1<m\leq \frac{1+s}{1-s}$, or equivalently, $0<\theta \leq 1$, and $E[u_0] \geq 0$. 
Suppose
\begin{equation}\label{eqrcritinde2}
\begin{aligned}
E_s[u_0]^{\theta} M[u_0]^{1-\theta} < E_s[Q_{s,m}]^{\theta} M[Q_{s,m}]^{1-\theta},
\end{aligned}
\end{equation}
where $Q_{s,m} > 0$ is the ground state solution of \eqref{E:GGS}.

If 
\begin{equation}\label{eqrcritinde}
\begin{aligned}
\|(-\Delta)^{\frac{s}{2}}u_0\|_{L^2(\mathbb{R}^2)}^{\theta} \|u_0\|_{L^2(\mathbb{R}^2)}^{1-\theta} < \|(-\Delta)^{\frac{s}{2}} Q_{s,m}\|_{L^2(\mathbb{R}^2)}^{\theta} 
\|Q_{s,m}\|_{L^2(\mathbb{R}^2)}^{1-\theta},
\end{aligned}
\end{equation}
then the solution $u(t)$ of \eqref{E:fZK} with the initial condition $u_0$ is uniformly bounded in $H^{s}(\mathbb{R}^2)$ for any $t\in [0,T)$. Moreover, for any $t\in [0,T)$
\begin{equation}\label{E:grad-t}
\begin{aligned}
\|(-\Delta)^{\frac{s}{2}}u(t)\|_{L^2(\mathbb{R}^2)}^{\theta} \|u_0\|_{L^2(\mathbb{R}^2)}^{1-\theta} < \|(-\Delta)^{\frac{s}{2}} Q_{s,m}\|_{L^2(\mathbb{R}^2)}^{\theta} 
\|Q_{s,m}\|_{L^2(\mathbb{R}^2)}^{1-\theta}.
\end{aligned}
\end{equation}
\end{itemize}
\end{theorem}

The proof for (C1) and (C2) of Proposition \ref{propcriti} follows Weinstein's classical approach \cite{We1983}, for (C3) see \cite{FaraLinaPast2014}, \cite{HR2008}, \cite{HR2007}, we provide the details below after making several comments. 

\begin{remark}
\begin{enumerate}
\item
The conclusion of Proposition \ref{propcriti} is still valid for $m>1$, not necessarily an integer, provided that for any $u_0\in H^r(\mathbb{R})$, $r\geq s$, there exist $0<T\leq \infty$, and a unique solution $u\in C([0,T);H^r(\mathbb{R}^2))$ of \eqref{E:fZK} with the initial condition $u_0$.
\item
We observe that (C1), (C2) and (C3) correspond to the $L^2$-subcritical, critical, and supercritical cases, respectively.  
\item
When $m=2s+1$, we expect the following conjecture to hold:
\end{enumerate}

\begin{conjecture}\label{C:critical}
Let\,\footnote{An ultimate goal would be $u_0 \in L^2(\mathbb R^2)$.}  
{$u_0 \in H^s(\mathbb R^2)$}\,
and $Q=Q_{s,2s+1}$ be the ground state solution of \eqref{E:GGS}. 
\begin{itemize}
\item[I.] 
If $\|u_0\|_{L^2(\mathbb{R}^2)}<\|Q\|_{L^2(\mathbb{R}^2)}$, then the solution $u(t)$ exists globally in time. 
\item[II.] 
If $\|u_0\|_{L^2(\mathbb{R}^2)} > \|Q\|_{L^2(\mathbb{R}^2)}$ and $u_0$ is sufficiently localized, then the solution $u(t)$ blows up in finite time. 
In particular, if $E[u_0]<0$ (hence, $\|u_0\|_{L^2(\mathbb{R}^2)} > \|Q\|_{L^2(\mathbb{R}^2)}$) and $u_0$ has some localization implies blow-up in finite time. 
\end{itemize}
\end{conjecture}
\begin{enumerate}
\item[]
In (C2) we prove the part I of Conjecture \ref{C:critical} for $u_0 \in H^r(\mathbb R^2)$, $r>2$, and make a partial progress for $u_0 \in H^s(\mathbb R^2)$ (conditional on the local wellposedness in $H^s$). The second part of Conjecture \ref{C:critical},
when $s=\frac{1}{2}$, $m=2$, is confirmed numerically in subsection \ref{Blow-up} for initial data with different decays at infinity (we show that there are blow-up solutions with positive and negative energy). Moreover, a stable blow-up regime is self-similar (in the core region) with the rescaled ground state as the blow-up profile.

\item
We note that (C3) is a generalization of \cite{FaraLinaPast2014}  for fKdV (see also \cite{SvetKaiWangBO} results for the gBO equation), and in a more general sense, is a generalization of the dichotomy first obtained for the NLS in \cite{KM2006} and \cite{HR2008}, where the opposite inequality in \eqref{eqrcritinde} was also possible to consider.
\end{enumerate}

\end{remark}

\begin{proof}[Proof of Proposition \ref{propcriti}]
From the definition of energy \eqref{genergy} and the $L^2$ conservation, together with \eqref{optconsg},  we get
\begin{equation}\label{enerineq}
\begin{aligned}
\|(-\Delta)^{\frac{s}{2}}u(t)\|_{L^2(\mathbb{R}^2)}^2&=2E_s[u(t)]+\frac{2}{m(m+1)}\int (u(x,y,t))^{m+1}\, dx dy\\
&\leq 2E_s[u_0]+\frac{2C_{GN}}{m(m+1)}\|(-\Delta)^{\frac{s}{2}} u(t)\|_{L^2(\mathbb{R}^2)}^{\frac{m-1}{s}}\|u_0\|_{L^2(\mathbb{R}^2)}^{\frac{s(m+1)-(m-1)}{s}}.
\end{aligned}
\end{equation}

If $1<m<2s+1$, the above inequality shows that $\|(-\Delta)^{\frac{s}{2}}u(t)\|_{L^2(\mathbb{R}^2)}$ is uniformly bounded for all $t\in [0,T)$. This completes the proof of (C1). 

If $m=2s+1$, we deduce the uniform bound if
\begin{equation}
\Big(1-\frac{2C_{GN}}{m(m+1)}\|u_0\|_{L^2(\mathbb{R}^2)}^{2s}\Big)>0.
\end{equation}
Plugging the sharp constant \eqref{eqminiequ4} into the above expression yields (C2).

Next, we consider $2s+1<m \leq \frac{1+s}{1-s}$.
Multiplying both sides of \eqref{enerineq1} by $\|u_0\|_{L^2}^{2(\frac1{\theta}-1)}$, we get 
\begin{equation}\label{E:ener2}
\begin{aligned}
\|(-\Delta)^{\frac{s}{2}}u(t)\|_{L^2(\mathbb{R}^2)}^2 \|u_0\|_{L^2}^{2(\frac1{\theta}-1)}
& \leq 2 E_s[u_0] M[u_0]^{\frac1{\theta}-1} \\
& + \frac{2C_{GN}}{m(m+1)}
\left( \|(-\Delta)^{\frac{s}{2}} u(t)\|_{L^2(\mathbb{R}^2)}^2
\|u_0\|_{L^2(\mathbb{R}^2)}^{2(\frac1{\theta}-1)} \right)^{\frac{m-1}{s}}.
\end{aligned}
\end{equation}
Setting 
\begin{equation*}
x(t)=\|(-\Delta)^{\frac{s}{2}}u(t)\|_{L^2(\mathbb{R}^2)}^2 \|u_0\|_{L^2}^{2(\frac1{\theta}-1)}, ~~ \text{ and } ~~ \, \beta=\frac{2C_{GN}}{m(m+1)},
\end{equation*}
we rewrite \eqref{E:ener2} as
\begin{equation*}
x(t)-\beta \, x(t)^{\frac{m-1}{2s}}\leq 2 E_s[u_0] M[u_0]^{\frac1{\theta}-1}.
\end{equation*}
We note that the function $f(x):=x-\beta x^{\frac{m-1}{2s}}$, $x\geq 0$, has a local maximum at $x_{0}=\big( \frac{1-r_c}{\beta} 
\big)^{\frac{1-r_c}{r_c}}$ with the maximum value $f(x_0)=r_c \cdot \big( \frac{1-r_c}{\beta} \big)^{\frac{1-r_c}{r_c}}$.
We impose the conditions
\begin{equation}\label{enerineq1}
2E_s[u_0]  M[u_0]^{\frac1{\theta}-1} <f(x_0) \, \, \text{ and } \, \, x(0)<x_0.
\end{equation}
By a continuity argument, it must be the case that $x(t)\leq x_0$ for any time $t$, when the solution $u(t)$ is defined. Thus, once we have the explicit relations for \eqref{enerineq1}, the above discussion yields the proof of (C3). By \eqref{eqminiequ4} 
we deduce the identity 
\begin{equation}\label{enerineq2.1}
\|Q_{s,m}\|_{L^2(\mathbb{R}^2)}^2=\frac{2(2-(1-s)(m+1))}{m-1-2s}E_s[Q_{s,m}],
\end{equation}
then a computation shows that  the first condition in \eqref{enerineq1} is equivalent to
\begin{equation}\label{enerineq2}
\begin{aligned}
E_s[u_0] M[u_0]^{\frac1{\theta}-1} < 
E_s[Q] M[Q]^{\frac1{\theta}-1},
\end{aligned}
\end{equation}
or \eqref{eqrcritinde}.
Likewise, $x(0)<x_0$ is equivalent to \eqref{eqrcritinde2} and $x(t)<x_0$ is rephrased as \eqref{E:grad-t}, finishing the proof.
\end{proof}


\subsection{Linear equation and radiation}\label{S:radiation}
When studying solitary waves in KdV-type equations, one ultimately encounters such regions as the soliton core region, the fast decaying tail (typically, located to the right of the moving soliton in the $1d$ problems), and the radiation region, where the dispersive oscillations propagate (typically to the left of the soliton). In higher dimensional problems similar behavior was observed and obtained in ZK models (see \cite{AsymStaZK3D}, \cite{AsymStaZK2D}, \cite{FHRY}, \cite{KRS2}, \cite{KRS1}). As we deal with the $2d$ HBO equation, we also  
investigate a region in the plane, where a solitary-wave type solution of \eqref{E:HBO} propagates dispersive oscillations, or radiates into that region.
For that we consider a linear evolution initial-value problem
\begin{equation}\label{leqHBO}
\left\{\begin{aligned}
&\partial_t u-\partial_x(- \Delta)^s u=0, \, \, (x,y,t) \in \mathbb{R}^3, \, \, 0<s<1,\\
&u(x,0)=u_0(x).
\end{aligned}\right.
\end{equation}
We first remark that Strichartz estimates for the equation \eqref{leqHBO} are known for any dimension $d\geq 2$ when $s=\frac{1}{2}$, and for $\frac{1}{2}<s<1$ when $d\geq 3$, see \cite{linaO,RobertS}. As far as we know, Strichartz estimates for \eqref{leqHBO} have not been determined for dispersions $0<s<\frac{1}{2}$. For dimension 1, see \cite{KPVOsint1D}. On the other hand, for a sufficiently regular initial condition $u_0$ the solution of \eqref{leqHBO} is
\begin{equation}
\begin{aligned}
u(x,y,t)=\frac{1}{(2\pi)^2}\int e^{ix\xi_1+iy \xi_2+i\omega(\xi_1,\xi_2)t}\widehat{u_0}(\xi_1,\xi_2) \, d\xi_2 d\xi_2,
\end{aligned}
\end{equation}
where $\omega(\xi_1,\xi_2)=\xi_1|(\xi_1,\xi_2)|^{2s}$ is the dispersion relation. Then the group velocity is given by 
\begin{equation}
\begin{aligned}
\nabla \omega(\xi_1,\xi_2)=|(\xi_1,\xi_2)|^{2s-2}\bigg((1+2s)\xi_1^2+\xi_2^2,2s\xi_1\xi_2 \bigg).
\end{aligned}
\end{equation}
The angle $\theta(\xi_1,\xi_2)$, determined by $\nabla \omega (\xi_1,\xi_2)$ and the positive $y$-axis, satisfies the relation 
\begin{equation}
\tan(\theta(\xi_1,\xi_2))=\frac{(1+2s)\xi_1^2+\xi_2^2}{2s\xi_1\xi_2}.
\end{equation}
Setting a reference frame centered at the center of a moving soliton, we obtain the minimal angle determined by the above relation, which satisfies the identity
\begin{equation}\label{E:angle}
\tan(\theta_{min})=\frac{(1+2s)^{\frac{1}{2}}}{s},
\end{equation}
see Figure \ref{F:angle} for a depiction. 

\begin{figure}[ht]
\includegraphics[width=0.45\textwidth]{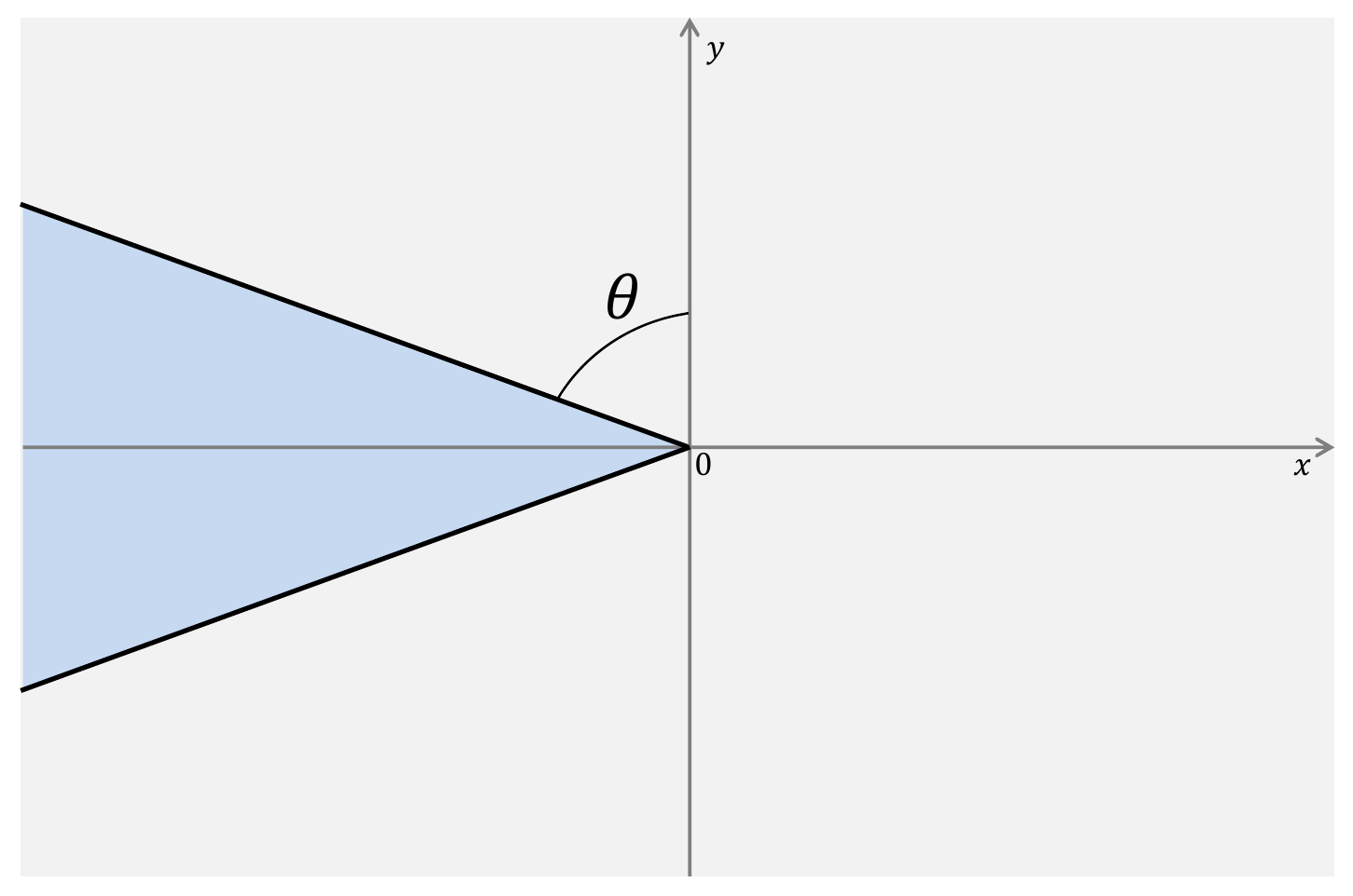}
\caption{The angle $\theta = \theta_{min}$ in \eqref{E:angle}, the blue area is the radiation region (in the frame moving with the soliton).}
\label{F:angle}
\end{figure}

For the case of the HBO equation ($s=\frac{1}{2}$), we have 
\begin{equation}\label{E:angleHBO}
\tan(\theta_{min})=2\sqrt{2},
\end{equation} 
or approximately, $\theta_{min} \approx 70. 52^{\circ}$. Therefore, for an appropriately set up reference frame (moving with the solitary wave), we use $\theta_{min}$  to define the wedge of the radiation region with an angle of a maximum value of $19. 48^{\circ}$ with the $x$-axis (or by symmetry a total angle of  $38.96^{\circ}$), see the blue region in Figure \ref{F:angle}. 
We provide numerical confirmation of the radiation region in Section \ref{S:Numerical solution} (for example, see Figures \ref{F:angle t3} and \ref{F:angle t5}).

\begin{remark}
In the case of the ZK equation ($s=1$ and $m=2$ in \eqref{E:fZK}), the argument above gives an angle for the radiative region of $60 ^{\circ}$, which is compatible with the results presented in \cite{AsymStaZK2D} (see Remarks 1.1 and 1.2), also in \cite{FHRY,
KRS1}. In particular, considering the surface determined by the ZK equation, the angle above is related to the region, where the mitigating factor introduced by the Strichartz estimate in \cite{CaKenZiesl} cancels out. Furthermore, the argument above can also be extended to \eqref{E:fZK} in higher dimensions (note that the nonlinear part does not play a role on the size of the angle of this region). In particular, the same representative angle was obtained for the $3d$ ZK equation in \cite{AsymStaZK3D}, see also \cite{KRS2}.
\end{remark}

We are now ready to study the equation \eqref{E:HBO} numerically; for that we first describe the numerical approach that we develop for the 2d HBO \eqref{E:HBO}, then we obtain the ground state solutions, and finally, we show the dynamical solutions of the $2d$ HBO equation. 


\section{Numerical approach}\label{S:Numerical method}
In this section, we describe our numerical method that we develop to solve the HBO equation \eqref{E:HBO} on the whole real space $\mathbb{R}^2$. We start with the description of the rational basis (eigen)functions, or also referred to as Wiener functions, on $\mathbb{R}$. Then, we apply the Galerkin approximation to the fractional Laplacian $(-\Delta)^{\frac{1}{2}}$ from \cite{SSTWY2020}. After the space discretization, the equation is reduced to the nonlinear ODE system, which can be solved by a variety of standard numerical integrators.

\subsection{Rational basis functions}
One method for solving the dispersive equations on $\mathbb{R}^d$ is to use the Fourier spectral discretization in space by taking a sufficiently large domain. Then, the fractional Laplacian is discretized in a straightforward manner on the frequency space, i.e., $\widehat{(-\Delta)^s u}= |\xi|^{2s}\hat{u}$ (for example, as it is done in \cite{LSS2013} and  \cite{KP2015}). However, this requires extremely large periodic domains to ensure a sufficiently good approximation of the whole real space $\mathbb{R}^d$, as the fractional Laplacian term $(-\Delta)^s u$ decays only with an algebraic rate. A large number of grid points is needed for a satisfactory resolution. Furthermore, despite of the large length of the domain and number of grid points, the Fourier discretization is still unable to capture the asymptotic behavior such as the decay rate when the solution approaches the computational boundary, since the numerical solution 
wraps around the boundary, while the actual solution decays at $\infty$.   

An alternative space discretization from \cite{Christov82} considers the use of the rational basis functions (sometimes also called Wiener functions), on the whole real line, i.e., 

\begin{equation}\label{rational}
 u(x,t)=\sum_{n=-\infty}^{\infty} \tilde{u}_{n}(t)\rho_n(x), \quad \rho_n(x)=\frac{(\alpha+ix)^n}{(\alpha-ix)^{n+1}},
\end{equation}
where $\alpha$ is a mapping parameter indicating that half of the grid points are located in the interval $[-\alpha,\alpha]$. It is shown in \cite{Christov82} that $\{ \rho_n(x) \}_{n=-\infty}^{\infty}$ 
form a complete orthogonal basis in $L^2(-\infty, \infty)$ with the orthogonality
$$ 
\int_{-\infty}^\infty \rho_m(x) \overline{\rho_n(x)}dx = \begin{cases}
      \pi/\alpha, &  m=n\\
      0, & m\neq n.
    \end{cases}   :=\frac{\pi}{\alpha} \delta_{m,n} .   
$$
Therefore, we have 
$$
\tilde{u}_n(t)=\frac{\alpha}{\pi}\int_{n=-\infty}^{\infty} u(x,t)\rho_n(x)dx.
$$

The derivatives of $u(x,t)$ can be easily computed by the relation
\begin{align}
u_x(x,t)=\sum_{n=-\infty}^{\infty} &\frac{i}{2\alpha}[n\tilde{u}_{n-1}+(2n+1)\tilde{u}_n+(n+1)\tilde{u}_{n+1}]\rho_n(x), \label{ux}
\end{align}
and the higher derivatives can then be done iteratively.

In numerical computations, a truncation of $N$ terms is used, i.e.,
$$ 
u(x,t)\approx \tilde{\mathbf{u}}^T \mathbf{\rho}:=\sum_{n=-N/2}^{N/2-1} \tilde{u}_n(t)\rho_n(x),
$$
where $\tilde{\bf{u}}=(\tilde{u}_{-N/2},\tilde{u}_{-N/2+1}, \cdots, \tilde{u}_{N/2-1})^T$ is the vector of the truncated coefficients, and the same for ${\rho}$.
This leads to the sparse matrix forms
\begin{align}\label{E: rational basis derivative}
 u_x \approx [\mathbf{S_1} \tilde{\mathbf{u}}]^T {\rho}, \quad u_{xx} \approx [\bf{S_2} \tilde{\bf{u}}]^T {\rho},
\end{align}
where $\bf{S_{1}}$ is given in \eqref{ux} via the coefficients of $\tilde{u}_n$, and $\mathbf{S_2}=\mathbf{S_1}\times \mathbf{S_1}$.

Now by a change of variable 
$$
x=\alpha\tan \frac{\theta}{2}, \quad -\pi\leq \theta \leq \pi,
$$
and a spatial discretization $x_j=\alpha\tan \frac{\theta_j}{2}, \theta_j=jh, h=2\pi/N, j=-N/2, \cdots, N/2$, we have
\begin{equation}\label{uj}
u_j=u(t,\alpha-ix_j)=\sum_{n=-N/2}^{N/2-1}\tilde{u}_n e^{i\theta_j}.
\end{equation}
The Fast Fourier transform (FFT) can be applied to obtain the coefficients $\tilde{u}_n$. This approximation can be easily extended to higher dimensions with a tensor product.

\subsection{Discretization of the fractional Laplacian 
on $\mathbb{R}^2$}
In this section, we describe the Galerkin approximation of the fractional Laplacian $(-\Delta)^s$ on $\mathbb{R}^2$. This method was introduced by Shen in \cite{SSTWY2020} with the Mapped Chebyshev functions. It can be easily applied to the rational basis functions, and can also be extended to other dimensions. 

We note that $\mathbf{S_1}=i\mathbf{S}$, where $\mathbf{S}$ is a real symmetric matrix from \eqref{ux}. Therefore, $\mathbf{S}$ is diagonalizable with all real eigenvalues written as
$$
\bf{S}=E \Lambda E^T,
$$
where $\mathbf{E}=(e_{j,k})_{j,k=-\frac{N}{2},\cdots \frac{N}{2}-1}$ is the matrix formed by the orthonormal eigenvectors $\vec{e}_k$ of $\bf{S}$, and $\mathbf{\Lambda}$ is the real diagonal matrix. Then, we have
\begin{align}\label{E:S1 S2 hat}
\mathbf{S_1}=i \mathbf{E \Lambda E^T}, \qquad \mathbf{S_2}=-\mathbf{E (\Lambda^2) E^T}.
\end{align}
For a matrix $\mathbf{M}$, we denote $\mathbf{M}(j,k)=m_{j,k}$ to be the $j$th row and $k$th column element. 
Denote $\lambda_k=\mathbf{\Lambda^2}(k,k)$ to be the $k$th eigenvalue of the diagonal matrix $\mathbf{\Lambda^2}$.

Consider the set of new basis functions $ \lbrace\hat{\rho}_k(x)\rbrace$, which is obtained as the diagonal transformation of the old basis $\lbrace\rho_k\rbrace$, i.e.,
\begin{align}\label{E:basis transform}
\hat{\rho}_k(x):= \sum_{j=-N/2}^{N/2-1} e_{j,k} \rho_j(x), \quad \vec{e}_k=(e_{-N/2,k}, \cdots, e_{N/2-1, k})^T.
\end{align}
Then, from the direct adaption of the proof in \cite[Lemma 2.1]{SSTWY2020}, we have
\begin{align}\label{E:basis orthogonal}
(\hat{\rho}_k, \hat{\rho}_j)_{L^2}= \frac{\pi}{\alpha} \delta_{k,j}, \quad \mbox{and} \quad (\hat{\rho}'_k, \hat{\rho}'_j)_{L^2}=\frac{\pi}{\alpha} \lambda_k \delta_{k,j},
\end{align}
which is called the ``biorthogonal property".

Let $\mathbf{\hat{u}}=(\hat{u}_{-N/2}, \cdots \hat{u}_{N/2-1})^T$.
It is easy to see that $\bf{\hat{u}=E^T\tilde{u}}$ in 1d. In 2d, let $\mathbf{\tilde{U}}=(\tilde{u})_{j,k=-\frac{N}{2},\cdots \frac{N}{2}-1}$ be the matrix of  the coefficients with respect to the basis functions $\lbrace \rho_{j,k} \rbrace$. Then, we have the coefficients matrix with respect to the basis functions $\lbrace \hat{\rho}_{j,k} \rbrace$ obtained by
$$
\bf{\hat{U}}=E^T\bf{\tilde{U}}E.
$$

Now, let $(-\Delta)^s u=f$. We show that $\hat{f}_{j,k}=(\lambda_j+\lambda_k)^s \hat{u}_{j,k}$ in $\mathbb{R}^2$. 
In \cite{SSTWY2020}, the authors developed the Galerkin approach to approximate the fractional Laplacian $(-\Delta)^s$ by using the Mapped Chebyshev basis functions in $\mathbb{R}$ (or $\mathbb{R}^d$). This can be applied to other sets of basis functions with the ``biorthogonal property", including the ``biorthogonal rational basis functions" that we described in \eqref{E:basis transform}. 

Indeed, in \cite{BLP2019}, the authors use the Dunford-Taylor formula to change the fractional Laplacian $(-\Delta)^s$ into the full Laplacian
$(-\Delta)$ in the Galerkin formulation, 
\begin{align}\label{E:flaplacian}
\left( (-\Delta)^{\frac{s}{2}}u, (-\Delta)^{\frac{s}{2}}v \right)_{L^2(\Omega)}= C_s \int_0^{\infty} t^{1-2s}\int_{\Omega} (-\Delta) (\mathbb{I}-t^2 \Delta)^{-1}u(x)v(x)dxdt,
\end{align} 
where $C_s=\frac{2\sin(\pi s)}{\pi}$, $\mathbb{I}$ is the identity operator and $\Omega$ can be either the bounded domain in $\mathbb{R}^d$ or $\Omega = \mathbb{R}^d$. Denote $w(x)=(\mathbb{I}-t^2 \Delta)^{-1}u(x)$. Then,
\begin{align}\label{E:flaplacian2}
-t^2\Delta w+w=u, \quad x  \in \Omega, 
\end{align}
and thus, 
\begin{align*}
(-\Delta) (\mathbb{I}-t^2 \Delta)^{-1}u(x)=-\Delta w=t^{-2}(u-w).
\end{align*}
Solving the equation \eqref{E:flaplacian2} and then evaluating the integral \eqref{E:flaplacian} with respect to $t$ becomes a crucial step in accurately evaluating the fractional Laplacian $(-\Delta)^s $. Subsequent works in \cite{BLP20192} proposed the method for evaluating the integral \eqref{E:flaplacian} by the $sinc$ functions. Later, the authors in \cite{SSTWY2020} observed the following integral identity 
$$ 
\int_0^{\infty} \frac{t^{(1-2s)}\lambda^{1-s}}{1+t^2 \lambda} dt= \frac{\pi}{2\sin(\pi s)} =\frac{1}{C_s},
$$ 
and thus, the integral system \eqref{E:flaplacian}--\eqref{E:flaplacian2} can be evaluated exactly in the frequency space
if we can write the inside integral $\int_\Omega t^{-2} (u-w) v dx$ in the diagonal form. As a consequence, the $(-\Delta)^s u$ is evaluated efficiently.

For example, let $\Omega$ be the periodic bounded domain, and we use the Fourier basis to approximate $u(x)$, i.e., $u(x) \approx u_N(x)=\sum_{-N/2}^{N/2-1} \hat{u}_k e^{ixk}.$ By setting $v_k=e^{-ixk}$ for each $k=-\frac{N}{2}, \cdots \frac{N}{2}-1$, we have
$\int_\Omega t^{-2} (u_N-w_N) v_k dx=\frac{k^2}{1+t^2k^2} \hat{u}_k.$ Putting into \eqref{E:flaplacian} for each $k$ yields
\begin{align*}
&\left( (-\Delta)^{\frac{s}{2}}u_N, (-\Delta)^{\frac{s}{2}}v_k \right)_{L^2(\Omega)}= C_s \int_0^{\infty} \frac{t^{(1-2s)} k^2}{1+t^2k^2} \hat{u}_k dt
&= k^{2s} \hat{u}_k C_s \int_0^{\infty} \frac{ t^{(1-2s)}k^{2(1-s)}}{1+t^2k^2} dt = 
|k^2|^s \hat{u}_k.
\end{align*}
In other words,
$$(-\Delta)^su(x) \approx (-\Delta)^su_N(x)= \sum_{-N/2}^{N/2-1} (|k^2|^s\hat{u}_k) e^{ixk},$$
which matches the form $\mathcal{F}\left((-\Delta)^s u\right)=|\xi|^{2s} \hat{u}$ in the usual sense.

When $\Omega= \mathbb{R}^d$, the Fourier basis is no longer preferrable, as the domain truncation may lead to large errors. The biorthogonal Mapped Chebyshev functionss from \cite{SSTWY2020} or the biorthogonal rational basis functions from \eqref{E:basis transform} can be used. For example, in 1d case, recall that $\mathbf{\Lambda^2}=\mathrm{diag}(\lambda_j)$ is the diagonal matrix and $\mathbf{E}=\lbrace e_{j,k} \rbrace $ is the orthonormal matrix. Then, we have 
$$(-\partial_{xx})^s u_N=\sum_{j=-N/2}^{N/2-1} |\lambda_j|^{s} \hat{u}_j \hat{\rho}_j(x).$$
Thus, the stiff matrix for $(-\partial_{xx})^s u_N$ is $\mathbf{\Sigma}=\mbox{diag}(|\lambda_j|^{s})$ for $j=-N/2,\cdots,N/2-1$. 

This can be extended to higher dimensional cases, see details in \cite{SSTWY2020}. In summary, let $\mathbf{\Sigma}$ be the stiff matrix for the fractional Laplacian $(-\Delta)^s$ in $\mathbb{R}^2$. 
We have $(-\Delta)^s u_N= \mathbf{\Sigma}_{j,k} \hat{u}_{j,k}=(\lambda_j+\lambda_k)^s \hat{u}_{j,k}$. 
In other words, we have $(-\Delta)^s u$ is equivalent to $\mathbf{\Sigma} \odot \mathbf{\hat{U}}$ on the frequency side, where $\mathbf{\Sigma}_{j,k}=(\lambda_j+\lambda_k)^s$, and $\odot$ denotes the pointwise product between the matrices ($\mathbf{A} \odot \mathbf{B}=(a_{jk}b_{jk})$). We also denote $ ^{\odot m}$ as the pointwise power of a matrix, e.g., for $m=2$, $\mathbf{U}^{\odot 2}=\mathbf{U} \odot \mathbf{U}$. 

Now, let $\mathbf{U}(t)\approx u(x_j,y_k,t)$ be the approximation of the solution $u(x,y,t)$ to \eqref{E:HBO}. The quantities $\mathbf{U, \ \tilde{U}, \ \hat{U}}$ can be obtained by FFT and the transformation matrix $\mathbf{E}$. To be specific, we first use the FFT from \eqref{uj} to find the coefficients $\tilde{u}_{j,k}$, then we use the relation $\mathbf{\hat{U}}=\bf{E^T {\tilde{U}} E}$ to obtain the biorthogonal coefficients $\hat{u}_{j,k}$, and similarly, going backward reversing the steps. 
Note that $\mathbf{S_1} \mathbf{\tilde{U}}= \mathbf{\hat{S}_1} \mathbf{\hat{U}}.$ Therefore, the first order stiff matrix, $\mathbf{\hat{S}_1=E^T S_1 E=E^T (\mathrm{i}E\Lambda E^T) E}=i\mathbf{\Lambda}$ with respect to the basis $\lbrace \hat{\rho}_{j,k} \rbrace$, is also diagonal.

Finally, the semi-discretization of the HBO equation \eqref{E:HBO} on the frequency space $\lbrace \hat{\rho}_{j,k} \rbrace$ yields 
\begin{align}\label{E:HBO space d}
\mathbf{\hat{U}}_t-\mathbf{\hat{S}_1(\Sigma \odot \hat{U})}+\frac{1}{m} \mathbf{\hat{S}_1}(\widehat{\mathbf{U}^ {\odot m}})=0.
\end{align}  

The matrix $\mathbf{\hat{U}}$ and $\mathbf{\Sigma}$ can be reordered into an $N^2 \times 1$ long vectors 
\begin{align*}
&\vec{\hat{U}}=(\hat{u}_{-N/2,-N/2},\hat{u}_{-N/2+1,-N/2},\cdots, \hat{u}_{-N/2,-N/2+1}, \cdots, \hat{u}_{N/2-1,N/2-1})^T;\\
&\vec{\Sigma}=(\mathbf{\Sigma}_{-N/2,-N/2},\mathbf{\Sigma}_{-N/2+1,-N/2},\cdots, \mathbf{\Sigma}_{-N/2,-N/2+1}, \cdots, \mathbf{\Sigma}_{N/2-1,N/2-1})^T.
\end{align*}
The stiff matrix $\mathbf{\hat{S}_1}$ for $\partial_x$ can be changed as $\mathbf{\hat{S}_1^x}=\mbox{kron}(\mathbf{I,\hat{S}_1})$, where $\mbox{kron}$ is the Kronecker product and $\mathbf{I}$ is the $N \times N$ identity matrix. Then, the equation \eqref{E:HBO space d} becomes
\begin{align}\label{E:HBO space d2}
\vec{\hat{U}}_t-\mathbf{\hat{S}^x_1}\left(\mbox{diag}(\vec{\Sigma} ) \vec{\hat{U}} \right)+\frac{1}{m} \mathbf{\hat{S}^x_1}(\widehat{\vec{U}^{\odot m}})=0.
\end{align}

The system \eqref{E:HBO space d2} only involves the diagonal matrices, and thus, can be solved with the computational cost $\mathcal{O}(N^2)$. However, we recall that $\mathbf{\hat{U}}=\bf{E^T {\tilde{U}} E}$, which is the multiplication between full matrices. Therefore, the transformation between the frequency space and the physical space needs $\mathcal{O}(N^3)$ operations, which is the total computational cost of our algorithm.

\subsection{Time integration}
The given system can be integrated by various time integrators. When applying the standard explicit time integrators, such as the Runge-Kutta (RK4) method, the time step $\Delta t$ has to be chosen to satisfy the so-called CFL condition, which is $\Delta t < \max_j{|\lambda_j|}^{-2s-1}$. 

To allow a larger time step size, the modified 4th order exponential time differencing (mETDRK4) from \cite{KT05} can be applied. This method allows us to take the time step $\Delta t \sim \max_j(|\lambda_j|)^{-1}$ due to the 1st order derivative on the nonlinear term $(\frac{1}{m}u^m)_x$. Other implicit Runge-Kutta methods can be used for the choice of even larger time steps, for example, the 4th order Runge-Kutta method with Gauss-Legendre collocation points (IRK4), which has been shown competitively efficient to the mETDRK4 method in simulating KdV equations (see e.g., \cite{YK2021}, \cite{KP2015}, \cite{BDKMS1995}). The resulting nonlinear system can be solved by the fixed point iteration, similar to \cite{SDKL2021} or \cite{KP2015}.
In our simulations, we used both the mETDRK4 and IRK4 methods, and the results match with each other.

\section{Computation of the ground state solution}\label{S:Soliton}
In this section, we show our numerical results for computing the ground state solution $Q$ that solves $-Q- (-\Delta)^s Q +\frac1{m}Q^m=0$, ${Q > 0,~~ Q \in H^{2s+1}(\mathbb R^2)}\cap C^{\infty}(\mathbb R^2)$, or the rescaled profiles of the solitary waves $Q_c$ in \eqref{SWEQ}, namely,
\begin{align}\label{E:Qc eqn}
c\, Q_c +(-\Delta)^s Q_c-\frac{1}{m}(Q_c)^m=0, 
\end{align}
where $c>0$ is a constant that generates a family of the rescaled ground state solutions \begin{equation}\label{E:Qc}
Q_c(x,y)=c^{\frac{1}{m-1}}\,Q(c^{\frac1{2s}}x, c^{\frac1{2s}} y),
\end{equation}
producing traveling solitary waves $u(x,y,t)=Q_c(x-c\,t,y)$.

We apply the Petviashvili's iteration to obtain the profiles $Q_c$. We give a brief review of this method, which 
has been well-studied in the literature and for details we refer the interested reader to \cite{Petviashvili1976}, \cite{PS2004}, \cite{OSSS2016}, \cite{LY2007}, \cite{YANG2009}, \cite{YLT2007}.

We denote by $Q_h$ our numerical solution for $Q_c$ in \eqref{E:Qc eqn} for a given $c>0$. Next, we define the operator $\mathbb{M}=(-\Delta)^s+ c \, \mathbb{I}$, where $\mathbb{I}$ is the identity operator.
Suppose that at the $l$th iteration, we obtained $Q_h^l$. We compute the constant (that will control our fixed-point iteration)
$$
\gamma_l = \left( \frac{m \langle Q_h^l, Q_h^l \rangle}{ \langle Q_h^l, \mathbb{M}^{-1} \left( (Q_h^l)^m \right) \rangle } \right)^{\frac{1}{m-1}}.
$$
Thus, we obtain the following iteration
\begin{align}\label{E:Q scheme}
Q_h^{l+1}=\frac{1}{m} \, \mathbb{M}^{-1} \Big( (\gamma_l \, Q_h^l)^m \Big).
\end{align}
We set the stopping criteria to be 
$\| Q_h^{l+1}-Q_h^l \|_{L^{\infty}(\mathbb{R}^2)} <Tol$ with $Tol=10^{-8}$ in our computations. 

\subsection{Ground state in the $2d$ (critical) HBO}\label{ground-state-HBO}
Until now the discussion in this paper has been for a general equation of type \eqref{E:fZK} or \eqref{E:AfZK}, and now this is where we completely turn to the $2d$ HBO equation \eqref{E:HBO}, that is, we only consider $m=2$ and $s=\frac12$ in \eqref{E:Qc eqn}.

\begin{figure}[ht]
\includegraphics[width=0.35\textwidth]{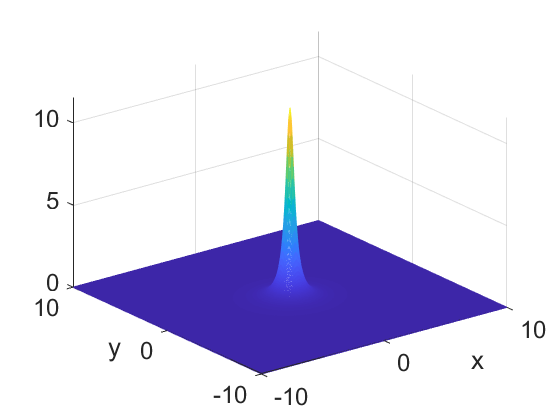}
\includegraphics[width=0.29\textwidth]{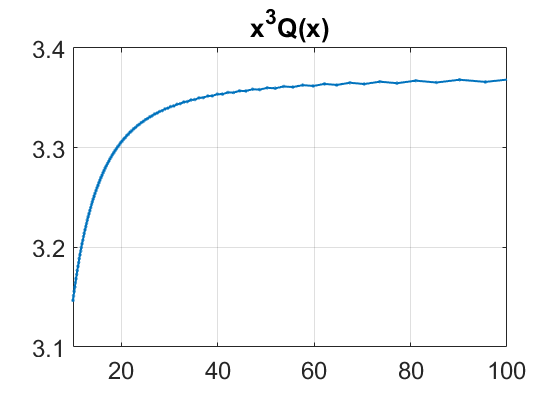}
\includegraphics[width=0.32\textwidth]{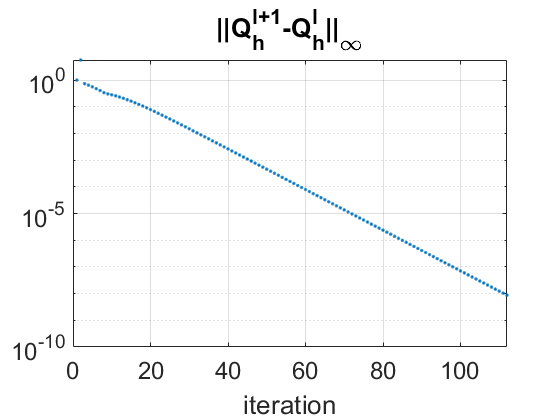}
\caption{Left: solution $Q_c$ of \eqref{E:Qc eqn}, $c=1$. Middle: spatial dependence of $x^3 \, Q(x,0)$. Right: $\| Q_h^{l+1}-Q_h^l \|_{L^{\infty}(\mathbb{R}^2)}$ on the log scale for each iteration. }
\label{F:profile GS Q}
\end{figure}

Figure \ref{F:profile GS Q} shows our numerical solution for $Q_c$ from \eqref{E:Qc eqn} with $c=1$, $m=2$, $s=\frac12$ (we set $\alpha=10$ and $N=512$). It shows a well-localized, radially symmetric, positive function (in agreement with Theorem \ref{existTHR}). To double check its decay, we track, for example, the quantity $x^3 \, Q(x,0)$ (i.e., the decay in one of the cross-sections, by $y=0$) and plot the resulting curve in the middle graph of Figure \ref{F:profile GS Q}. Observe the convergence to the horizontal asymptote as $x$ grows large. We obtain similar results in other cross-sections. This confirms that the solution decays as $1/|x|^3$, stated in \eqref{poldecayGS} of Theorem \ref{existTHR}.

From our numerical simulations, we can see that the difference  $\| Q_h^{l+1}-Q_h^l \|_{L^{\infty}(\mathbb{R}^2)}$ decays exponentially, see the right graph in Figure \ref{F:profile GS Q}. This agrees with theoretical results about the Petviashvili's iteration in \cite{PS2004}, \cite{OSSS2016} and \cite{LY2007}.

To check further the accuracy and consistency of our computation of $Q$, we define the error quantities $e_1$, $e_2$ and $e_3$ from the Pohozaev identities \eqref{Phideneq1.1}, \eqref{Phideneq2.1} and a multiple of energy (in the $L^2$-critical case $E[Q]=0$) \eqref{PhiE} as
\begin{align*}
&e_1= \| (-\Delta)^{\frac{1}{4}}Q_h \|_{L^{2}(\mathbb{R}^2)}^2-2 \| Q_h \|_{L^{2}(\mathbb{R}^2)}^2 ,\\
&e_2= \| Q_h \|_{L^{3}(\mathbb{R}^2)}^3-6 \| Q_h \|_{L^{2}(\mathbb{R}^2)}^2,\\
&e_3= 3\| (-\Delta)^{\frac{1}{4}}Q_h \|_{L^{2}(\mathbb{R}^2)}^2 - \| Q_h \|_{L^{3}(\mathbb{R}^2)}^3.
\end{align*}

Table \ref{T:Q consistency} shows the numerical values for $e_1$, $e_2$ and $e_3$ depending on the mapping parameter $\alpha$ and the number of nodes $N$. We can see that the error decreases as we increase the value of $\alpha$, or, in other words, if we increase the length of the computational domain. On the other hand, increasing the number of nodes $N$ will not decrease the error (compare the second column with the last column in Table \ref{T:Q consistency} for $\alpha=20$ and $N=512$ vs. $1024$).
 
\begin{table}[ht]
\begin{tabular}{ |c|c|c|c|c|c| } 
 \hline
 $N$     & $256$ & $512$ & $1024$& $2048$& $1024$  \\ 
 \hline
 $\alpha$ & $10$ & $20$ & $40$& $80$& $20$  \\ 
 \hline
 $e_1$ &  $0.81473$ & $0.21302$ & $0.058093$& $0.019039$ & $0.21121$  \\ 
 \hline
 $e_2$ & $1.6295$ & $0.42605$ & $0.11619$& $0.038078$ & $0.42684$ \\ 
 \hline
 $e_3$ &  $0.81473$ & $0.21302$ & $0.058093$& $0.019039$ & $0.20679$ \\ 
 \hline
\end{tabular}
\smallskip
\caption{The values of $e_{1,2,3}$ for different values of the mapping parameter $\alpha$ and number of nodes $N$.}
\label{T:Q consistency}
\end{table}

\begin{table}[ht]
\begin{tabular}{ |c|c|c|c|c|c| } 
 \hline
 $N$     & $256$ & $256$ & $512$& $512$& $512$  \\ 
 \hline
 $\alpha$ & $10$ & $20$ & $10$& $20$& $40$  \\ 
 \hline
 $\|Q\|_{L^2}^2$ &  $42.6381$ & $39.1681$ & $42.6406$& $42.7366$ & $39.3294$  \\ 
 \hline
\end{tabular}
\smallskip
\caption{The values of $\|Q\|_{L^2}^2$ for different values of the mapping parameter $\alpha$ and number of nodes $N$.}
\label{T:Q mass}
\end{table} 

For later purposes, we compute the $L^2$-norm of $Q$. Table \ref{T:Q mass} shows how this value depends on the mapping parameter $\alpha$ and the number of nodes $N$.

\section{Numerical solutions of the HBO equation}\label{S:Numerical solution}

In this section we discuss our numerical findings for the dynamical HBO equation \eqref{E:HBO}. We first discuss solutions that exist for all times, then we explore the possibility of finite time blow-up, and finish with investigating the interactions between two solitary waves. We recall that the equation \eqref{E:HBO} is $L^2$-critical, and the Conjecture \ref{C:critical} states that the ground state $Q$ would be a possible threshold for the globally vs. finite time existing solutions (to be more precise, the $L^2$ norm of   $Q$). To investigate that we consider various multiples and translations of $Q$ as well as other types of data with different decay rates, and confirm the conjecture. Furthermore, our analysis shows that the blow-up solutions are self-similar (in its core region) with the profiles of the rescaled ground state solutions. As far as the globally existing solutions we observe that eventually they all disperse into the radiation. Even those solutions, which initially start traveling to the right (in the $x$-direction) and try to approach a rescaled ground state profile, due to the outgoing dispersive oscillatory radiation (in the opposite direction or region): the location of the peak of the solution travels to the right (in the positive $x$-direction), but stops (possibly for some time), and then travels to the left, completely shedding via dispersive oscillations into the radiation. Furthermore, we observe the angle of the radiation wedge as it was discussed in Section \ref{S:radiation}. 
In the interaction of two solitary waves we show different scenarios of interaction, including a strong interaction, where two traveling waves combine into one that  will either radiate away or blow up in finite time, depending on the total combined mass and initial geometrical configuration. 

\subsection{Globally existing solutions}\label{S:global}
We start with considering the initial data of the form 
\begin{equation}\label{E:Qdata}
u_0(x,y) = A \, Q(x,y),
\end{equation}
where $Q(x,y)$ is the solution of \eqref{E:Qc eqn} with $c=1$ (or rather its numerical approximation $Q_h$ obtained in the previous section) and the constant $A>0$. In this part we consider data such that $\|u_0\|_{L^2(\mathbb R^2)} < \|Q\|_{L^2(\mathbb R^2)}$, 
thus, we take $A<1$. For completeness, we mention that $E[u_0] = A^2(1-A) \|Q\|^2_{L^2(\mathbb R^2)}$ by Pohozaev identities. 

\begin{figure}[ht]
\includegraphics[width=0.32\textwidth]{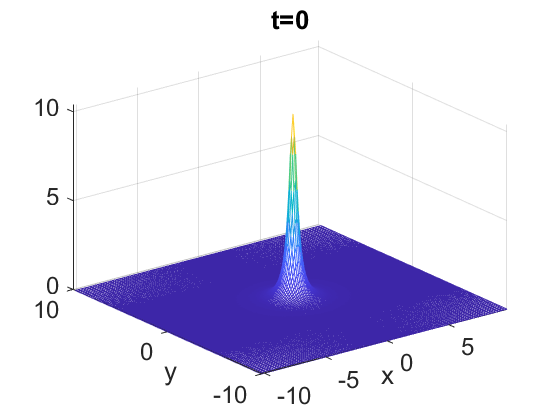}
\includegraphics[width=0.32\textwidth]{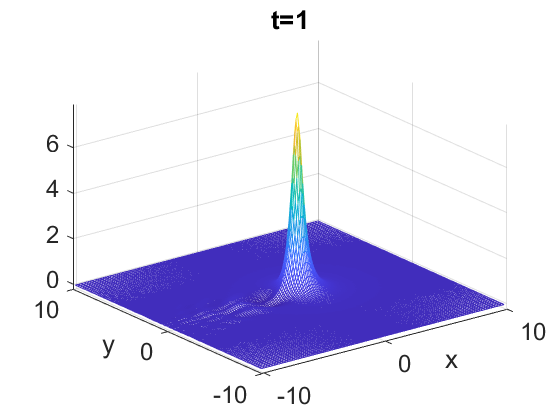}
\includegraphics[width=0.32\textwidth]{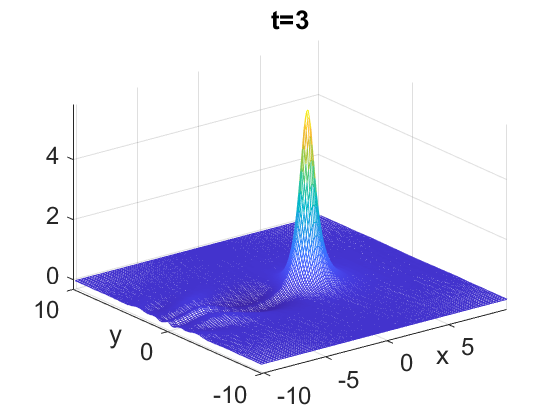}
\includegraphics[width=0.32\textwidth]{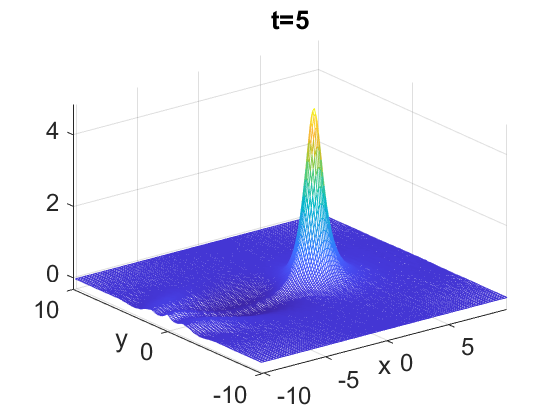}
\includegraphics[width=0.32\textwidth]{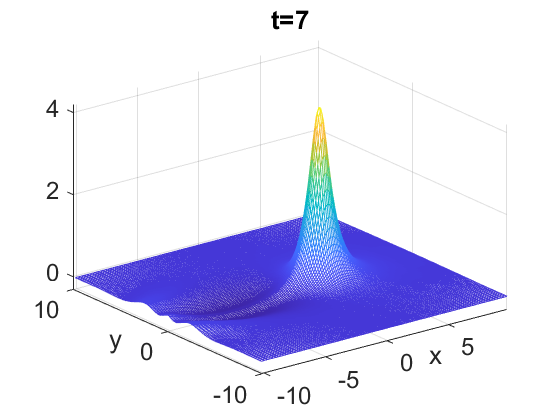}
\includegraphics[width=0.32\textwidth]{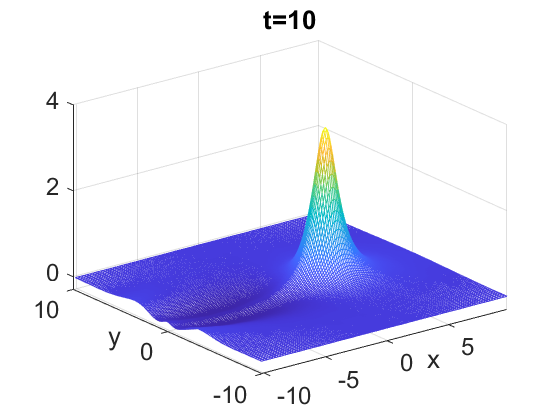}
\caption{Snapshots of the solution $u(t)$ with $u_0=0.9\, Q$.}
\label{F:profile 09Q}
\end{figure}

We first set $A = 0.9$ 
and track the time evolution of $u(t)$ up to $t=10$ (the end of the computational time in this simulation), the snapshots of this solution at times $t=0,1,3,5,7,10$ are given in Figure \ref{F:profile 09Q}. Starting from a radially symmetric initial condition at $t=0$, the main peak travels along the $x$-axis in its positive direction while decreasing in its $L^\infty$ norm (note that the height is decreasing in time in Figure \ref{F:profile 09Q}). The dispersive oscillations start developing right away, which we refer to as the radiation; the oscillations are outgoing in the negative $x$-direction. We note that the solution around the solitary wave core preserves its radial symmetry. For that we plot the cross-sections at the initial time and at the ending time of our simulations ($t=10$).

\begin{figure}[ht]
\includegraphics[width=0.32\textwidth]{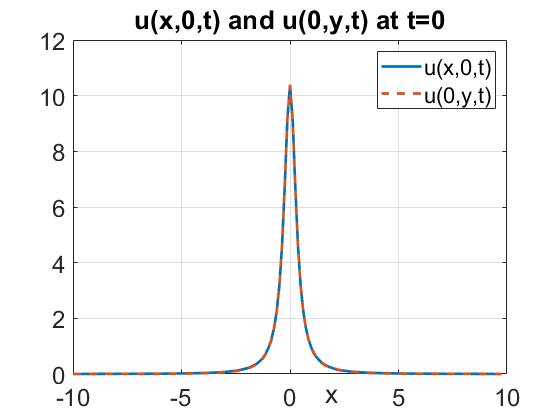}
\includegraphics[width=0.32\textwidth]{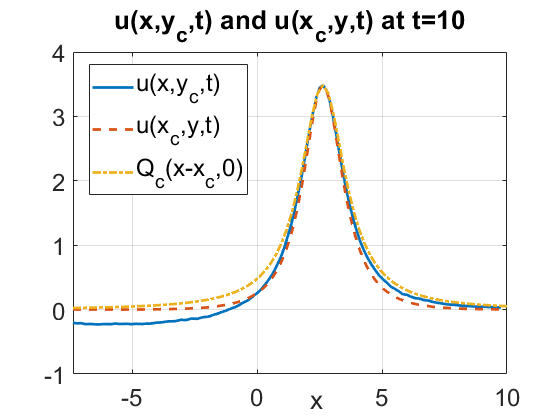}
\includegraphics[width=0.32\textwidth]{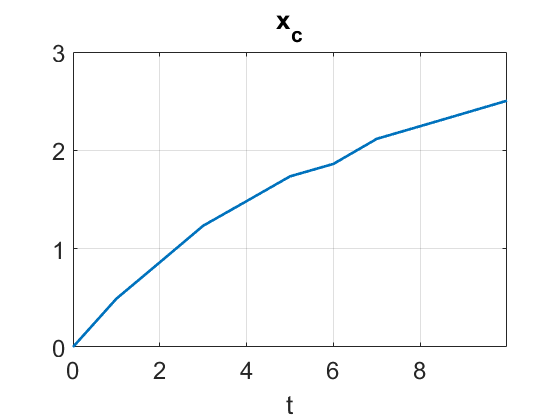}
\caption{Left: initial profile of $u_0=0.9\,Q$ given via cross-sections by $y=0$ and $x=0$ planes. Middle: cross-sections of the solution $u(t)$ at $t=10$ by $y=y_c=0$ (solid blue) and $x=x_c=2.5$ (dash red) planes, 
the rescaled and shifted profile $Q_c$, dotted yellow curve. 
Right: time evolution of the peak location $x_c$ for $0<t<10$.}
\label{F:09Q data}
\end{figure}

The cross-sections by the $y=0$ and $x=0$ planes of the initial profile are given on the left plot of Figure \ref{F:09Q data}, both coincide, since the initial profile is radially symmetric. The middle plot of Figure \ref{F:09Q data} shows both cross-sections by $y_c=0$ (solid blue line) and by 
$x_c = 2.5$ (dashed red line) at the final time of this simulation $t=10$. By $(x_c,y_c)$ we denote the coordinate of the peak of the solution (at a given time), i.e., 
\begin{equation}\label{E:location}
\|u\|_{L^{\infty}(\mathbb{R}^2)}=|u(x_c,y_c)|.
\end{equation} 
Note that in the middle plot the profile from the $x_c=2.5$ cross-section (dash red curve)  is intentionally shifted to the right to show the symmetry of the profile at $t=10$ (otherwise, the peak in this slice (dash red curve) would be at $y=0$). On the same graph we also plot the rescaled and shifted profile of $Q$, that is, $Q_c(x-x_c,0)$ (dotted yellow line) to show that the solution has a good match (in the area excluding the radiation region to the left). 
The parameter $c$ is the scaling parameter as defined in \eqref{E:Qc eqn}-\eqref{E:Qc} (for the case $s=\frac12$, $m=2$)
and we compute it as follows in our simulations
\begin{equation}\label{E:c-parameter}
c = \frac{\|u(t_{max}) \|_{L^\infty}}{\|Q\|_{L^\infty}},
\end{equation}
where $t_{max}$ is the maximal time in our computations.

\begin{figure}[ht]
\includegraphics[width=0.32\textwidth]{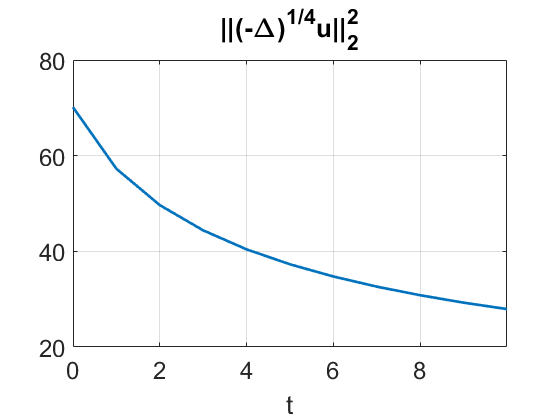}
\includegraphics[width=0.32\textwidth]{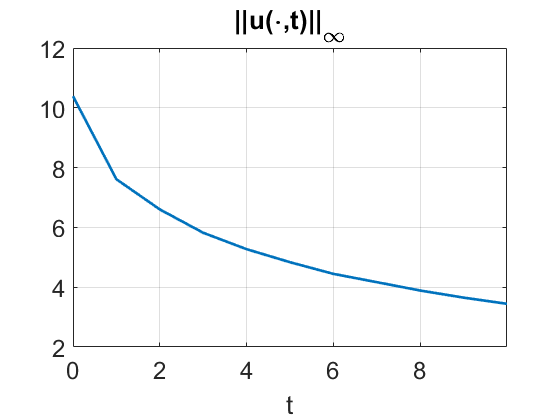}
\includegraphics[width=0.32\textwidth]{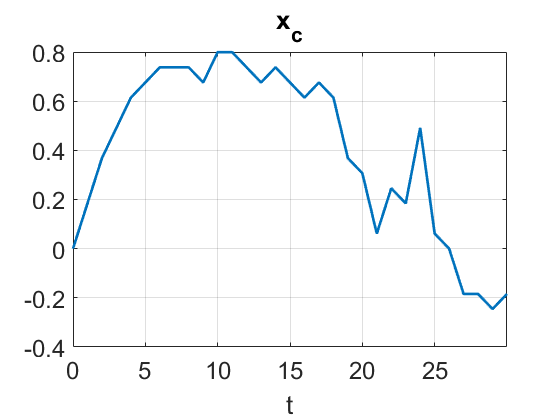}
\caption{Left: time evolution of the kinetic energy for $u_0=0.9Q$. Middle: time dependence of $\|u(t)\|_{L^{\infty}}$ for $u_0=0.9Q$. Right: time evolution of the peak location $x_c$ for $u_0=0.85Q$: observe that the peak stops traveling to the right (in the positive $x$-direction) and then moves in the opposite direction (the solution eventually radiates). A similar behavior is expected for $u_0=0.9Q$. 
} 
\label{F:09Q data 2}
\end{figure}

From the right graph in Figure \ref{F:09Q data} it seems that the solution approaches a rescaled solitary wave that is traveling to the right of the $x$-axis with the decreasing height and decreasing speed shedding some radiation in the negative $x$-direction. It is plausible to suppose that this asymptotic behavior continues (as we showed in (C2) of Theorem \ref{propcriti} that solutions (at least sufficiently smooth) with the mass under the threshold are uniformly $H^s$ bounded globally in time), however, this is not the case. For this specific initial condition $u_0=0.9\, Q$ it is challenging to track reliably the evolution beyond $t_{max}=10$, therefore, we consider slightly smaller initial amplitude $A$ in \eqref{E:Qdata}. We are able to track the time evolution of $u_0 = 0.85\, Q$ (as well as smaller $A$) and on the right graph of Figure \ref{F:09Q data 2} we show the trajectory of $x_c$. 
We note that the peak stops traveling to the right when the location $x_c$ stops around $t=10$ and, after a short pause (the profile at that time has good matching with the solitary wave $Q_c$ as in the middle graph of Figure \ref{F:09Q data}), starts moving to the left (though sometimes moving forward and again backward, this is due to dispersive oscillations that can create double peaks, e.g. see top right of Figure \ref{F:1x2 data}), and then gets dispersed into the radiation. 
\bigskip

We next consider initial data that decays slower than $Q$ (recall that $Q$ decays as $1/|x|^3$) 
\begin{equation}\label{ID:square-decay}
u_0(x,y)=\frac{A}{1+x^2+y^2}, \quad A>0.
\end{equation}
Noting that $ \|u_0\|_{L^2(\mathbb R^2)}^2= A^2 \,\pi$, we obtain the threshold value for $A$, i.e., when $\|u_0\|_{L^2(\mathbb R^2)} = \|Q\|_{L^2(\mathbb R^2)}$, or equivalently,   $A_{th} = \|Q\|_{L^2(\mathbb R^2)}/{\sqrt \pi} \approx 3.7$, 
where we used the value for the norm of $Q$ from Table \ref{T:Q mass}. 

\begin{figure}[ht]
\includegraphics[width=0.32\textwidth]{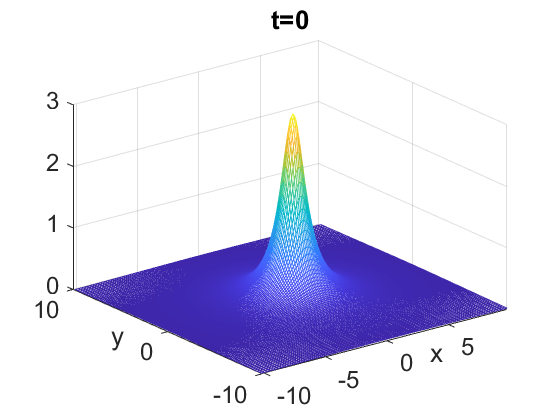}
\includegraphics[width=0.32\textwidth]{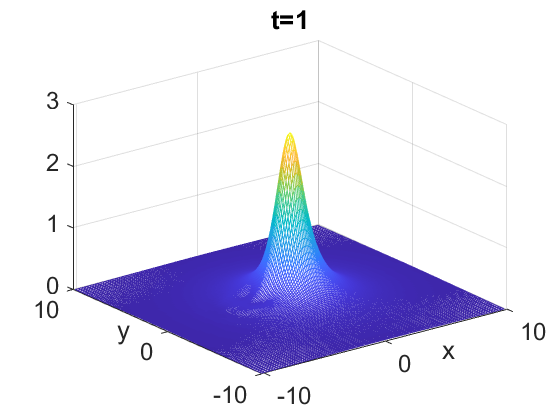}
\includegraphics[width=0.32\textwidth]{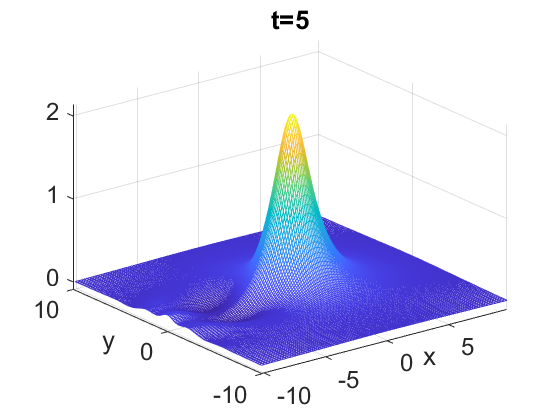}
\includegraphics[width=0.32\textwidth]{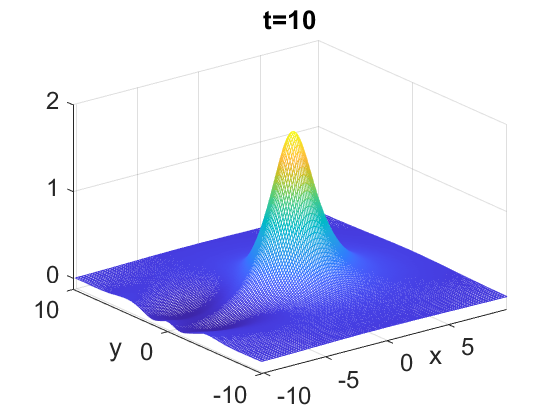}
\includegraphics[width=0.32\textwidth]{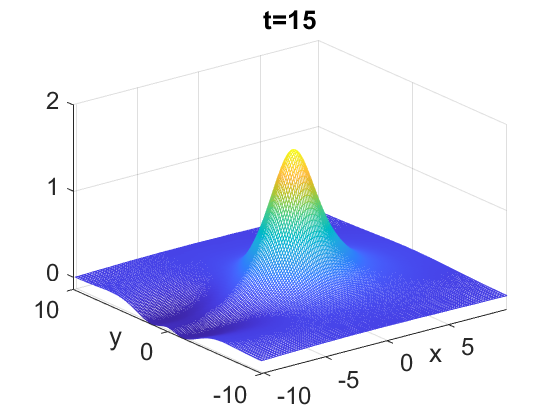}
\includegraphics[width=0.32\textwidth]{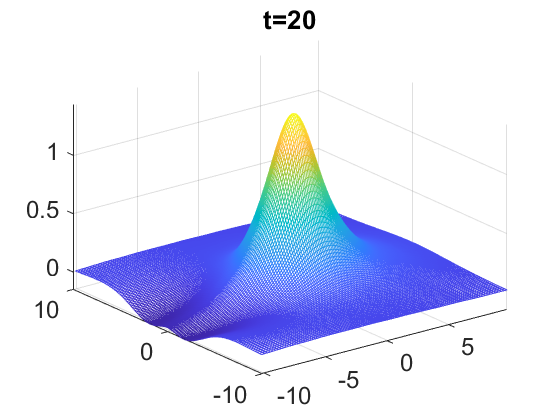}
\includegraphics[width=0.32\textwidth]{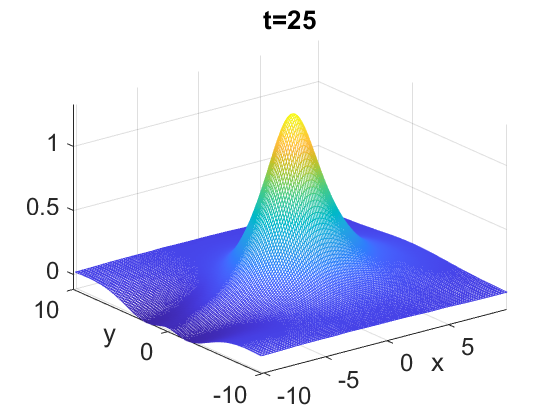}
\includegraphics[width=0.32\textwidth]{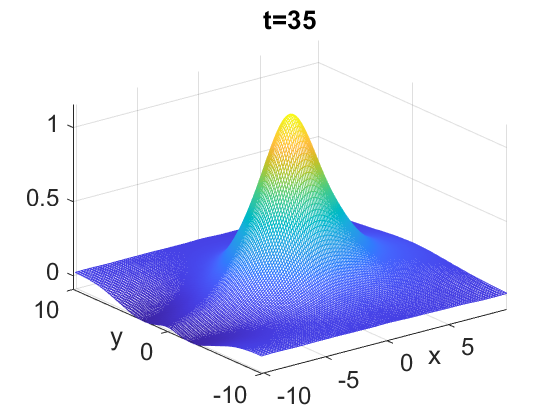}
\includegraphics[width=0.32\textwidth]{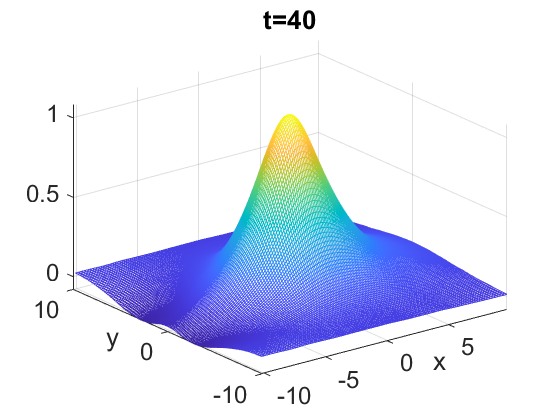}
\caption{Snapshots of the solution $u(t)$ with $u_0=\frac{3}{1+(x^2+y^2)}$.}
\label{F:profile x2}
\end{figure}

We now take $A = 3$, so that $\|u_0\|_{L^2(\mathbb R^2)} < \|Q\|_{L^2(\mathbb R^2)}$ (we also compute $E[u_0] \approx 2.14$), and track its time evolution. Figure \ref{F:profile x2} shows the snapshots of $u(t)$ at times $t=0, 1, 5, 10, 15, 20, 25, 30, 35, 40$. The height is decreasing while the location of the peak is not moving significantly for some time, there is some shift in the positive $x$-direction around $t=20$ (see tracking of $x_c$ in the middle subplot of Figure \ref{F:x2 data 2}), the radiation develops immediately in the negative $x$-direction, and the peak location after $t = 30$ starts moving to the left, or in the negative $x$-direction. On the left of Figure \ref{F:x2 data} the plot shows that the $L^\infty$ norm of the solution decreases in time. 

\begin{figure}[ht]
\includegraphics[width=0.32\textwidth]{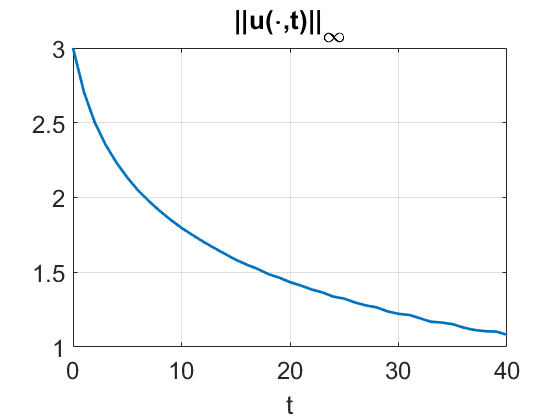}
\includegraphics[width=0.32\textwidth]{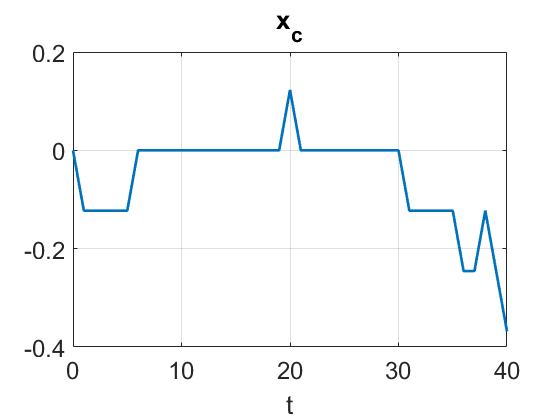}
\includegraphics[width=0.32\textwidth]{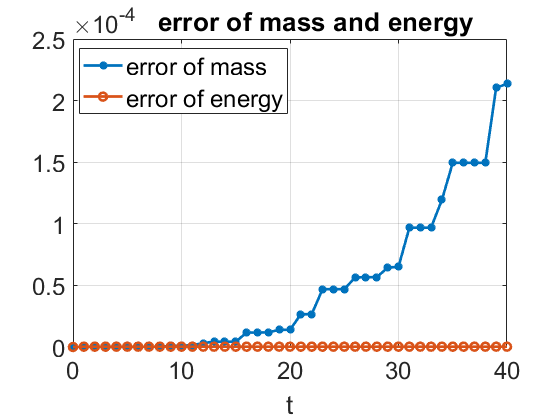}
\caption{Evolution of $u_0=\frac{3}{1+x^2+y^2}$: time dependence of $\|u(t)\|_{L^{\infty}}$ (left), trajectory of $x_c$ in time (middle), errors of conserved quantities (right).}
\label{F:x2 data 2}
\end{figure}

\begin{figure}[ht]
\includegraphics[width=0.32\textwidth]{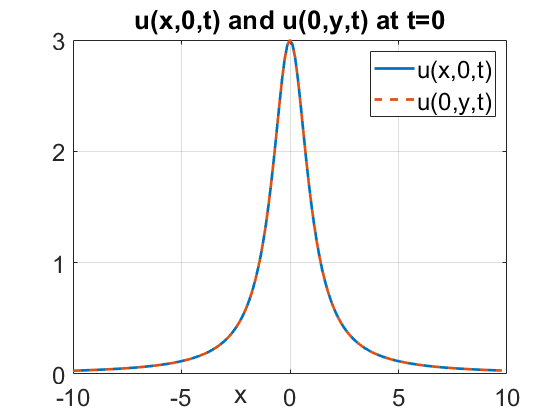}
\includegraphics[width=0.32\textwidth]{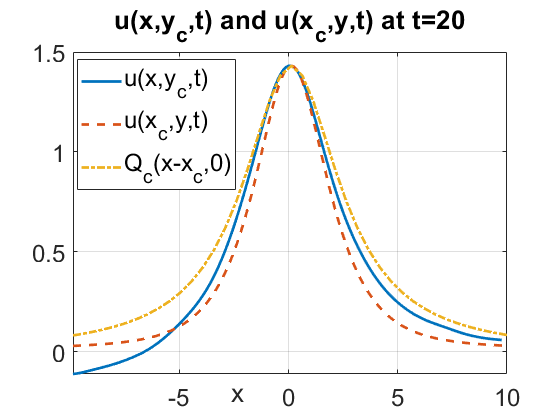}
\includegraphics[width=0.32\textwidth]{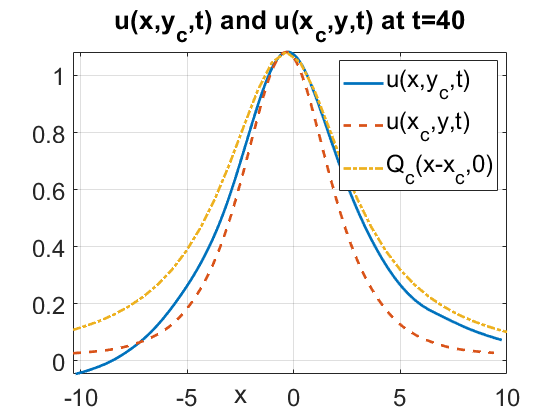}
\caption{Cross-sections of the solution $u(t)$ with $u_0=\frac{3}{1+x^2+y^2}$ at different times. Left: $t=0$, cross-sections by $y=0$ and $x=0$ planes. Middle: $t=20$, cross-sections by $y=y_c=0$ (solid blue) and $x=x_c =0.1$ (dashed red), compared with $Q_c$ (dotted yellow). Right: $t=40$, cross-sections by $y=y_c=0$ (solid blue) and $x=x_c = -0.35$ (dashed red), compared with $Q_{\tilde c}$ (dotted yellow). Note much tighter fit to $Q$ in the middle graph when the peak was traveling to the right. }
\label{F:x2 data}
\end{figure}

To understand better what happens with this solution, we check the cross-sections at various times and plot them in Figure \ref{F:x2 data}. On the left graph the radial symmetry of the initial data is obvious; in the middle at $t=20$ the peak has moved to $x_c=0.1$ and one can see some resemblance of both cross-sections to the $Q_c$ profile; on the right the peak location has moved to left to $x_c=-0.35$ (at the final computational time $t=40$) and the fitting to the $Q_c$ is much less than in the middle graph for both cross-sections. This indicates that when the peak is traveling to the right, it is trying to approach the rescaled ground state profile, while when the peak of the solution is traveling to the left, it goes into the dispersive oscillatory behavior and, of course, no profile matching is expected. 
\medskip



\begin{figure}[ht]
\includegraphics[width=0.32\textwidth]{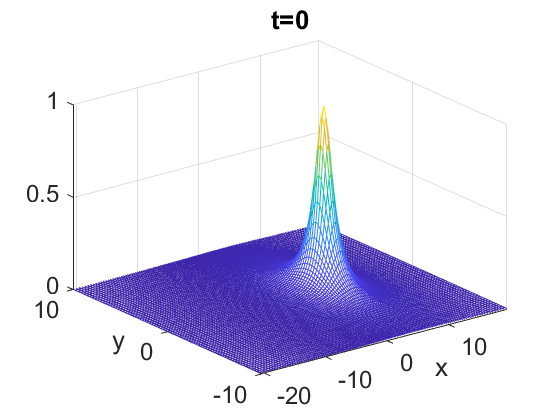}
\includegraphics[width=0.32\textwidth]{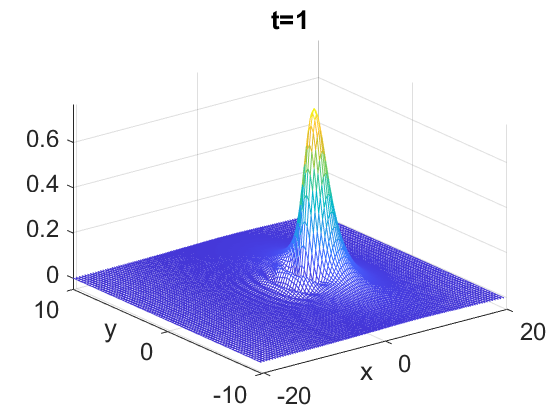}
\includegraphics[width=0.32\textwidth]{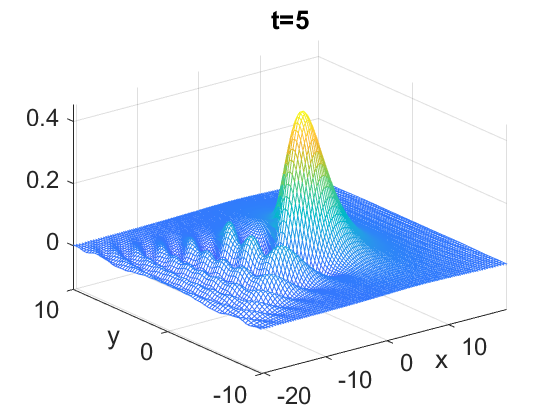}
\includegraphics[width=0.32\textwidth]{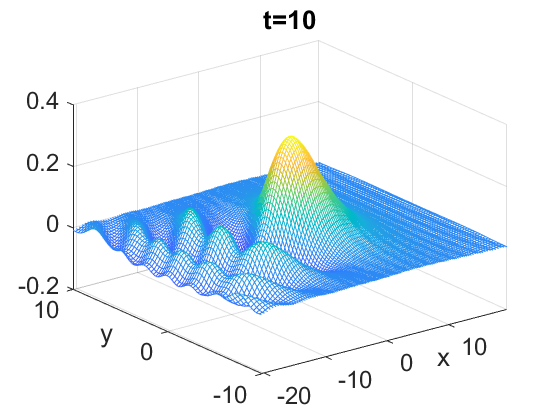}
\includegraphics[width=0.32\textwidth]{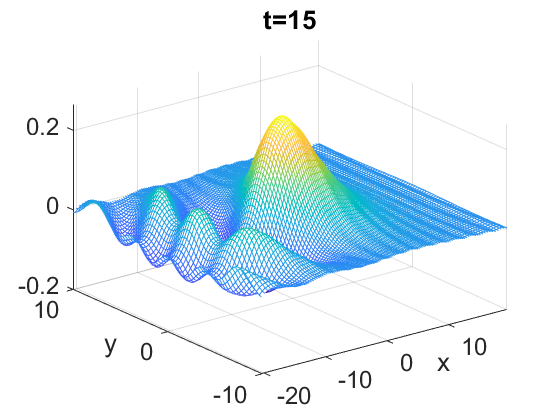}
\includegraphics[width=0.32\textwidth]{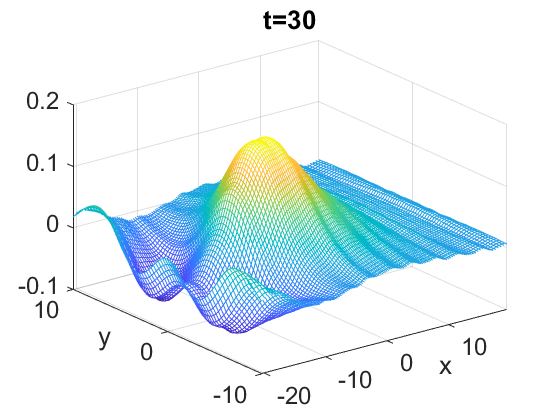}
\caption{Snapshots of the solution $u(t)$ with $u_0=\frac{1}{1+((x-5)^2+y^2)}$.}
\label{F:profile 1x2}
\end{figure}

We also investigate smaller amplitude data 
\begin{equation}\label{ID:square-decay-small}
u_0(x,y)=\frac{A}{1+(x-a)^2+y^2},
\end{equation}
with $A=1$ and a shift $a=5$ (for the computational domain purposes). The snapshots of the time evolution of this $u_0$ are provided in Figure \ref{F:profile 1x2}. 

\begin{figure}[ht]
\includegraphics[width=0.35\textwidth]{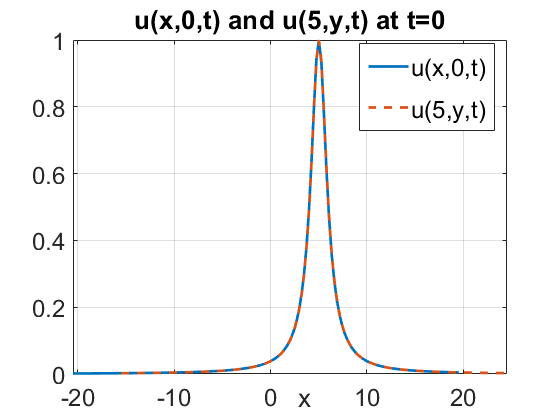}
\includegraphics[width=0.35\textwidth]{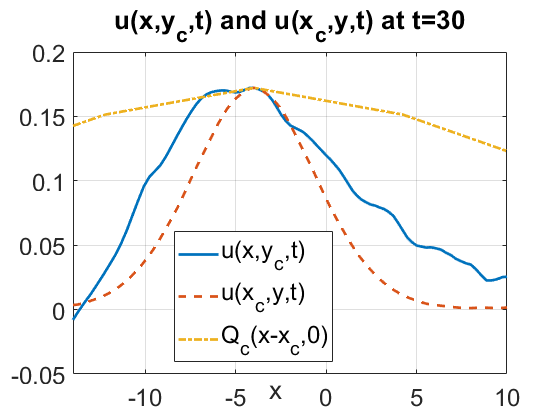}\\
\includegraphics[width=0.32\textwidth]{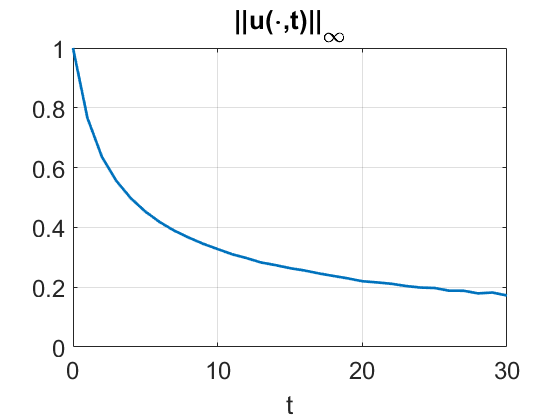}
\includegraphics[width=0.32\textwidth]{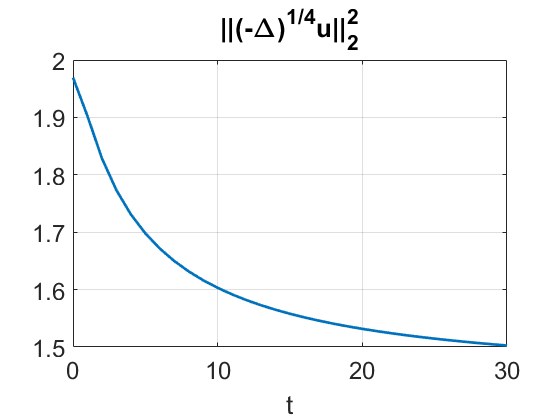}
\includegraphics[width=0.32\textwidth]{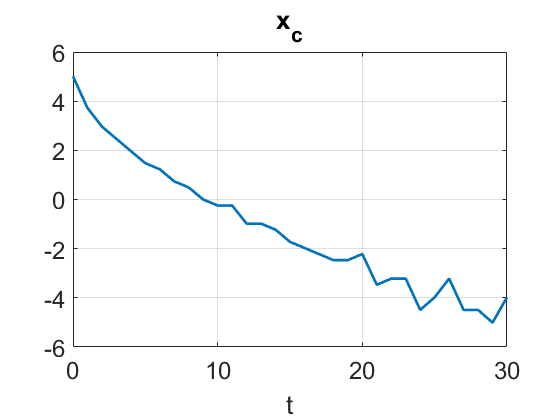}
\caption{Top left: initial profile of $u_0=\frac{1}{1+(x-5)^2+y^2}$ given via cross-sections by planes $y=0$ and $x=5$. Top right: cross-sections of the solution $u(t)$ at $t=30$ by $y=y_c=0$ (solid blue) and $x=x_c =-4$ (dashed red),
here, we shifted the second cross-section to $x_c$ to check symmetry and compare with $Q_c$ (dotted yellow). Bottom: time dependence of $\|u(t)\|_{L^{\infty}}$, $\|(-\Delta)^{1/4}u (t)\|^2_{L^2(\mathbb R^2)}$, and $x_c$ trajectory.}
\label{F:1x2 data}
\end{figure}

Unlike the  previous example (with a larger amplitude $A=3$), the peak of the solution moves in the left $x$-direction right away (see the graph of the $x_c$ trajectory in the bottom right of Figure \ref{F:1x2 data}), meaning that both the peak of the solution moves to the left and the oscillatory radiation develops immediately and disperses to the left.  
\smallskip

Furthermore, this is a good example to note the shape of the radiation region: the dispersive oscillations extend into a wedge region (for example, it can be clearly seen in the top right plot of Figure \ref{F:profile 1x2}). We investigate this more carefully in subsection \ref{S:wedge} below.
\bigskip

\newpage

Another type of data we consider is the one which has a faster decay than the ground state, that is, an exponential decay, 
\begin{equation}\label{ID:gaussian}
u_0(x,y) = A \, e^{-(x^2+y^2)}.
\end{equation} 
We consider $A < A_{th} \approx 5$, since $\|u_0\|^2_{L^2(\mathbb R^2)} = \frac{\pi}2 \,A^2$.

\begin{figure}[ht]
\includegraphics[width=0.32\textwidth]{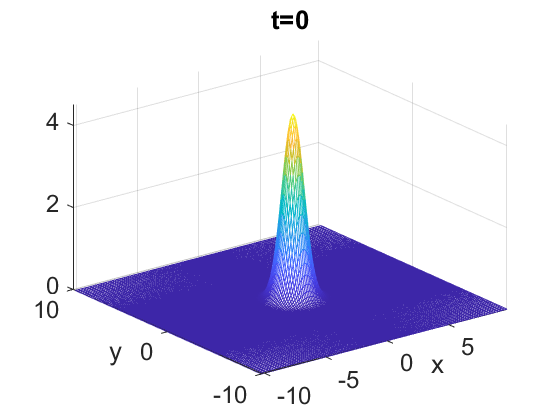}
\includegraphics[width=0.32\textwidth]{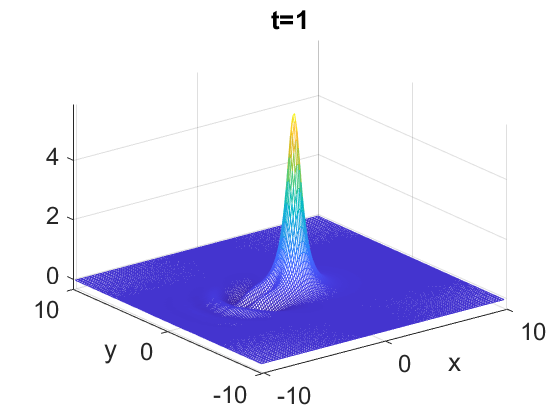}
\includegraphics[width=0.32\textwidth]{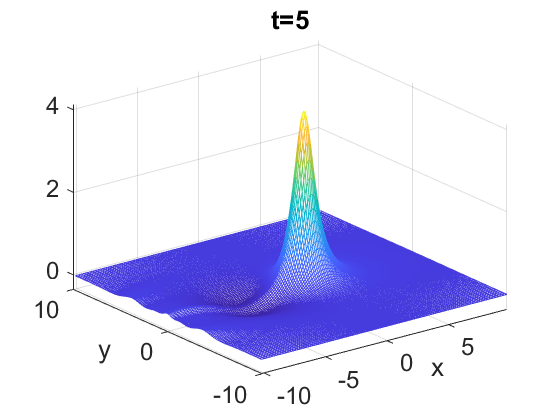}
\includegraphics[width=0.32\textwidth]{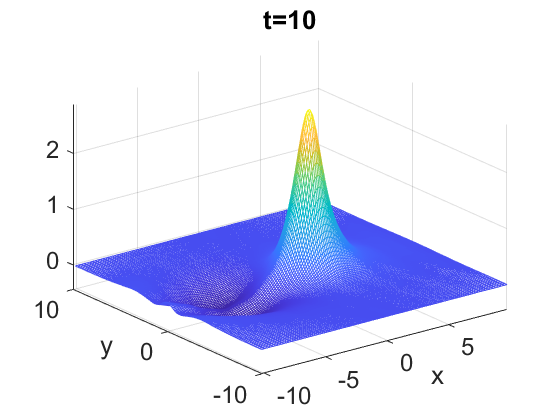}
\includegraphics[width=0.32\textwidth]{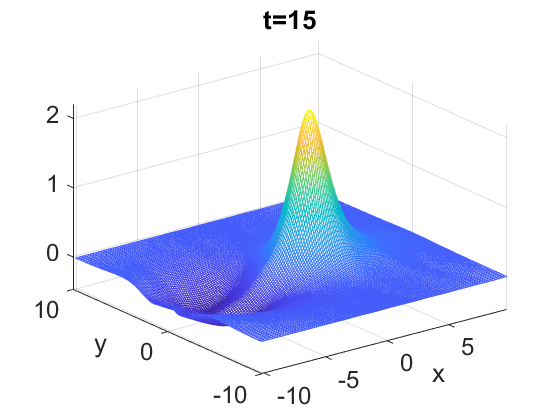}
\includegraphics[width=0.32\textwidth]{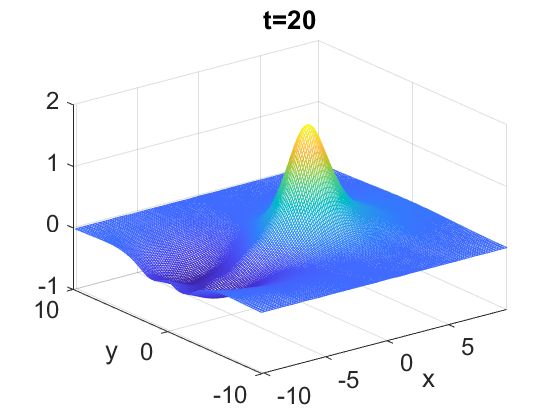}
\includegraphics[width=0.32\textwidth]{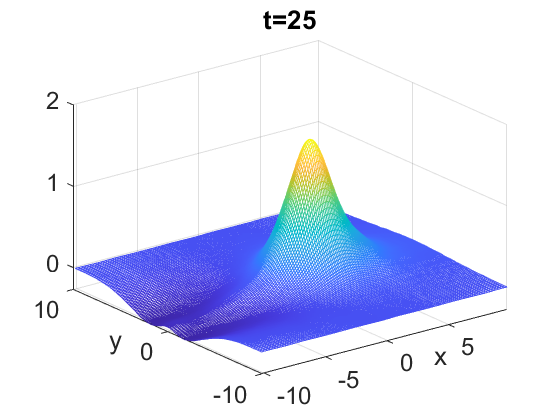}
\includegraphics[width=0.32\textwidth]{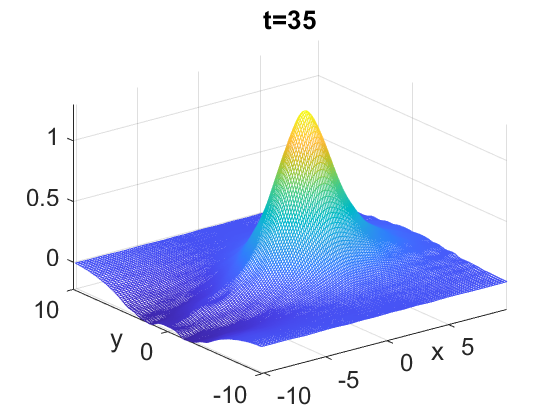}
\includegraphics[width=0.32\textwidth]{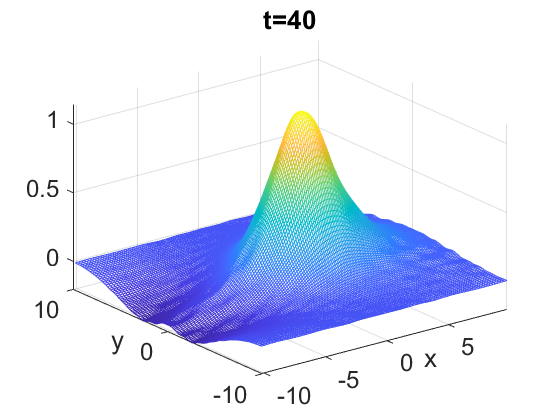}
\caption{Snapshots of the solution $u(t)$ with $u_0={4.5} \,e^{-(x^2+y^2)}$.}
\label{F:profile G}
\end{figure}

We show the snapshots of the solution $u(t)$ with $A=4.5$ (here, $E[u_0] \approx 4.23$) in Figure \ref{F:profile G}.
One can easily notice that the solution starts moving to the right dispersing the radiative oscillations to the left on the $x$-axis. The height is decreasing in time, this can be seen in the snapshots and also on the left graph of Figure \ref{F:G data 2}. 
 
\begin{figure}[ht]
\includegraphics[width=0.32\textwidth]{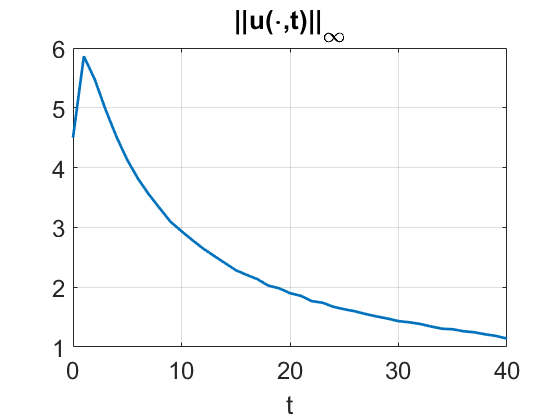}
\includegraphics[width=0.32\textwidth]{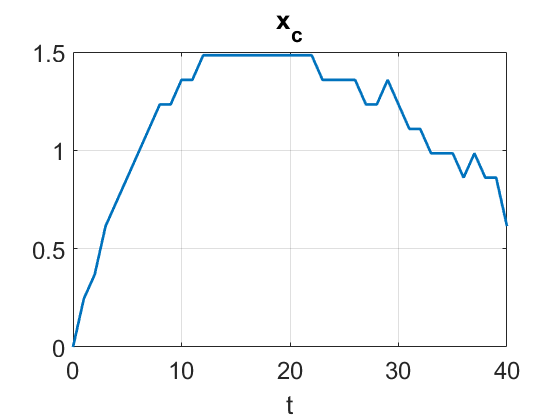}
\includegraphics[width=0.32\textwidth]{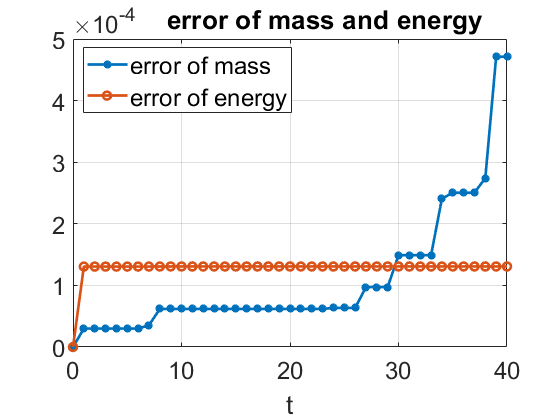}
\caption{Evolution of $u_0={4.5} \,e^{-(x^2+y^2)}$: time dependence of $\|u(t)\|_{L^{\infty}}$ (left), 
trajectory of $x_c$ in time (middle), errors of conserved quantities (right).}
\label{F:G data 2}
\end{figure}

\begin{figure}[ht]
\includegraphics[width=0.32\textwidth]{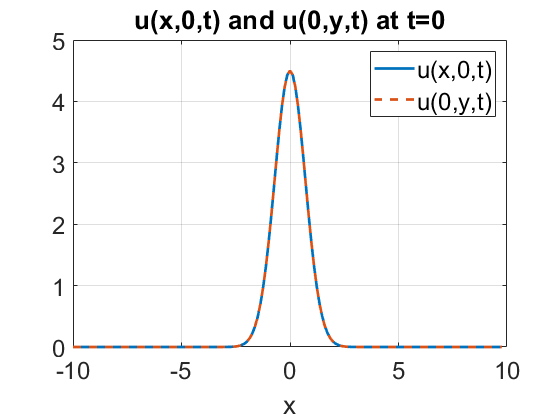}
\includegraphics[width=0.32\textwidth]{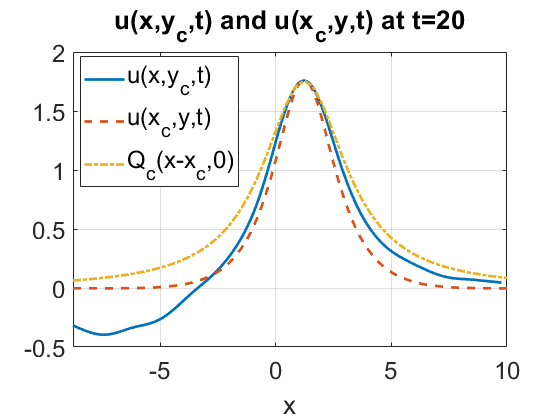}
\includegraphics[width=0.32\textwidth]{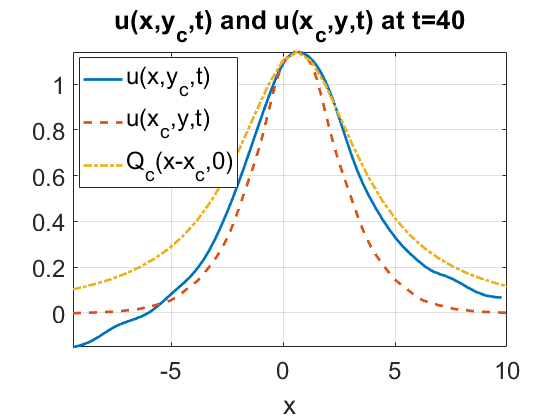}
\caption{The solution profile for $u_0=4.5 \, e^{-(x^2+y^2)}$. We can see it scatters in a radial symmetric manner and $\|u\|_{L^{\infty}(\mathbb{R}^2)}$ is also decreasing in time.}
\label{F:G data}
\end{figure}

In the middle plot of Figure \ref{F:G data} we track the location of the peak $x_c$. 
One can note that up to about time $t=12$ the solution moves to the right, though the peak's location $x_c$ stops for some time around $x_c = 1.5$ up to $t = 22$, and then starts moving in the negative direction. We continue our simulations until $t=40$ (note that the error of the energy is stable, but the error in the mass conservation is starting to increase after $t=30$, therefore, for accuracy we stop our simulations at $t=40$). For comparison we plot the cross-sections at times $t=0, 20, 40$ in Figure \ref{F:G data}, observing some tightness and closeness to $Q_c$ (around the peak location) up to about time $t=20$ and then getting further away from the ground state profile and becoming asymmetric, especially in the cross-section by $y=0$ (even around the peak location). 
\bigskip

\newpage
.

\newpage

We next check the non-radial initial data of the form
\begin{equation}\label{ID:nonradial}
u_0(x,y)=\frac{A}{1+\big(x^2+(0.5y)^2\big)^2}. 
\end{equation}
We note that $\|u_0\|_{L^2(\mathbb R^2)} =  \frac{\pi}{\sqrt 2} \, A$, 
therefore, to check our conjecture, we consider $A < A_{th}$
$\approx 2.9$.

\begin{figure}[ht]
\includegraphics[width=0.32\textwidth]{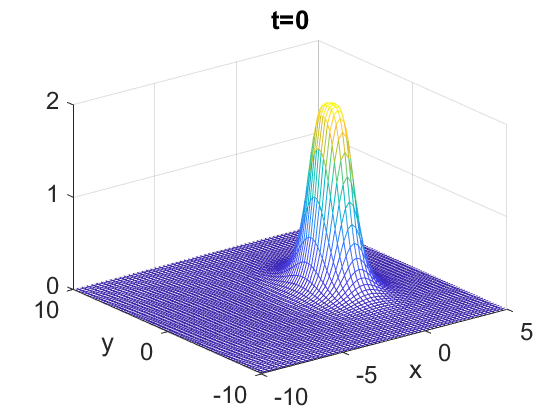}
\includegraphics[width=0.32\textwidth]{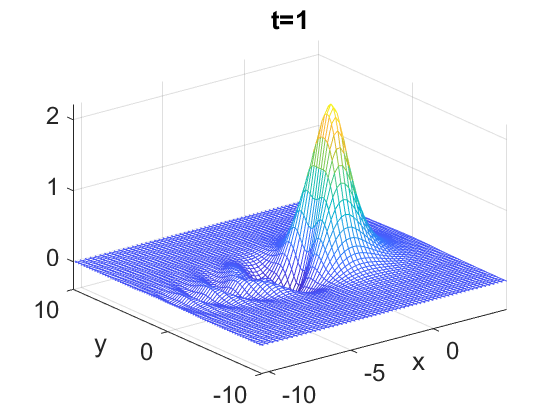}
\includegraphics[width=0.32\textwidth]{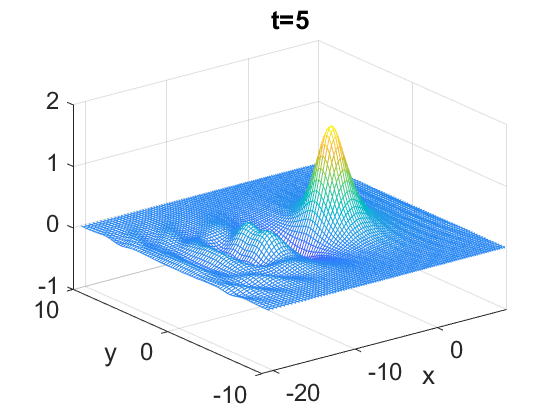}
\includegraphics[width=0.32\textwidth]{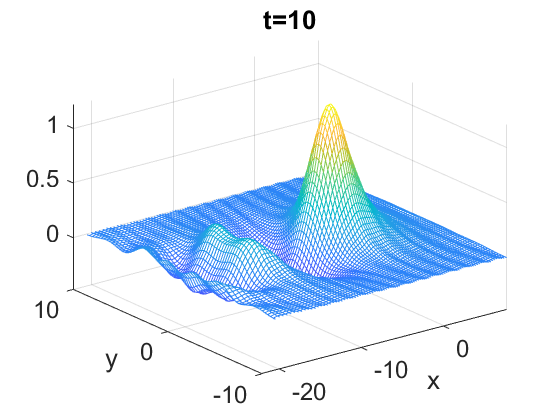}
\includegraphics[width=0.32\textwidth]{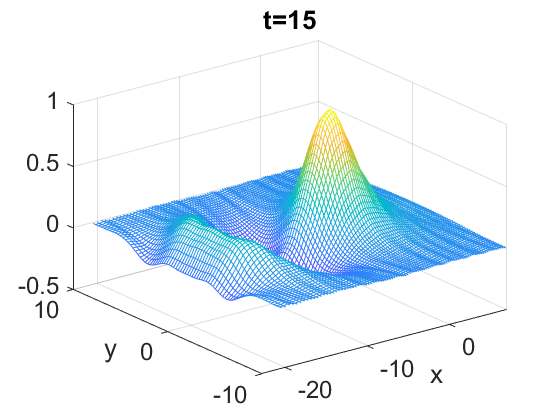}
\includegraphics[width=0.32\textwidth]{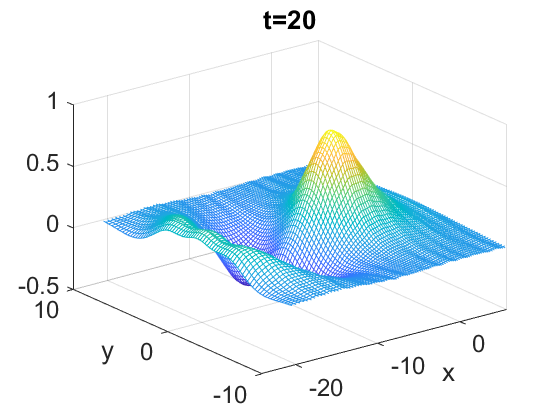}
\caption{Snapshots of the solution $u(t)$ with $u_0=\frac{A}{1+(x^2+(0.5y)^2)^2}$, $A=2$.}
\label{F:profile x42}
\end{figure}

The snapshots of the time evolution at times $t = 0, 1, 5, 10, 15, 20$ for the initial condition \eqref{ID:nonradial} with $A=2$ are shown in Figure \ref{F:profile x42}. One can note that the solution decreases in the height, and comparing the snapshots, it is possible to notice that the peak of this solution is moving in the negative $x$-direction. We show  in the right graph of Figure \ref{F:profile x42 data} that $\|u(t)\|_{L^{\infty}(\mathbb{R}^2)}$ is decreasing in time (after a small increase initially). The left and middle graphs show the solution profile sliced in $x$ and $y$ directions to see how the non-radial data evolves in time and becomes more radially symmetry (compare the solid blue and dashed red curves in the middle plot of Figure \ref{F:profile x42 data}, which are much closer to each other around the core of the profile, indicating more radially symmetric evolution, although being further away from the rescaled $Q$); one notices that the dispersive oscillations appear right away and radiate in a wedge around the negative $x$-axis. 

\begin{figure}[ht]
\includegraphics[width=0.32\textwidth]{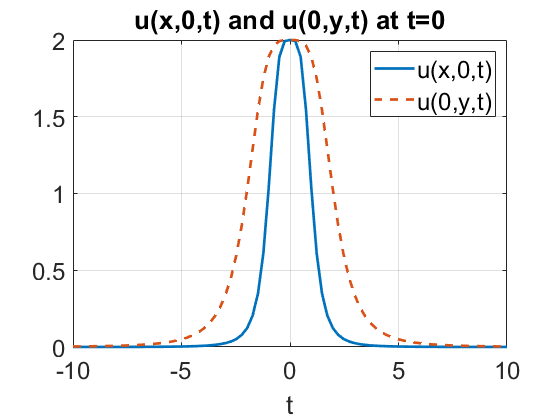}
\includegraphics[width=0.32\textwidth]{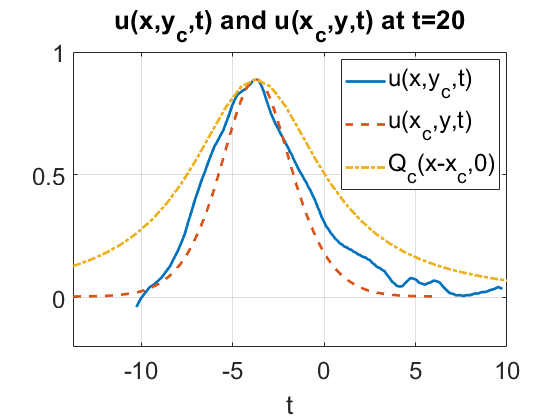}
\includegraphics[width=0.32\textwidth]{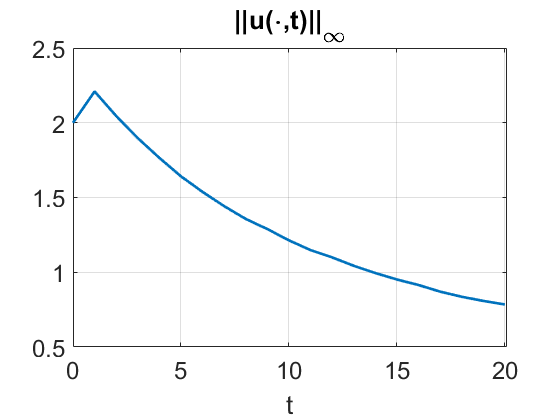}
\caption{$u_0=\frac{2}{1+(x^2+(0.5y)^2)^2}$. Left and Middle: $u(x,y_0,t)$ and $u(x_0,y,t)$ at different time $t$, where $(x_0,y_0)$ is the coordinate for $\max |u(x,y)|$. One can see it scatters to the radially symmetric profile. Right: $\|u\|_{L^{\infty}(\mathbb{R}^2)}$ decreases in time $t$.}
\label{F:profile x42 data}
\end{figure}


\subsection{Angle of radiation wedge}\label{S:wedge}

\begin{figure}[ht]
\includegraphics[width=0.32\textwidth]{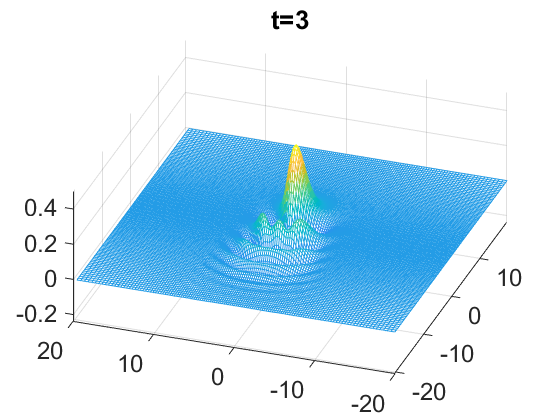}
\includegraphics[width=0.32\textwidth]{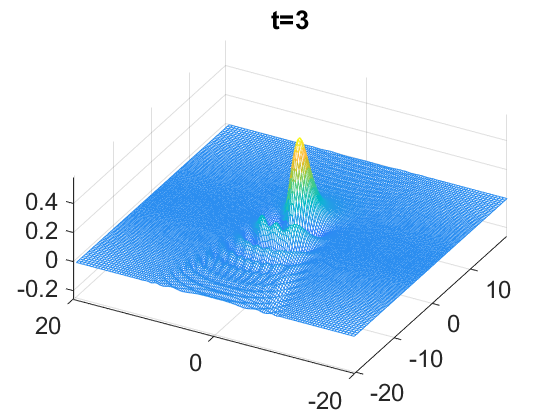}
\includegraphics[width=0.32\textwidth]{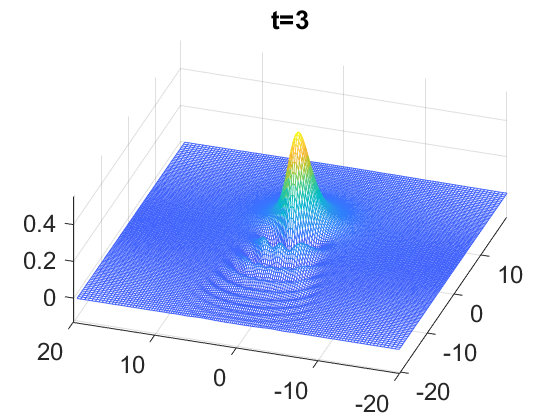}
\includegraphics[width=0.32\textwidth]{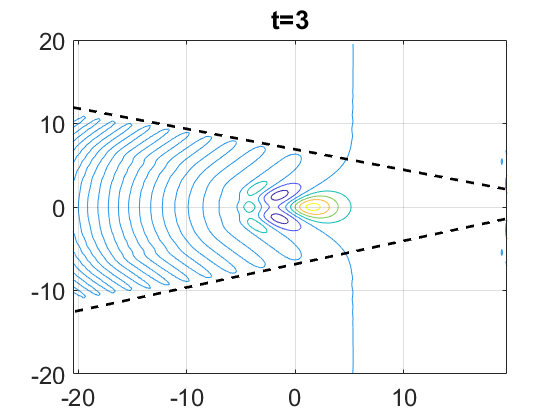}
\includegraphics[width=0.32\textwidth]{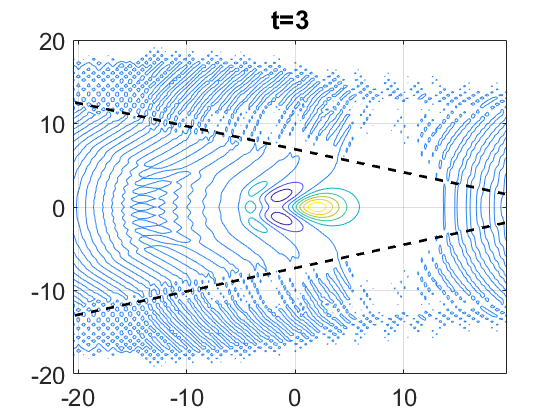}
\includegraphics[width=0.32\textwidth]{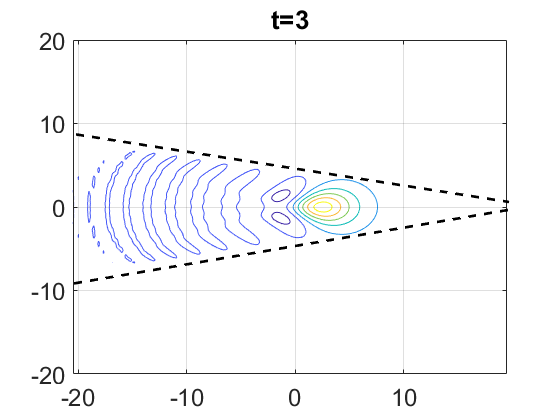}
\caption{$t=3$: radiation region snapshots for different data (top); contour plots with an angle estimation of the radiation wedge (black dash lines) (bottom). 
Left: $u_0=e^{-(x^2+y^2)}$. 
Middle: $u_0=\frac1{1+(x^2+y^2)^2}$.  
Right: $u_0=\frac1{1+x^2+y^2}$. 
}
\label{F:angle t3}
\end{figure}
\begin{figure}[ht]
\includegraphics[width=0.32\textwidth]{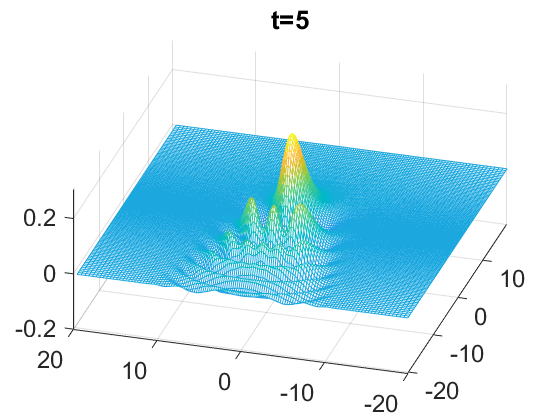}
\includegraphics[width=0.32\textwidth]{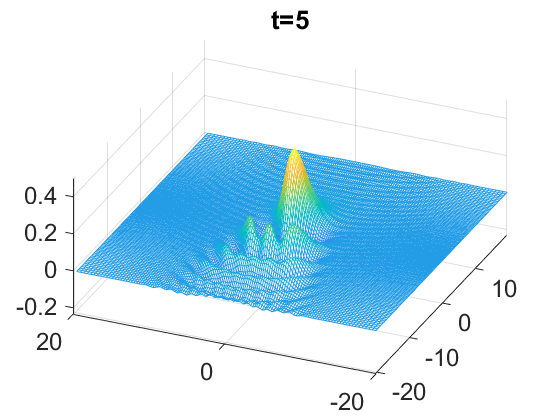}
\includegraphics[width=0.32\textwidth]{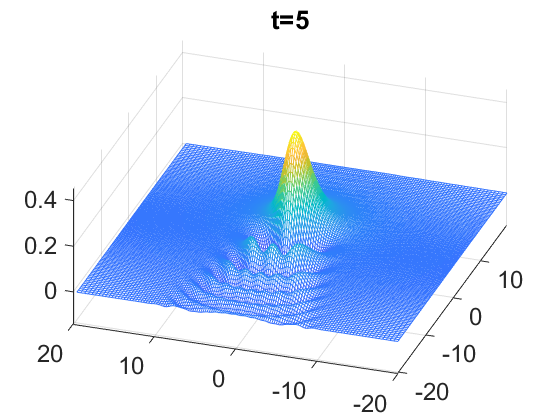}
\includegraphics[width=0.32\textwidth]{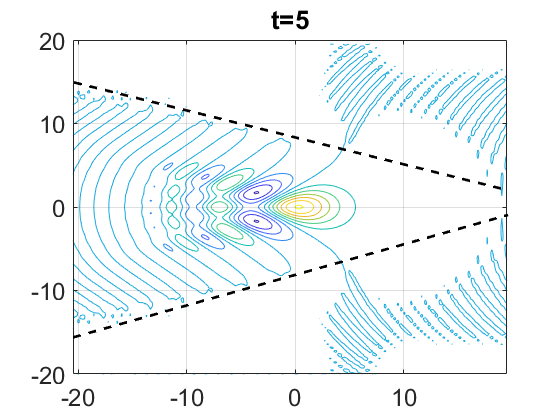}
\includegraphics[width=0.32\textwidth]{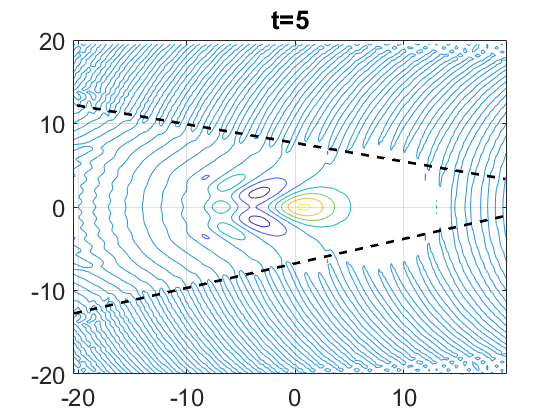}
\includegraphics[width=0.32\textwidth]{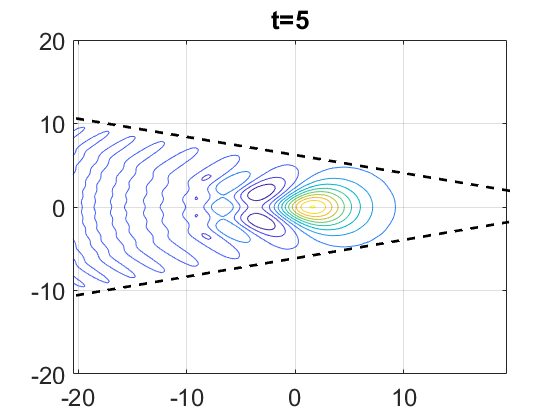}
\caption{$t=5$: radiation region snapshots for different data (top); contour plots with an angle estimation of the radiation wedge (black dash lines) (bottom). 
Left:  $u_0=e^{-(x^2+y^2)}$. 
Middle: $u_0=\frac1{1+(x^2+y^2)^2}$.  
Right: $u_0=\frac1{1+x^2+y^2}$. 
}
\label{F:angle t5}
\end{figure}

We investigate the radiative region of solutions, in particular, the angle of the wedge that was obtained in Section \ref{S:radiation}. For that we consider the following initial data 
\begin{equation}\label{E:three-data}
u_0(x,y)=e^{-(x^2+y^2)}, \qquad u_0(x,y)=\frac1{1+(x^2+y^2)^2}, \qquad  u_0(x,y)=\frac1{1+(x^2+y^2)}.
\end{equation}
Figure \ref{F:angle t3} and \ref{F:angle t5} show the solutions profiles for each of the above initial condition at the times $t=3$ and $t=5$. The top row  in both figures shows a snapshot of the solution at either $t=3$ or $t=5$ of the data \eqref{E:three-data} 
in the left, middle and right columns, respectively. The bottom row offers the contour views and shows an estimate for the angle of radiation wedge with 
the black-dash lines. 

For the simplicity of interpretation, we simply measure the tangent of the angle. One can see that for $u_0=e^{-(x^2+y^2)}$, the dispersive oscillations are restricted to the angle $\theta$ with a crude estimate of $\theta \approx \arctan(13/40) \approx 18.00^{\circ}$. For $u_0=\frac1{1+(x^2+y^2)^2}$, the angle is $\theta \approx \arctan(11/40) \approx 15.38^{\circ}$. For $u_0=\frac1{1+(x^2+y^2)}$, which is the slowest decaying among the considered initial data, the angle is $\theta \approx \arctan(10/40) \approx 14.04^{\circ}$. 
Comparing these observations with the wedge in Figure \ref{F:angle} and in \eqref{E:angleHBO}  in Section \ref{S:radiation}, we observe that our angle approximations  
lie within the angle $19.42^o$ (or $\tan \theta = 2\sqrt 2$), thus, confirming the result of Section \ref{S:radiation}. (Note that the wedge is traveling with the solitary wave in time, and therefore, is shifting with the solution to the right, or in the positive $x$-direction.)

\subsection{Blow-up solutions}\label{Blow-up}
Here, we study the second part of Conjecture \ref{C:critical}, a possibility to develop a finite time blow-up in the case when $\|u_0\|_{L^{2}(\mathbb{R}^2)} > \|Q\|_{L^{2}(\mathbb{R}^2)}$. (We remark that numerically it is impossible to study exactly the threshold case $\|u_0\|_{L^{2}(\mathbb{R}^2)} = \|Q\|_{L^{2}(\mathbb{R}^2)}$ and we have already demonstrated that various data with $\|u_0\|_{L^{2}(\mathbb{R}^2)} < \|Q\|_{L^{2}(\mathbb{R}^2)}$ generate global in time solutions.) 
\smallskip

\begin{figure}[ht]
\includegraphics[width=0.32\textwidth]{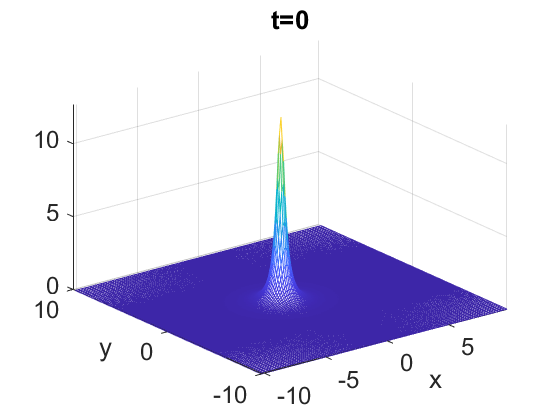}
\includegraphics[width=0.32\textwidth]{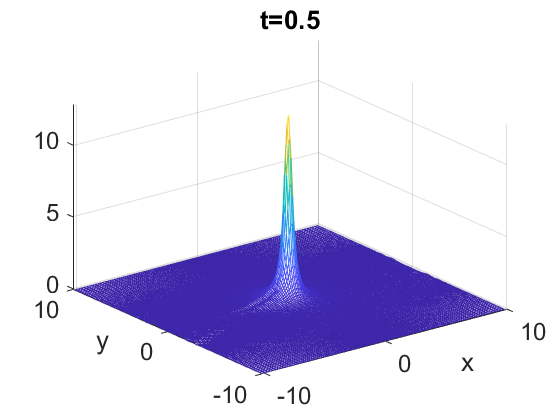}
\includegraphics[width=0.32\textwidth]{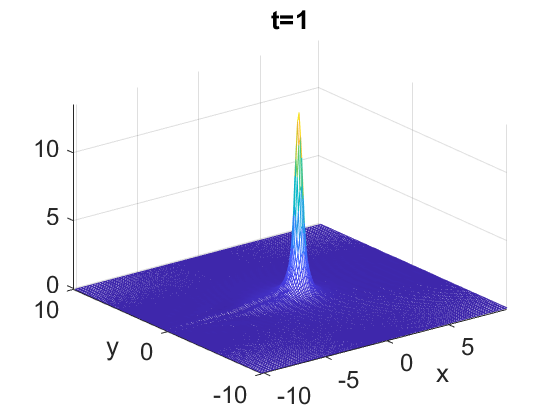}
\includegraphics[width=0.32\textwidth]{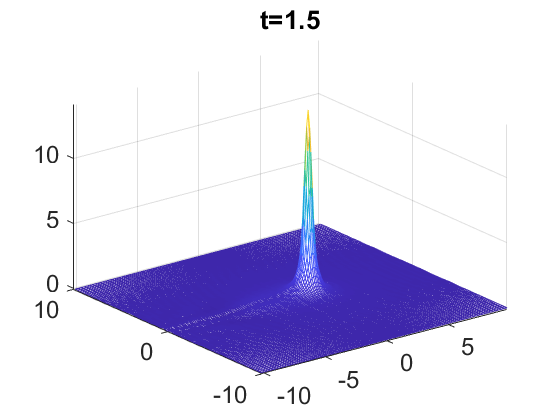}
\includegraphics[width=0.32\textwidth]{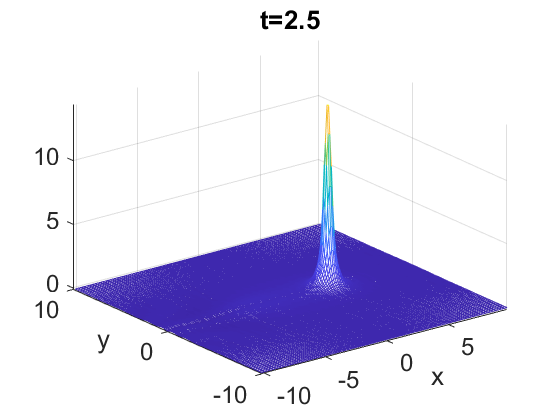}
\includegraphics[width=0.32\textwidth]{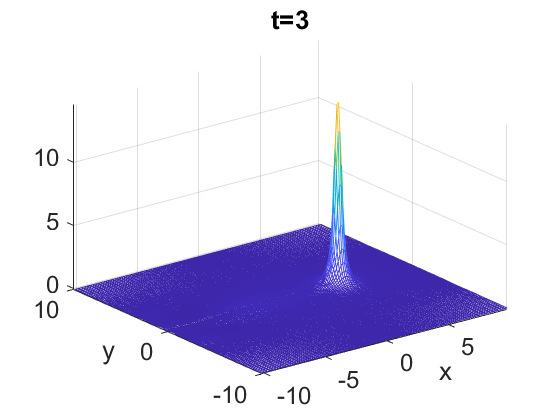}
\caption{Snapshots of the solution $u(t)$ with $u_0=1.1Q(x+1,y)$.}
\label{F:profile 11Q B}
\end{figure}

We start with the initial data of the perturbed ground state \eqref{E:Qdata}, noting that for any $A>1$, the energy of such data is negative, $E[AQ]<0$. 
Fixing $A = 1.1$, i.e., $u_0(x,y)=1.1 \, Q(x,y)$, we compute the time evolution $u(t)$ and plot the details in Figures \ref{F:profile 11Q B}-\ref{F:11Q data B}. For the purposes of staying within the (symmetrical) computational domain, we shift this initial condition in the negative $x$-direction. Figure \ref{F:profile 11Q B} shows snapshots of the time evolution for $u_0(x,y)=1.1\, Q(x+1,y)$ at $t=0, 0.5, 1, 1.5, 2.5, 3$. Observe that the solution becomes tighter around its peak and the height is slowly increasing in time (see also Figure \ref{F:11Q data B}). Furthermore, the peak is traveling to the right, in the positive $x$-direction. Until the time when the solution travels beyond our computational domain ($[-\alpha,\alpha]$), we observe that both the $L^\infty$ norm and the kinetic energy keeps increasing in time (blue solid line in the bottom row graphs of Figure \ref{F:11Q data B}). This gives an indication of possible blow-up, however, since the initial data is very close to the threshold, 
our current numerical simulations do not provide sufficient information in this case.

\begin{figure}[ht]
\includegraphics[width=0.44\textwidth]{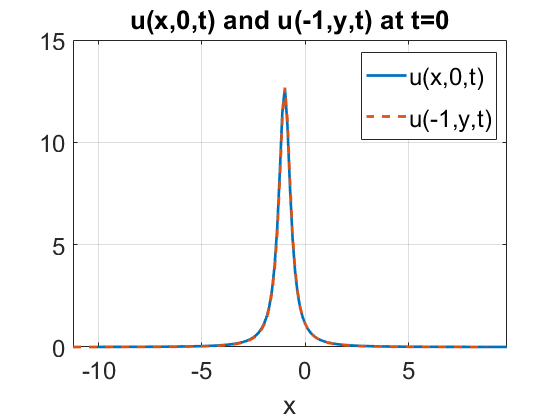}
\includegraphics[width=0.45\textwidth]{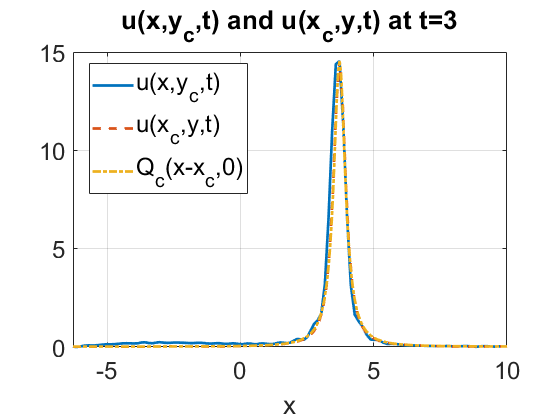}\\
\includegraphics[width=0.44\textwidth]{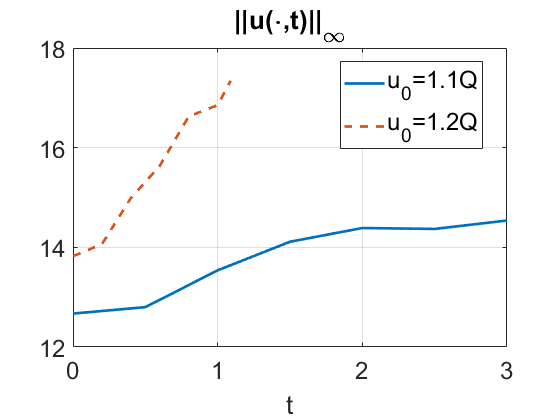}
\includegraphics[width=0.45\textwidth]{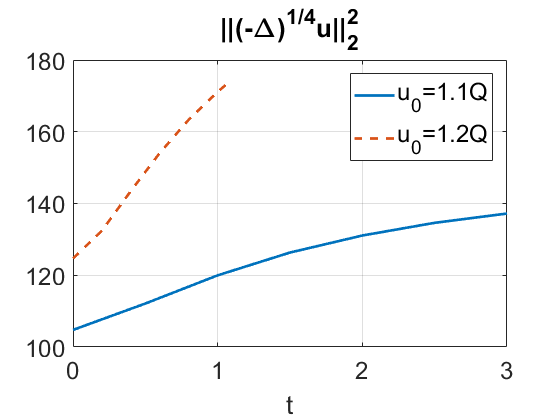}
\caption{Top: cross-sections of the solution $u(t)$ with $u_0=1.1\,Q(x+1,y)$ at $t=0$ and $t=3$ (and good matching with the rescaled ground state $Q_c$). Bottom: the norm growth in time for the time evolution with $u_0 = 1.1 Q$ in solid blue and with $u_0=1.2Q$ in dash red, both solutions indicate a blow-up behavior.}
\label{F:11Q data B}
\end{figure}

Therefore, we modify slightly the amplitude in the initial condition and  consider $u_0(x,y) = 1.2 \, Q(x+1,y)$, for which we track the norms $\|u(t)\|_{L^\infty(\mathbb R^2)}$ as well as $\|(-\Delta)^{1/4} u(t)\|_{L^2(\mathbb R^2)}$ for a comparison. In the bottom graphs of Figure \ref{F:11Q data B}, one can see that the time evolution in this case blows up almost immediately (around the time $t=1.1$). Since the quantities $\|u(t)\|_{L^\infty(\mathbb R^2)}$ and $\|(-\Delta)^{1/4} u(t)\|_{L^2(\mathbb R^2)}$ have a similar behavior in both cases of initial condition ($A=1.1$ and $1.2$), we can draw the conclusion that the solution with $u_0=AQ$ blows up for $A>1$. 

We take a step further in studying the blow-up behavior of solutions in this equation and look at the blow-up profiles. For example, an excellent matching of the cross-sections of the solution generated by $u_0=1.1\,Q$ at time $t=3$ can be observed on the right top plot of Figure \ref{F:11Q data B}. In the case of $A=1.2$, we obtain a similar matching. This indicates that a stable critical blow-up in the equation \eqref{E:HBO} follows a self-similar dynamics with the ground state profile. 
\medskip

We next test the non-radial data of the form
\begin{equation}\label{E:nonradial-blowup}
u_0(x,y)=\frac{A}{1+\big( (x+a)^2+(0.5y)^2 \big)^2},
\end{equation}
which has the norm $\|u_0\|_{L^2(\mathbb R^2)} = \frac1{\sqrt 2} \, A \, \pi$.
The threshold value for $A$ is $A_{th} \approx 2.9$. 

We study the data \eqref{E:nonradial-blowup} with $A>A_{th}$ (and the shift $a=2.5$ for convenience of graphing) and observe the blow-up behavior. For example, for $A=4.5$ the snapshots of the time evolution at times $t=0, 0.2, 0.5, 1.5, 2, 2.07$ are shown in Figure \ref{F:profile x4 B}. In the following Figure \ref{F:x4 data B} we provide the cross-sections of the solution at the beginning ($t=0$) and at the last computational time before the blow-up ($t=2.07$), as well as the growth of the norms in time. One can notice that starting with a non-radially symmetric initial data (top left graph shows the asymmetry in the cross-sections), the solution evolves into a radially symmetric solitary wave with a rescaled and shifted ground state profile,  indicating the behavior of a radially symmetric self-similar blow-up dynamics (in the core region).

\begin{figure}[ht]
\includegraphics[width=0.32\textwidth]{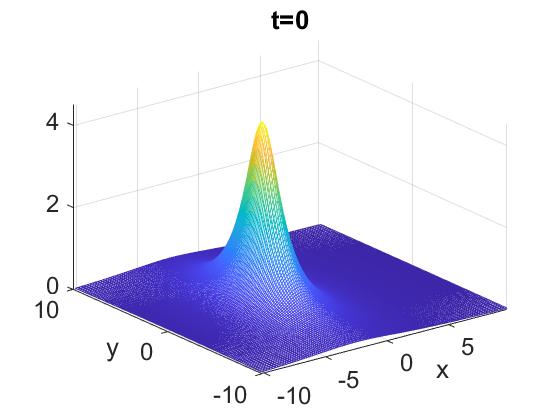}
\includegraphics[width=0.32\textwidth]{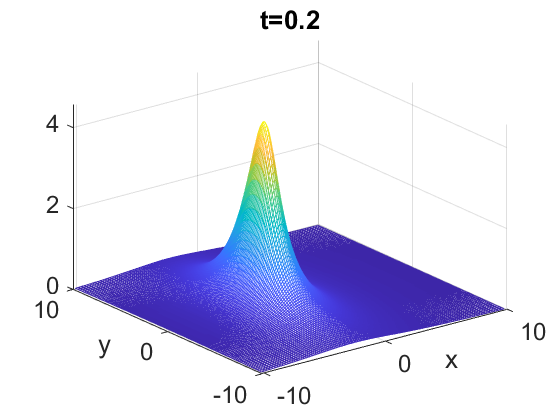}
\includegraphics[width=0.32\textwidth]{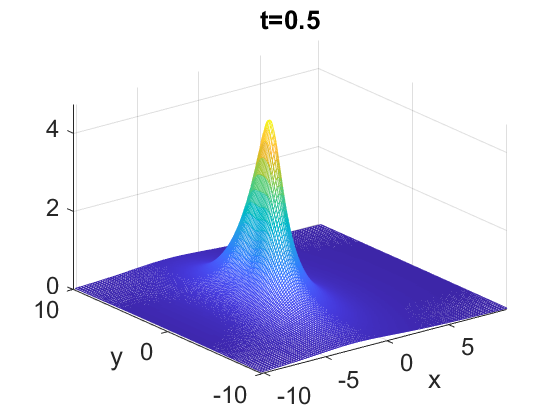}
\includegraphics[width=0.32\textwidth]{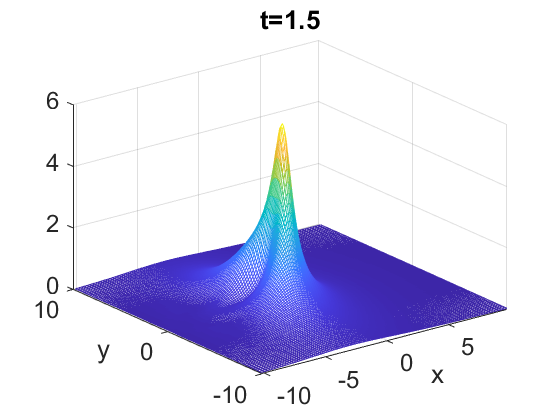}
\includegraphics[width=0.32\textwidth]{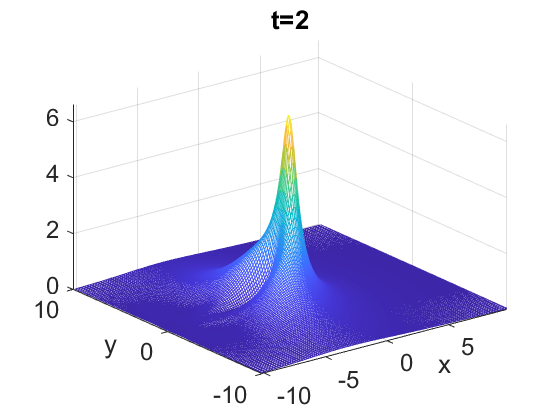}
\includegraphics[width=0.32\textwidth]{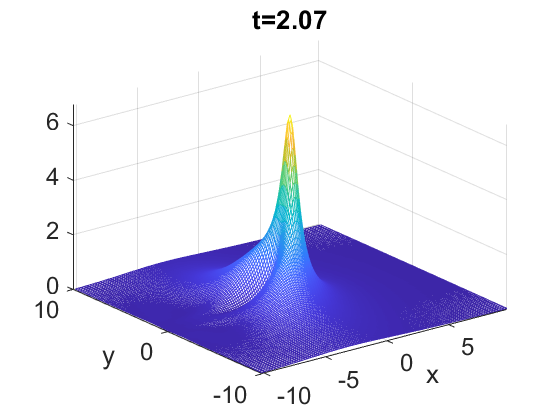}
\caption{Snapshots of the solution $u(t)$ with $u_0=\frac{4.5}{1+((x+2.5)^2+(0.5y)^2)^2}$.}
\label{F:profile x4 B}
\end{figure}

\begin{figure}[ht]
\includegraphics[width=0.40\textwidth]{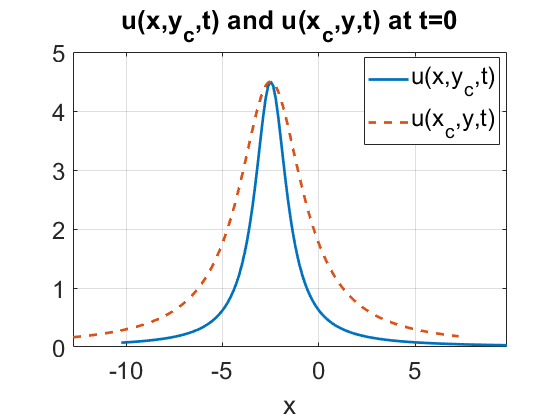}
\includegraphics[width=0.40\textwidth]{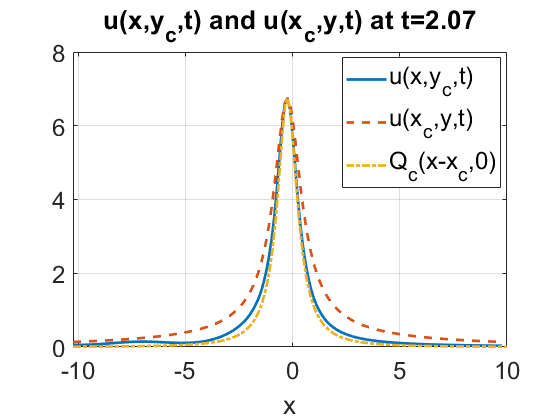}\\
\includegraphics[width=0.41\textwidth]{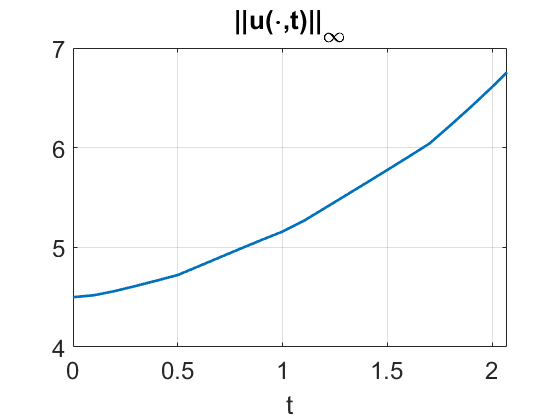}
\includegraphics[width=0.40\textwidth]{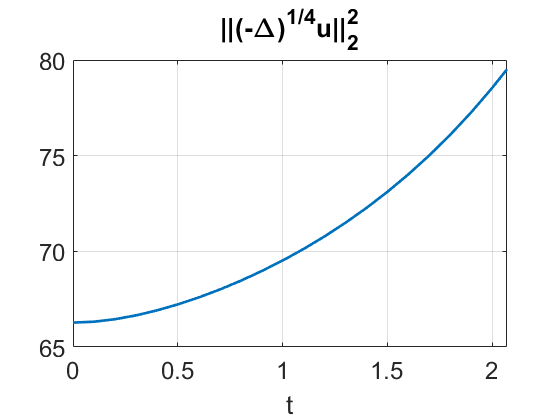}
\caption{The cross-sections of the solution $u(t)$ with $u_0=\frac{4.5}{1+((x+2.5)^2+(0.5y)^2)^2}$ at $t=0$ and $t=2.07$ (top); growth of the norms in time (bottom). 
}
\label{F:x4 data B}
\end{figure}

We also check the non-radial data with slower decay
\begin{equation}\label{E:nonradial-blowup2}
u_0(x,y)=\frac{A}{1+(x+a)^2+(0.5y)^2},
\end{equation}
which has the norm $\|u_0\|_{L^2(\mathbb R^2)} = \sqrt{2 \, \pi} \, A$, and hence, the threshold value $A_{th} \approx 2.6$. Taking $A>A_{th}$, we observe a similar blow-up behavior as shown in Figures \ref{F:profile x4 B}-\ref{F:x4 data B}. 
\medskip

Finally, we mention the gaussian type of data \eqref{ID:gaussian}, which has an exponential decay. Taking $A > A_{th} \approx 5$, we note that 
for example, $A=5.5$ in $u_0 = A\, e^{-(x^2+y^2)}$ and $E[u_0] = 0.87 >0$ produces a blow-up in finite time, similarly, 
$A=6$ in the same $u_0$, which gives $E[u_0] = -2.11 < 0$, also produces a blow-up solution. Thus, it is possible to have solutions with the positive and negative energy that blow-up in finite time (see remarks after Conjecture \ref{C:critical}).  
\smallskip

We conclude that the part 2 of Conjecture \ref{C:critical} holds for all data that we considered. Furthermore, a stable blow-up shows a self-similar dynamics with the ground state profile.  

\newpage

\subsection{Interaction of solitary waves}\label{S:interaction}
We next investigate the interaction of two solitary waves, for that we take two 
rescaled solitary waves $Q_{c_1}$ and $Q_{c_2}$ as defined in \eqref{E:Qc}, and track their evolution and interaction. 

We first consider the two solitary-waves that are separated along the $x$-axis, i.e.,
\begin{equation}\label{E:2Q-1}
u_0(x,y)=a_1 \, Q_{c_1}(x+x_1,y)+a_2\, Q_{c_2}(x+x_2,y).
\end{equation}

\begin{figure}[ht]
\includegraphics[width=0.32\textwidth]{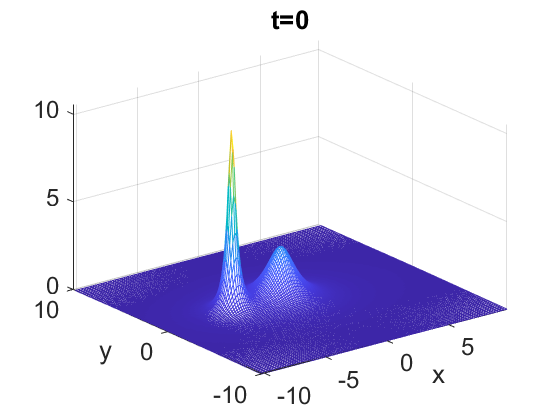}
\includegraphics[width=0.32\textwidth]{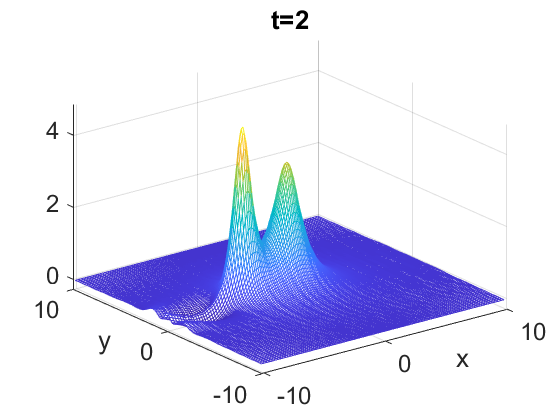}
\includegraphics[width=0.32\textwidth]{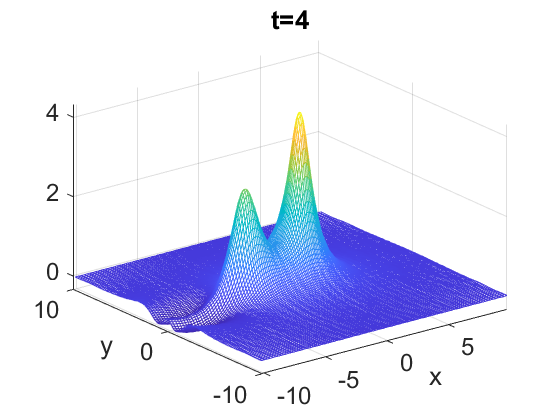}
\includegraphics[width=0.32\textwidth]{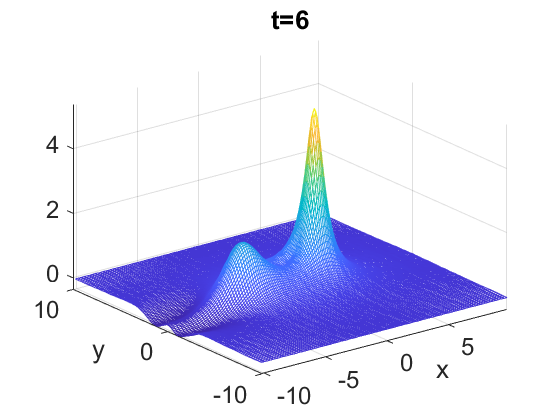}
\includegraphics[width=0.32\textwidth]{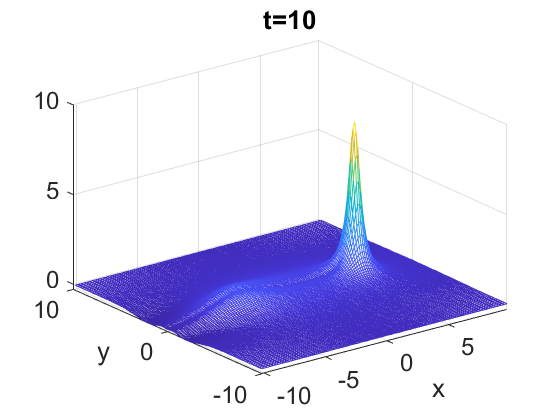}
\includegraphics[width=0.32\textwidth]{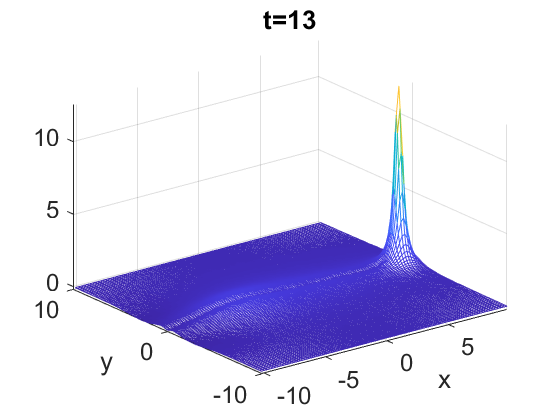}
\caption{Snapshots of interaction for $u_0=0.9 \, Q(x+5,y) + Q_{0.25}(x+1,y)$.}
\label{F:profile 2Q B}
\end{figure}

In Figure \ref{F:profile 2Q B}, we show the snapshots of the time evolution for the initial condition 
$$
u_0=0.9\, Q(x+5,y)+Q_{0.25}(x+1,y).
$$ 
While the mass for each bump is smaller than our predicted threshold $\|Q\|_{L^{2}(\mathbb{R}^2)}$, the total mass is greater than $\|Q\|_{L^{2}(\mathbb{R}^2)}$ (here, $\|u_0\|_{L^2(\mathbb R^2)}^2 \approx 83.06$
). In Figure \ref{F:profile 2Q B}, we can see the higher solitary wave travels faster than the lower one, and they interact, merging together, between the time $2<t<4$. After $t>4$, they split. In the process of interaction, the initial higher bump obtains sufficient amount of mass, and continue traveling in the positive $x$-direction, it blows up in finite time, while the smaller bump loses the mass and completely radiates to the left. 

\begin{figure}[ht]
\includegraphics[width=0.38\textwidth]{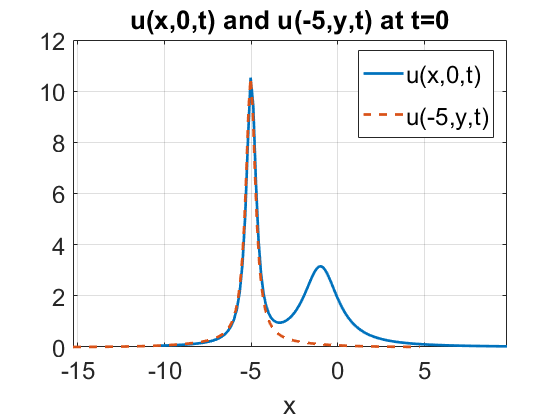}
\includegraphics[width=0.38\textwidth]{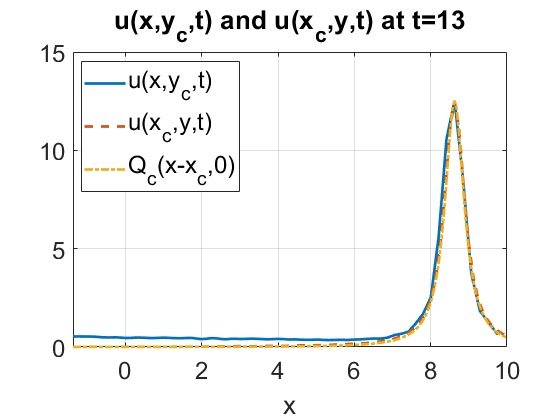}\\
\includegraphics[width=0.38\textwidth]{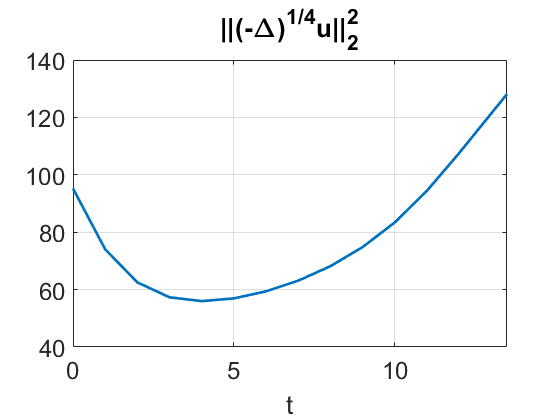}
\includegraphics[width=0.38\textwidth]{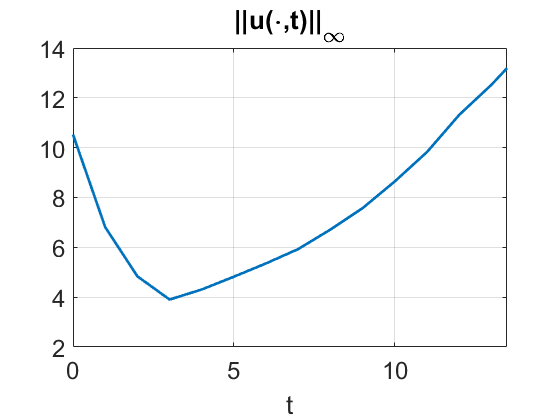}
\caption{Top: cross-sections before and after the interaction for $u_0=0.9Q(x+5,y)+Q_{0.25}(x+1,y)$. As the higher bump obtains sufficient mass after the interaction, it blows up with the rescaled ground state profile $Q_c$. Bottom: time dependence of the kinetic energy and the $L^\infty$ norm.}
\label{F:2Q B data}
\end{figure}

If the two solitary waves are not sufficiently large, 
they will interact similarly: the higher, and thus, faster one will merge into the slower one, and then split from each other, after the interaction, with the faster one obtaining some additional mass from the slower one, but eventually both will disperse into the radiation, see Figure \ref{F:profile 2Q S} for 
$$
u_0=0.7\, Q_{0.5}(x+4,y)+0.5 \, Q_{0.25}(x,y)).
$$
In this case the total mass $\|u_0\|_{L^2(\mathbb R^2)}^2 \approx 41.5605.$

\begin{figure}[ht]
\includegraphics[width=0.32\textwidth]{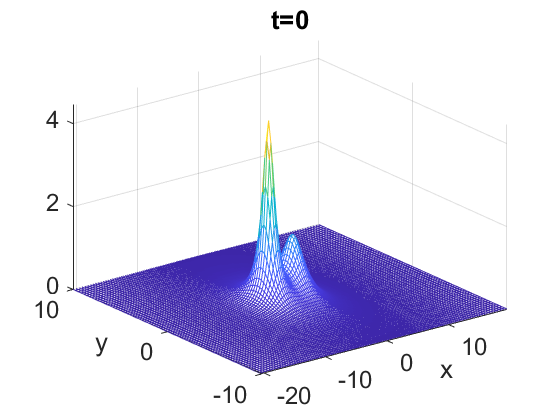}
\includegraphics[width=0.32\textwidth]{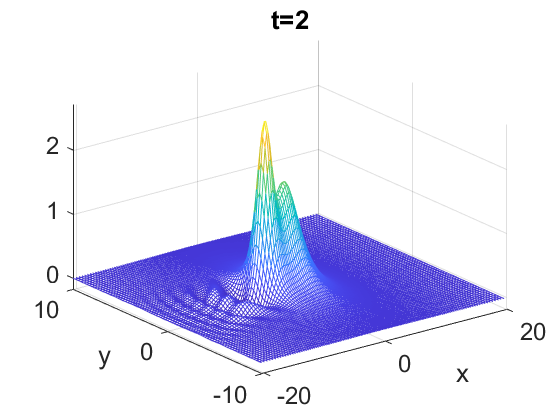}
\includegraphics[width=0.32\textwidth]{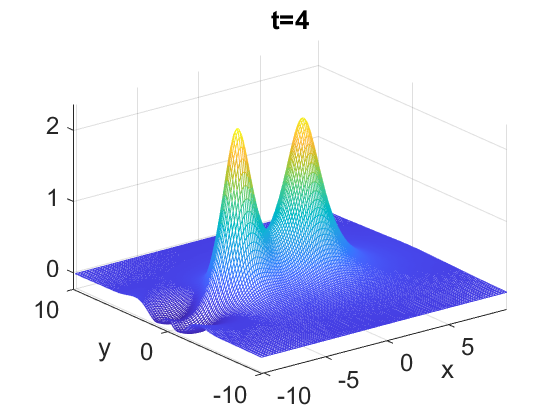}
\includegraphics[width=0.32\textwidth]{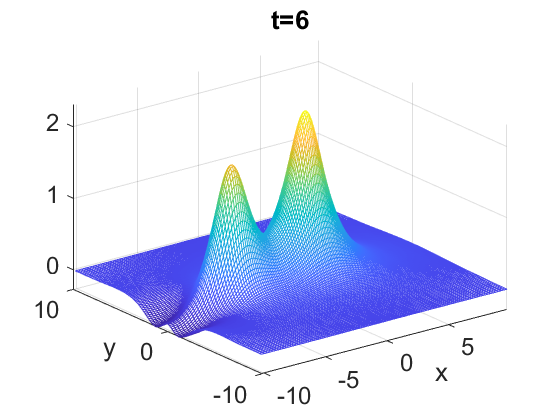}
\includegraphics[width=0.32\textwidth]{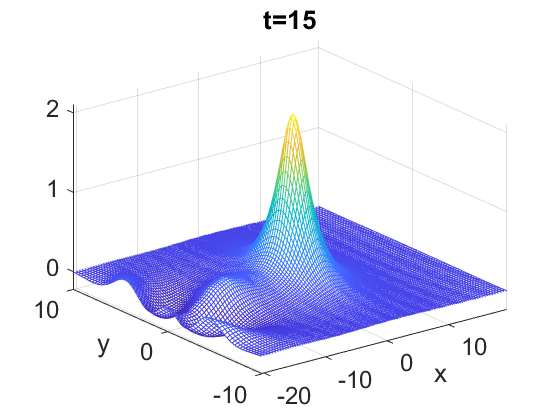}
\includegraphics[width=0.32\textwidth]{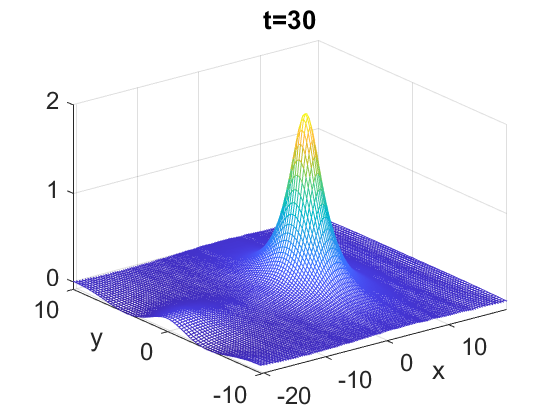}
\caption{Snapshots of interaction for $u_0=0.7\,Q_{0.5}(x+4,y)+0.5\,Q_{0.25}(x,y))$.}
\label{F:profile 2Q S}
\end{figure}

\begin{figure}[ht]
\includegraphics[width=0.38\textwidth]{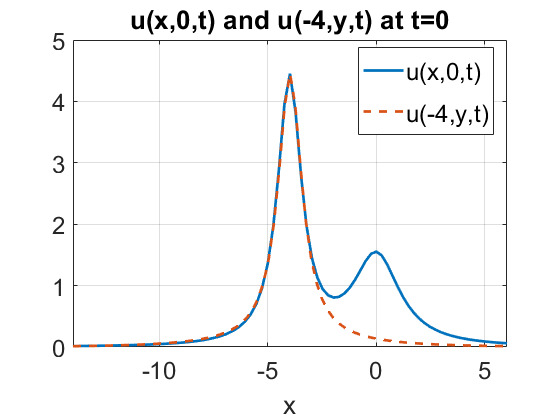}
\includegraphics[width=0.38\textwidth]{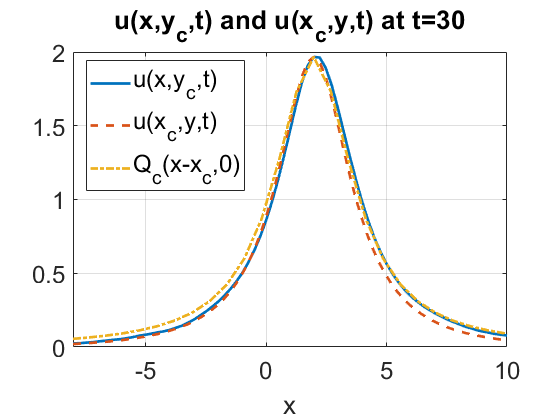}\\
\includegraphics[width=0.38\textwidth]{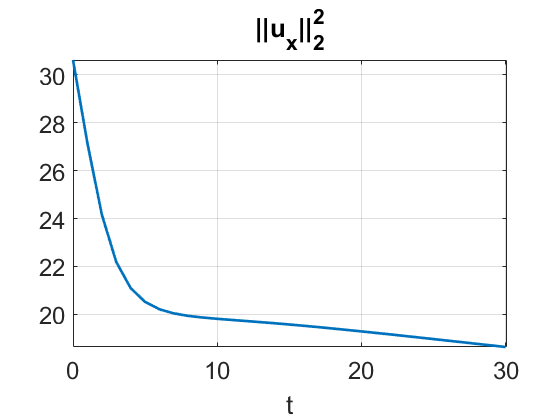}
\includegraphics[width=0.38\textwidth]{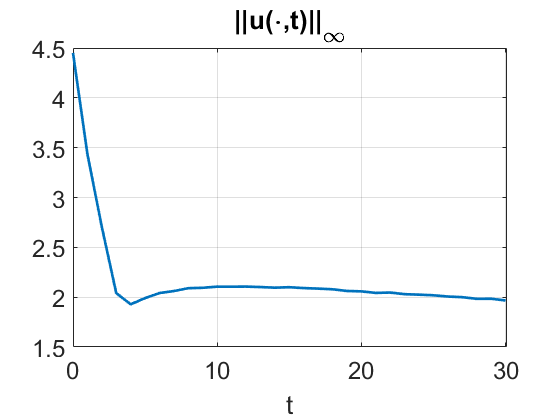}
\caption{Top: cross-sections before and after interaction for 
$u_0=0.7\,Q_{0.5}(x+4,y)+0.5\,Q_{0.25}(x,y))$. The higher bump obtains some extra mass after the interaction, however, it is not sufficient to develop a blow-up, and thus, both bumps eventually radiate. While decreasing in its heigh, the solution maintains radial symmetry and is close to the rescaled profile $Q_c$.}
\label{F:2Q S data}
\end{figure}

\medskip

We next modify the initial data and separate the two solitary waves in the $y$-coordinate. For example, consider 
\begin{equation}\label{E:inter1}
u_0=0.9\, Q(x,y-5) + 0.9 \,Q(x,y+5),
\end{equation}
so there is about 10 units of separation in $y$. One can see the two bumps moving parallel along the $x$-direction without much of an interaction, and eventually, radiate away, see Figure \ref{F:profile 2Qy S} as if they would just exist on their own. The energy in this case is $E[u_0] \approx 2.69 > 0$.

\begin{figure}[ht]
\includegraphics[width=0.32\textwidth]{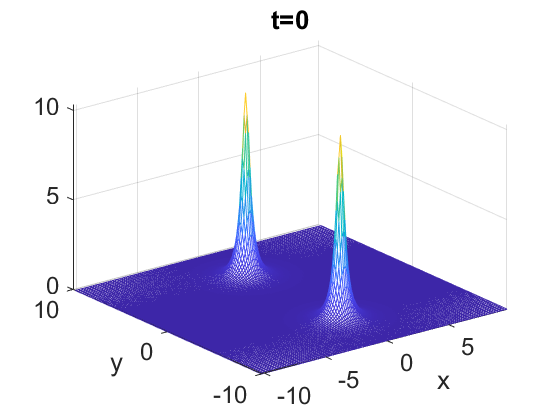}
\includegraphics[width=0.32\textwidth]{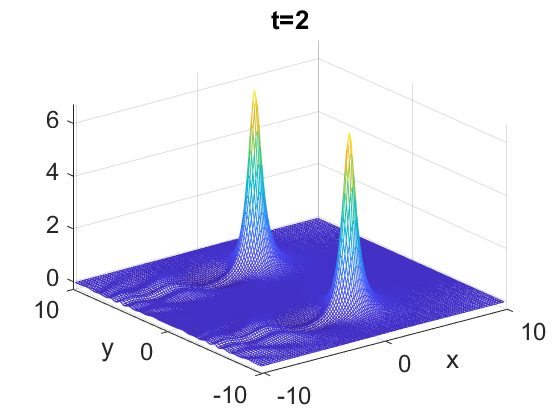}
\includegraphics[width=0.32\textwidth]{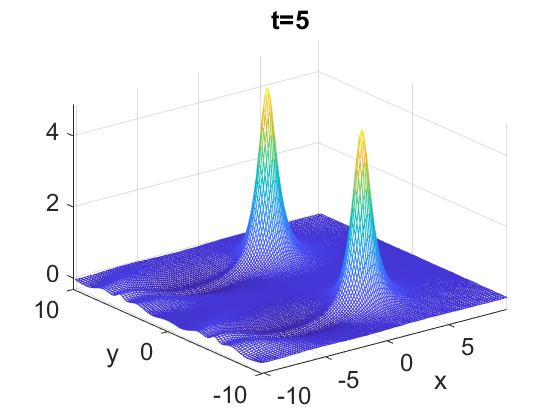}
\includegraphics[width=0.32\textwidth]{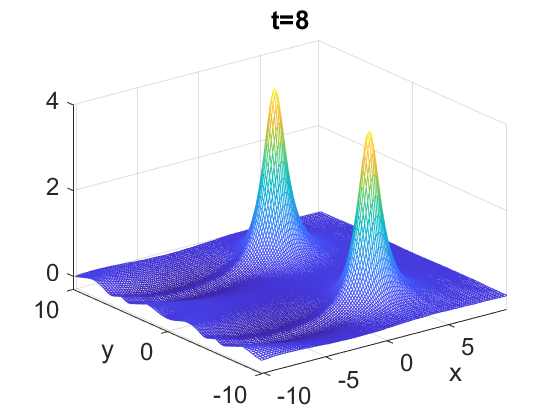}
\includegraphics[width=0.32\textwidth]{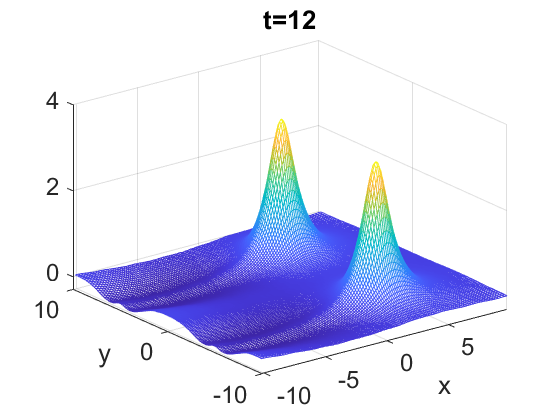}
\includegraphics[width=0.32\textwidth]{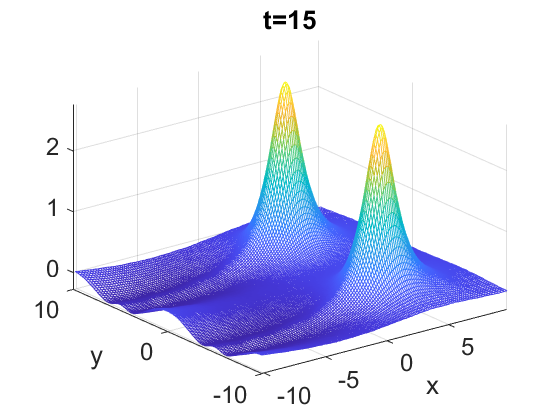}
\caption{Snapshots of time evolution for $u_0=0.9\, Q(x,y-5) + 0.9 \,Q(x,y+5))$.}
\label{F:profile 2Qy S}
\end{figure}

In our final example, we consider the same two solitary waves as before, but now they are separated in the $y$-coordinate not as much, so the two bumps are sufficiently close to each other. The initial condition is 
\begin{equation}\label{E:inter2}
u_0=0.9 \, Q(x+5,y-1) + 0.9 \,Q(x+5,y+1)),
\end{equation} 
there is only 2 units of separation in $y$, for a depiction see Figure \ref{F:profile 2Qy B}. 

\begin{figure}[ht]
\includegraphics[width=0.32\textwidth]{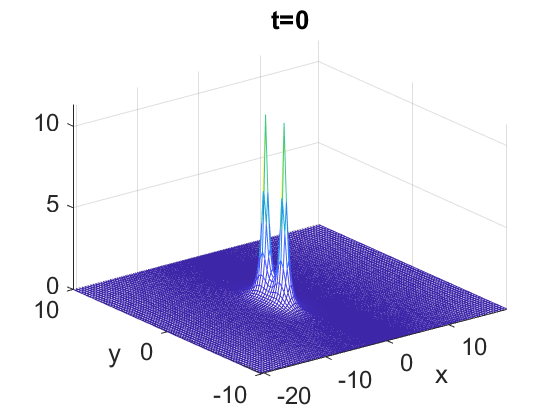}
\includegraphics[width=0.32\textwidth]{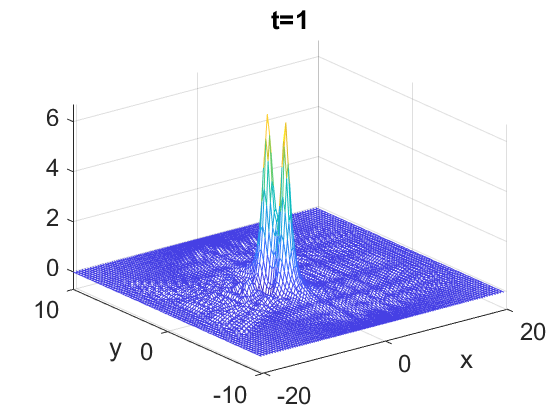}
\includegraphics[width=0.32\textwidth]{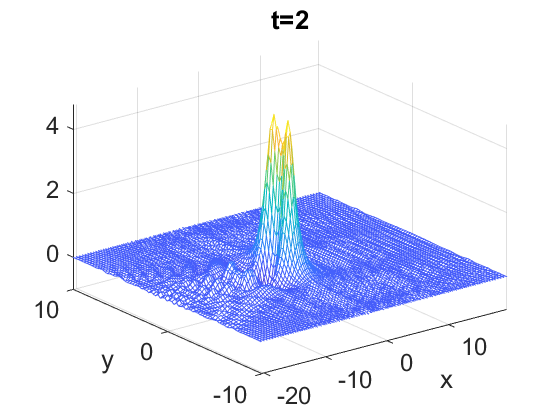}
\includegraphics[width=0.32\textwidth]{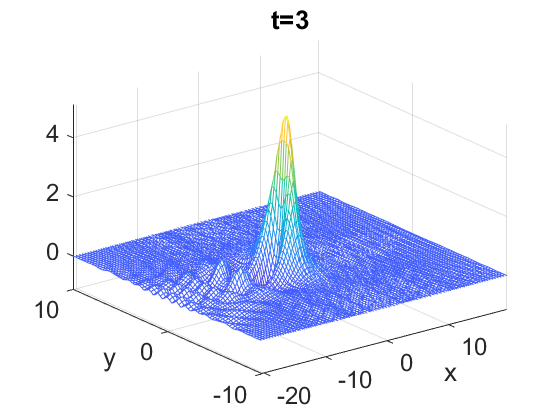}
\includegraphics[width=0.32\textwidth]{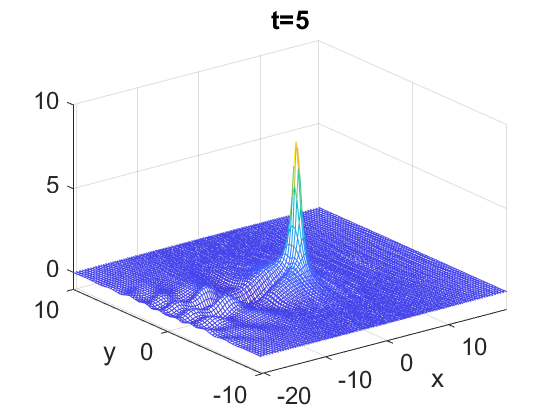}
\includegraphics[width=0.32\textwidth]{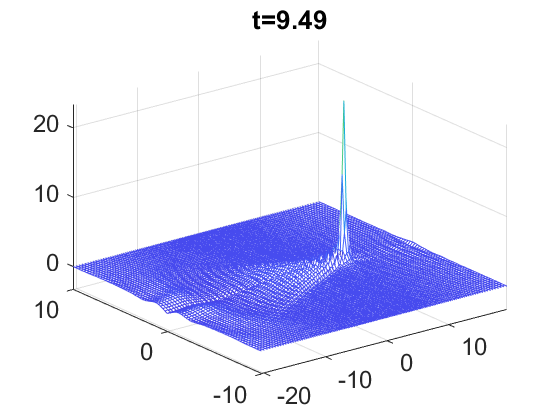}
\caption{Snapshots of strong interaction for $u_0=0.9\, Q(x+5,y-1)+ 0.9 \,Q(x+5,y+1)$.}
\label{F:profile 2Qy B}
\end{figure}

\begin{figure}[ht]
\includegraphics[width=0.32\textwidth]{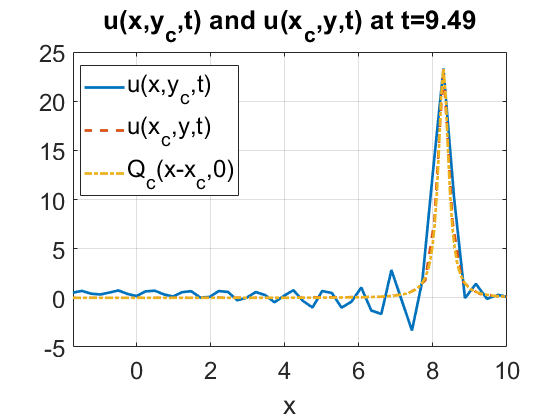}
\includegraphics[width=0.32\textwidth]{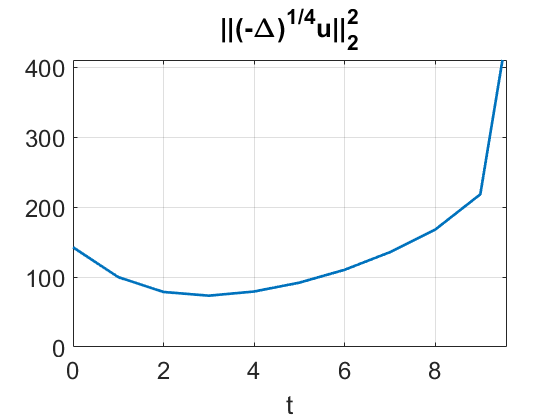}
\includegraphics[width=0.32\textwidth]{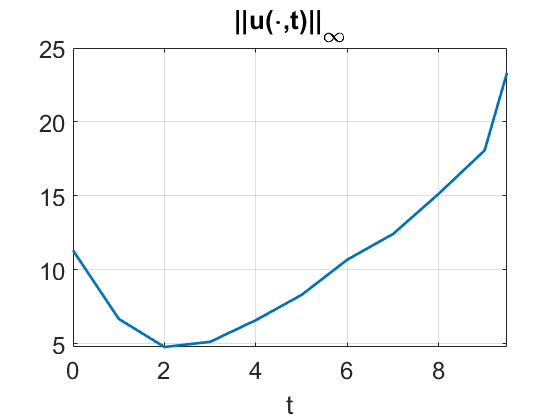}
\caption{Details on the interaction for $u_0=0.9(Q(x+5,y-1)+Q(x+5,y+1))$. Left: cross-sections after the interaction at the final computational time and matching with the rescaled $Q_c$. Middle and right: time dependence of the kinetic energy and the $L^\infty$ norm.}
\label{F:2Qy B data}
\end{figure}

The interaction happens as the two solitary waves merge into one. A joint lump will have sufficient mass and will generate a finite time blow-up solution. One can see that as merging together occurs, the radiation wedge is being generated, which clearly continues after the two bumps merged into one. Since the merged solution has large enough mass, it will blow up in finite time (see the height of snapshots in the bottom row of Figure \ref{F:profile 2Qy B}). 
We show the cross-sections at the final (computational) time $t=9.49$ as well as the matching with the rescaled $Q_c$ in the left graph of Figure \ref{F:2Qy B data}. 
The dependence on time of the kinetic energy and the $L^\infty$ norm is shown in the middle and right graphs of the same figure. One can notice that the height during the merging drops significantly (around time $t=2$), but due to the sufficient mass, the solution picks up the growth of its height and its kinetic energy. We note that the energy in the last two examples \eqref{E:inter1} and \eqref{E:inter2} is the same and positive, i.e., $E[u] \approx 2.69 > 0$. 


\newpage

\bibliographystyle{abbrv}
\bibliography{ref_HBO}

\end{document}